\documentclass[12pt,a4paper,reqno,draft]{amsart}
\usepackage[T1]{fontenc}
\usepackage{amsmath,amssymb,ifthen}
\usepackage{fullpage,comment}
\usepackage{graphicx,psfrag,subfigure}
\usepackage[dvips]{color}
\def\A{\mathbb A}

\def\R{{\mathbb R}}
\def\N{{\mathbb N}}

\def\EE{{\mathcal E}}
\def\BB{{\mathcal B}}

\def\MM{{\mathcal M}}
\def\NN{{\mathcal N}}
\def\OO{{\mathcal O}}
\def\PP{{\mathcal P}}
\def\RR{{\mathcal R}}
\def\SS{{\mathcal S}}

\def\TT{{\mathcal T}}
\def\JJ{{\mathcal J}}

\def\XX{{\mathcal X}}
\def\YY{{\mathcal Y}}

\def\uint{u^{{\rm int}}}
\def\vint{v^{{\rm int}}}

\def\bw{{\boldsymbol{w}}}
\def\bB{{\boldsymbol{B}}}

\def\norm#1#2{\|#1\|_{#2}}

\def\set#1#2{\big\{#1\,:\,#2\big\}}

\def\eps{\varepsilon}

\def\v{\mathbf{v}}

\newcommand{\normHmeh}[3][]{#1\|#2#1\|_{H^{-1/2}(#3)}}
\newcommand{\normLtwo}[3][]{#1\|#2#1\|_{L^2(#3)}}
\newcommand{\normHeh}[3][]{#1\|#2#1\|_{H^{1/2}(#3)}}

\def\normL2#1#2{\|#1\|_{L^2(#2)}}

\newcommand{\dual}[3][]{#1\langle#2\,,\,#3#1\rangle}

\newcounter{constantsnumber}
\def\namec#1#2{%
 \ifthenelse{\equal{#1}{rel}}{C_{\rm rel}}{%
  \ifthenelse{\equal{#1}{mesh}}{C_{\rm mesh}}{%
  \ifthenelse{\equal{#1}{sz}}{C_{\rm sz}}{%
  \ifthenelse{\equal{#1}{dislocrel}}{C_{\rm dlr}}{%
  \ifthenelse{\equal{#1}{eff}}{C_{\rm eff}}{%
  \ifthenelse{\equal{#1}{main}}{C_{\rm V}}{%
  \ifthenelse{\equal{#1}{opt}}{C_{\rm opt}}{%
  \ifthenelse{\equal{#1}{normequiv}}{C_{\rm norm}}{%
  \ifthenelse{\equal{#1}{reliable}}{C_{\rm rel}}{%
  \ifthenelse{\equal{#1}{efficient}}{C_{\rm eff}}{%
  \ifthenelse{\equal{#1}{dlr}}{C_{\rm dlr}}{%
  \ifthenelse{\equal{#1}{stable}}{C_{\rm stab}}{%
  \ifthenelse{\equal{#1}{reduction}}{C_{\rm red}}{%
   \ifthenelse{\equal{#1}{unibound}}{C_{\rm hot}}{%
    \ifthenelse{\equal{#1}{hotConst}}{C_{\rm hot}}{%
   \ifthenelse{\equal{#1}{inverseK}}{C_{\rm K}}{%
  \ifthenelse{\equal{#1}{refined}}{C_{\rm ref}}{%
  \ifthenelse{\equal{#1}{estconv}}{C_{\rm est}}{%
  \ifthenelse{\equal{#1}{optimal}}{C_{\rm opt}}{%
  \ifthenelse{\equal{#1}{qo}}{C_{\rm qo}}{%
  \ifthenelse{\equal{#1}{mon}}{C_{\rm mon}}{%
  \ifthenelse{\equal{#1}{cea}}{C_{\mbox{\scriptsize C\'ea}}}{%
  \ifthenelse{\equal{#2}{newcounter}}{\refstepcounter{constantsnumber}\label{const#1}}{}C_{\ref{const#1}}}%
}}}}}}}}}}}}}}}}}}}}}}

\newcounter{contractionnumber}
\def\nameq#1#2{%
  \ifthenelse{\equal{#1}{reduction}}{q_{\rm red}}{%
  \ifthenelse{\equal{#1}{estconv}}{q_{\rm est}}{%
  \ifthenelse{\equal{#1}{cea}}{q_{\mbox{\scriptsize C\'ea}}}{%
  \ifthenelse{\equal{#2}{newcounter}}{\refstepcounter{contractionnumber}\label{contraction#1}}{}q_{\ref{contraction#1}}}%
}}}

\def\namer#1#2{%
  \ifthenelse{\equal{#1}{reduction}}{\rho_{\rm red}}{%
  \ifthenelse{\equal{#1}{estconv}}{\rho_{\rm est}}{%
  \ifthenelse{\equal{#1}{cea}}{\rho_{\mbox{\scriptsize C\'ea}}}{%
  \ifthenelse{\equal{#1}{qo}}{\rho_{\mbox{\scriptsize qo}}}{%
  \ifthenelse{\equal{#2}{newcounter}}{\refstepcounter{contractionnumber}\label{contraction#1}}{}\rho_{\ref{contraction#1}}}%
}}}}

\def\dist{{\rm dist}}

\newtheorem{theorem}{Theorem}

\newtheorem{lemma}[theorem]{Lemma}

\newtheorem{algorithm}[theorem]{Algorithm}
\newtheorem{definition}[theorem]{Definition}
\newenvironment{remark}{\medskip\noindent\textbf{Remark.}\ \it}{\qed\smallskip}

\def\T{\mathbb T}

\def\v{\boldsymbol{v}}
\def\w{\boldsymbol{w}}
\def\B{\boldsymbol{B}}
\def\A{\boldsymbol{A}}

\numberwithin{equation}{section}
\numberwithin{theorem}{section}
\begin{document}
\title{Optimal adaptivity for\\non-symmetric FEM/BEM coupling}
\author{Michael Feischl}
\address{School of Mathematics and Statistics,
         The University of New South Wales,
         Sydney 2052, Australia}
\email{m.feischl@unsw.edu.au}
\thanks{Supported by the Australian Research Council (ARC) under
grant number DE170100222 and by the Austrian Research Fund (FWF) under grant number P27005.
}
\begin{abstract}We develop a framework which allows us to prove the essential general quasi-orthogonality for the non-symmetric Johnson-N\'ed\'elec finite element/boundary element coupling.
	General quasi-orthogonality was first proposed in~\cite{axioms} as a necessary ingredient of optimality proofs and is the major difficulty on the way to prove
	rate optimal convergence of adaptive algorithms for many strongly non-symmetric problems. The proof exploits a new connection between the general quasi-orthogonality and
	$LU$-factorization of infinite matrices. 
	We then derive that a standard adaptive algorithm for the Johnson-N\'ed\'elec coupling converges with optimal rates.
	The developed techniques are fairly general and can most likely be applied to other problems like Stokes equation.
	\end{abstract}
\maketitle 

\section{Introduction}
The theory of rate optimal adaptive algorithms for finite element methods originated in the seminal paper~\cite{stevenson07} by Stevenson and was further improved in~\cite{ckns}
by Cascon, Kreuzer, Nochetto, and Siebert.
These papers prove essentially, that a standard adaptive algorithm of the form
\begin{align*}
 \fbox{Solve}\longrightarrow\fbox{Estimate}\longrightarrow\fbox{Mark}\longrightarrow\fbox{Refine}
\end{align*}
generates asymptotically optimal meshes for the approximation of the solution of a Poisson problem.
The new ideas sparked a multitude of papers applying and extending the techniques to different problems, see e.g.,~\cite{ks,cn} for conforming methods,~\cite{rabus10,BeMao10,bms09,cpr13,mzs10}  
for nonconforming methods,~\cite{LCMHJX,CR2012,HuangXu} for mixed formulations, and~\cite{fkmp,gantumur,affkp,ffkmp:part1,ffkmp:part2} for boundary element methods (the list is not exhausted, see also~\cite{axioms} and
the references therein).
All the mentioned results, however, focus on symmetric problems in the sense that the underlying equation induces a symmetric operator. 
The first proof of rate optimality for a non-symmetric problem which does not rely on additional assumptions is given in~\cite{nonsymm} for a general second order elliptic operator with non-vanishing diffusion
coefficient of the form
\begin{align*}
-{\rm div}(A\nabla u) + b\cdot \nabla u + cu=f.
\end{align*}
This approach, however, relies heavily on the fact that the non-symmetric part of the operator $ (b\cdot \nabla u + cu)$ is only a compact perturbation (one differentiation instead of two for the diffusion part). The present work aims to shed some light on the completely unexplored 
world of rate optimality for strongly non-symmetric problems (meaning that the non-symmetric part of the operator is not substantially \emph{smaller} in any sense). Although the work is focused on the particular model problem of Johnson-N\'ed\'elec FEM/BEM coupling, we believe that the developed techniques
will be very useful for many other non-symmetric problems of the form $Au=F$ for a non-symmetric operator $A$ and a right-hand side $F$. 

For the particular case of FEM/BEM coupling, only convergence of
the adaptive algorithm is known. This was first proven rigorously in~\cite{fembem} for the standard residual based error estimator which was first derived in~\cite{afp}.

Using the abstract framework for rate optimality developed in~\cite{axioms}, we observe that the major obstacle is the \emph{general quasi-orthogonality} property introduced in~\cite{axioms}.
The property is a generalization of the usual orthogonality property
\begin{align}\label{eq:orth}
 \norm{u-u_{\ell+1}}{}^2+\norm{u_{\ell+1}-u_\ell}{}^2 =\norm{u-u_\ell}{}^2,
\end{align}
for increasingly accurate nested Galerkin approximations $u_\ell,u_{\ell+1}$ of the exact solution $u$. The orthogonality~\eqref{eq:orth} follows immediately from the well-known Galerkin orthogonality, if the underlying problem
induces a symmetric bilinear form and thus a Hilbert (energy-) norm $\norm{\cdot}{}$. If the problem is non-symmetric, however,~\eqref{eq:orth} fails to hold (even in approximate forms usually called quasi-orthogonality)
and this breaks all existing optimality proofs. In~\cite{axioms}, we prove that \emph{general quasi-orthogonality} is the weakest possible orthogonality condition in the sense that it is necessary to prove optimality.
While rigorously stated in Section~\ref{sec:axioms} below, \emph{general quasi-orthogonality} roughly implies that the approximation error has a decomposition of the form
\begin{align*}
 \norm{u-u_\ell}{}^2\simeq \norm{u_{\ell+1}-u_\ell}{}^2+\norm{u_{\ell+1}-u_{\ell+2}}{}^2+\ldots\quad\text{for all }\ell\in\N
\end{align*}
for nested Galerkin approximations $u_\ell,u_{\ell+1},u_{\ell+2},\ldots$ of $u$. While this property seems hard to prove by itself, we discover an interesting connection to the $LU$-factorization of infinite matrices in this work. 

This connection can be formulated as follows: Assume that there exists a Riesz basis $B$ of the underlying Hilbert space such that the problem $Au=F$ can be equivalently stated as a matrix equation
\begin{align*}
 Mx=G\quad\text{with}\quad M\in \R^{\N\times\N},\,G\in\R^\N,
\end{align*}
where $M_{vw}=\dual{Aw}{v}$ for all $v,w\in B$, $G_v:=\dual{F}{v}$, and $u=\sum_{v\in B}x_v v$. If the matrix $M$ has an $LU$-factorization $M=LU$ for lower/upper-triangular infinite matrices $L,U\in\R^{\N\times\N}$ such that
$L,U,L^{-1},U^{-1}\colon \ell_2\to\ell_2$ are bounded operators, then \emph{general quasi-orthogonality} holds. 

We exploit the stated connection by constructing a suitable Riesz basis for the particular FEM/BEM coupling and then proving that a bounded $LU$ factorization exists. The construction of the Riesz basis 
is quite challenging, since for the FEM/BEM coupling, we need compatible basis functions in $H^1$ and $H^{-1/2}$. This requires an extension
of the well-known Scott-Zhang projection to make operators on different level commute with each other. To prove that the resulting matrix $M$ has the desired $LU$-factorization,
we rely on techniques from wavelet methods, which prove that $M$ is exponentially decaying (in the sense of Jaffard class matrices). While being probably an artifact of the proof, we are forced
to introduce a grading condition on the adaptively generated meshes, as was also done in~\cite{l2opt}.

This opens the door to prove rate optimality of the adaptive algorithm. It turns out that all the other
requirements for rate optimality (formulated in~\cite{axioms}) can be shown by combining arguments from FEM and BEM (for both methods, rate optimality has been proved already).

The strategy of the proof is very general, but the details a tailored to the present FEM/BEM coupling problem. We are confident that similar techniques can be used to prove optimality of adaptive algorithms for the
Stokes equation and many other related non-symmetric problems. The author would like to note that, to the best of his knowledge, the theory of $LU$-factorization of infinite matrices currently cannot
answer very interesting (and for this work very useful) questions like: \emph{Which positive definite matrices have a bounded factorization?} This is the reason why this paper is quite technical
despite the simple underlying idea. It is possible that advances in this direction could improve (e.g., by removing the grading condition) and simplify the present result.

\medskip

As an interesting side result, the Riesz basis constructed in Section~\ref{section:basis} below can, in principle, be computed and used for actual implementations. This brings the benefit of
uniformly bounded condition numbers of the involved matrices without preconditioning.

\subsection{Outline of the paper}
The main result is given in Theorem~\ref{thm:opt} in Section~\ref{section:apps}. 
In Section~\ref{section:exp}, we introduce
the class of Jaffard matrices, and show that they admit a bounded $LU$-factorization under certain conditions.
In Section~\ref{section:lu}, we show that general quasi-orthogonality is equivalent to the fact that a certain (infinite)
system matrix of the problem at hand has a bounded $LU$-factorization. This observation is the key element of the paper. 
The remainder of the work is devoted to building a system matrix for the Johnson-N\'ed\'elec coupling, which fits into this framework
of Jaffard class matrices. Therefore, we construct a local wavelet basis in Section~\ref{section:basis}. To that end,
we use a new quasi-interpolation operator from Section~\ref{section:szb} which is based on the classical Scott-Zhang projection.
In Section~\ref{section:metric} we construct certain metrics which characterize the exponential decay of the system matrix.
Finally, Section~\ref{section:disc} constructs the system matrix.

\subsection{Notation}
We use $\#A$ to denote the cardinality of a set $A$. Moreover, $I$ denotes the identity matrix $I\in\R^{n\times n}$, $I_{ij}=0$ for all $i\neq j$ and $I_{ij}=1$ for all $i=j$. The dimension $n\in\N$
is only specified when not clear from the context. The standard space of squared summable sequences is denoted by $\ell_2=\ell_2(\N)$. We denote the $\ell_2$-norm by $\norm{\cdot}{\ell_2}$, whereas the operator norm for operators on $\ell_2$ is denoted by $\norm{\cdot}{2}$.
Operators on $\ell_2$ are often identified with infinite matrices $M\in\R^{\N\times\N}$ and we use the norms $\norm{M}{1}:=\sup_{i\in\N}\sum_{j\in\N}|M_{ij}|$ and 
$\norm{M}{\infty}:=\sup_{j\in\N}\sum_{i\in\N}|M_{ij}|$.
\section{General assumptions}
\subsection{Preliminaries}
In the following, $\Omega\subseteq \R^2$ is a polygonal domain with boundary $\Gamma:=\partial\Omega$.
Given a Lipschitz domain $\omega\subseteq \R^2$, we denote by $H^s(\omega)$ the usual Sobolev spaces for $s\geq 0$. For non-integer values of $s$, we use real interpolation to define $H^s(\omega)$. 
Their dual spaces $\widetilde H^{-s}(\omega)$ are defined by extending the $L^2$-scalar product.
Given $\gamma\subseteq \partial\omega$, we define $H^s(\gamma)$ as the trace space of $H^{s+1/2}(\omega)$ for all $s>0$. Again, the dual space $\widetilde H^{-s}(\gamma)$ is defined via the extended 
$L^2$-scalar product.

\begin{remark}
There is no reason for the author to believe that the methods developed in this work are restricted to the 2D case. However, the technical difficulties are already substantial for $d=2$ and thus we
decided to restrict to this case for clarity of presentation. 
\end{remark}

\subsection{Variational form}
The main goal of this paper is to prove optimality of FEM/BEM coupling. However, most of the methods work in a much broader context. Therefore, we start with an abstract variational problem and
go back to the concrete application in Section~\ref{section:apps}. To that end, suppose $\XX$ is a separable Hilbert space. Moreover, suppose that $(\XX_\ell)_{\ell\in\N}$ is a nested sequence of subspaces, i.e.,
\begin{align*}
 \XX_\ell\subseteq \XX_k\subseteq \XX\quad\text{for all }\ell\leq k\in \N.
\end{align*}
Assume that $a(\cdot,\cdot)\colon \XX\times \XX\to \R $ is a bounded bilinear form, which is additionally elliptic, i.e., 
\begin{align}\label{eq:aelliptic}
\inf_{x\in\XX}\frac{a(x,x)}{\norm{x}{\XX}^2}=c_0>0.
\end{align}
For $f\in \XX^*$, define $u\in\XX$ and $u_\ell\in \XX_\ell$ for all $\ell\in\N$ as the unique solutions of
\begin{align}\label{eq:solutions}
a(u,v)=f(v)\quad\text{for all }v\in\XX\quad{and}\quad a(u_\ell,v)=f(v)\quad\text{for all } v\in\XX_\ell.
\end{align}
We further assume that $\XX$ is a space of functions on the domain $\Omega$ and that the subspaces $\XX_\ell$ correspond to some triangulations $\TT_\ell$ discussed in detail below.

\subsection{Mesh refinement}\label{section:mesh}
Let $\TT_0$ be a triangulation of $\Omega$ into compact triangles which resolves the corners of $\Gamma$. Given two triangulations $\TT,\TT^\prime$, we write $\TT^\prime={\rm refine}(\TT,\MM)$ for some $\MM\subseteq \TT$ if 
$\TT^\prime$ is generated from $\TT$ by refinement of all $T\in\MM$ via newest vertex bisection. 
We write $\TT^\prime\in{\rm refine}(\TT)$
if $\TT^\prime$ is generated from $\TT$ by a finite number of iterated newest-vertex-bisection  refinements and we denote the set of all possible refinements by $\T:={\rm refine}(\TT_0)$.
Given  $\omega\subseteq\Omega$, we call $\TT^\prime|_\omega$ a local refinement of $\TT$,
if there exists $\TT^{\prime\prime}\in {\rm refine}(\TT)$ such that $\TT^\prime|_\omega=\TT^{\prime\prime}|_\omega$.
Given $T\in\TT$ for some $\TT\in\T$, ${\rm level}(T)$ denotes the number of bisections necessary to generate $T$ from a parent element in $\TT_0$.

We define $\NN(\TT)$ as the set of nodes of $\TT$ and $\EE(\TT)$ as the set of edges of $\TT$.

We define $h_\TT\in\PP^0(\TT)$ as the mesh-size function by $h_\TT|_T:={\rm diam}(T)$ for all $T\in\TT$.

\medskip

Given $T\in\TT\in\T$, we define the patch
\begin{align*}
 \omega(T,\TT):=\set{T^\prime\in\TT}{T\cap T^\prime\neq \emptyset}.
\end{align*}
Given a subset $\Omega^\prime\subseteq \Omega$, we define the patch
\begin{align*}
 \omega(\Omega^\prime,\TT):=\set{T\in\TT}{\big(\bigcup\omega(T,\TT)\big)^\circ\cap \Omega^\prime\neq \emptyset},
\end{align*}
where $(\cdot)^\circ$ denotes the interior of a set. Note that in case of $\Omega^\prime=T$, the two definitions coincide.
The extended patches $\omega^k(\Omega^\prime,\TT)$ are defined iteratively by
\begin{align*}
 \omega^1(\Omega^\prime,\TT):=\omega(\Omega^\prime,\TT),\quad\text{and}\quad  \omega^k(\Omega^\prime,\TT):=\omega(\bigcup\omega^{k-1}(\Omega^\prime,\TT),\TT).
\end{align*}

\begin{definition}\label{def:widehat}
	We consider an auxiliary sequence $(\widehat\TT_\ell)_{\ell\in\N}$ of uniform refinements such that $\widehat\TT_0=\TT_0$ and
	\begin{align*}
		\widehat\TT_{\ell+1}={\rm refine}^k(\widehat\TT_\ell,\widehat\TT_\ell),
	\end{align*}
	which means that each element of $\widehat\TT_\ell$ is bisected $k$-times to obtain $\widehat\TT_{\ell+1}$.
	There exist constants $C_{\rm base},C_{\rm mesh}\geq 1$ which depend on $k$ and  on $\TT_0$ such that
	\begin{align*}
	C^{-1}_{\rm base}C_{\rm mesh}^{-\ell}\leq	{\rm diam}(T)\leq C_{\rm base}C_{\rm mesh}^{-\ell}
	\end{align*}
	for all $T\in\widehat\TT_\ell$ and all $\ell\in\N$.
	We choose $k=k_{\rm mesh}$ sufficiently large such that $C_{\rm mesh}\geq (C_{\rm sz}+1)^4$, where $C_{\rm sz}$ is defined in Lemma~\ref{lem:gensz} below.
	
\end{definition}

\subsection{Adaptive algorithm}

Given a triangulation $\TT\in\T$, we assume that we can compute an error estimator $\eta(\TT)=\sqrt{\sum_{T\in\TT}\eta_T(\TT)^2}$.
In the application to the FEM-BEM coupling below, we have to restrict to adaptive triangulations with mild grading in the sense that there exists $D_{\rm grad}\in\N$ such that
\begin{align}\label{eq:graded}
 |{\rm level}(T)-{\rm level}(T^\prime)|\leq 1\quad\text{for all }T^\prime \in \omega^{D_{\rm grad}}(T,\TT).
\end{align}
This condition is necessary for the present proof and also appears in~\cite{l2opt} to prove optimal convergence in the $L^2$-norm. 
By $\T_{\rm grad}\subseteq \T$, we denote all triangulations which satisfy~\eqref{eq:graded} for a given $D_{\rm grad}\in\N$. Lemma~\ref{lem:approxclass} below shows that the restriction
does not alter the optimal convergence rate. Numerical experiments (see, e.g.,~\cite{fembem}) suggest that the restriction is not even necessary for optimal convergence rate, and thus
might just be an artifact of the proof. In the following, we assume that $D_{\rm grad}$ is sufficiently large to satisfy all the conditions in the proofs below.

We assume that the sequence $\TT_\ell$ is generated by an adaptive algorithm of the form
\begin{algorithm}\label{algorithm}
Input: $\ell=0$, $\TT_0$, $D_{\rm grad}\in\N$, $0< \theta\leq 1$, $f\in \XX^\star$.\\
For $\ell=0,1,\ldots$ do:
 \begin{enumerate}
  \item Compute $u_\ell\in\XX_\ell$.
  \item Compute error estimator $\eta_T(\TT_\ell)$ for all $T\in\TT_\ell$.
  \item Mark set of minimal cardinality $\MM_\ell\subseteq \TT_\ell$ such that
  \begin{align}\label{eq:doerfler}
   \sum_{T\in\MM_\ell}\eta_T(\TT_\ell)\geq \theta \sum_{T\in\TT_\ell}\eta_T(\TT_\ell).
  \end{align}
  \item Refine at least the elements $\MM_\ell$ of $\TT_\ell$ to obtain $\TT_{\ell+1}$.
  \item Refine additional elements to ensure that $\TT_{\ell+1}$ satisfies~\eqref{eq:graded}.
 \end{enumerate}
Output: sequence of meshes $\TT_\ell$ and corresponding solutions $u_\ell$.
\end{algorithm}

\subsection{Rate optimality}
We aim to analyze the best possible algebraic convergence rate which can be obtained by the adaptive algorithm.
This is mathematically characterized as follows: For the exact solution $u\in \XX$, we define an approximation class $\A_s$ by
\begin{align}\label{def:approxclass}
 u \in \A_s
 \quad\overset{\rm def.}{\Longleftrightarrow}\quad
\norm{u}{\A_s}:= \sup_{N\in\N}\min_{\TT\in\T\atop \#\TT-\#\TT_0\leq N}N^s\eta(\TT) < \infty.
\end{align}
 By definition, a convergence rate $\eta(\TT) = \OO(N^{-s})$ is theoretically possible if the optimal meshes are chosen.
 In view of mildly graded triangulations, we define
 \begin{align*}
   u \in \A_s^{\rm grad}
 \quad\overset{\rm def.}{\Longleftrightarrow}\quad
\norm{u}{\A_s^{\rm grad}}:= \sup_{N\in\N}\min_{\TT\in\T_{\rm grad} \atop \#\TT-\#\TT_0\leq N}N^s\eta(\TT) < \infty.
 \end{align*}
In Lemma~\eqref{lem:approxclass} below, we show that in many situations $\A_s^{\rm grad}=\A_s$.
In the spirit of~\cite{axioms}, rate optimality of the adaptive algorithm means that there exists a constant $C_{\rm opt}>0$ such that
\begin{align*}
 C_{\rm opt}^{-1}\norm{u}{\A_s}\leq\sup_{\ell\in\N_0}\frac{\eta(\TT_\ell)}{(\#\TT_\ell-\#\TT_0+1)^{-s}}\leq C_{\rm opt}\norm{u}{\A_s},
\end{align*}
for all $s>0$ with $\norm{u}{\A_s}<\infty$.

\subsection{The Axioms}\label{sec:axioms}
To formulate the axioms below, we define for a given triangulation $\TT\in\T$ the corresponding space $\XX_\TT$ as well as the discrete solution $u_\TT\in\XX_\TT$ of 
\begin{align*}
 a(u_\TT,v)=\dual{f}{v}\quad\text{for all }v\in\XX_\TT.
 \end{align*}
As proved in~\cite{axioms}, we need to check the axioms~(A1)--(A4) to ensure rate optimality for a given adaptive algorithm: There exist constant $C_{\rm red}$, $C_{\rm stab}$, $C_{\rm qo}$, $C_{\rm dlr}$,
$C_{\rm ref}>0$,
and $0\leq q_{\rm red}<1$
such that

\renewcommand{\theenumi}{{\rm A\arabic{enumi}}}%
\newcounter{subterm}%
\begin{enumerate}
\item \label{A:stable}\textbf{Stability on non-refined elements}: For all refinements $\widehat\TT\in\T$ of a triangulation $\TT\in\T$,
for all subsets $\mathcal{S}\subseteq\TT\cap\widehat\TT$ of non-refined elements, it holds that
\begin{align*}
\Big|\Big(\sum_{T\in\mathcal{S}}\eta_T(\widehat\TT)^2\Big)^{1/2}-\Big(\sum_{T\in\mathcal{S}}\eta_T(\TT)^2\Big)^{1/2}\Big|\leq C_{\rm stab}\,\norm{u_\TT- u_{\widehat\TT}}{\XX}.
\end{align*}
\item \label{A:reduction}\textbf{Reduction property on refined elements}: Any refinement $\widehat\TT\in\T$ of a triangulation $\TT\in\T$ satisfies
\begin{align*}
\sum_{T\in\widehat\TT\setminus\TT}\eta_T(\widehat\TT)^2\leq q_{\rm red} \sum_{T\in\TT\setminus\widehat\TT}\eta_T(\TT)^2 + C_{\rm red}\norm{u_\TT- u_{\widehat\TT}}{\XX}^2.
\end{align*}
\item\label{A:qosum}\textbf{General quasi-orthogonality}: For one sufficiently small $\eps\leq 0$ the output of Algorithm~\ref{algorithm}
satisfies for all $\ell,N\in\N_0$
\begin{align}\label{eq:qosum2}
 \sum_{k=\ell}^{\ell+N} \Big(\norm{u_{k+1}-u_k}{\XX}^2-C_{\rm qo}\eps\norm{u-u_k}{\XX}^2\Big)\leq C_{\rm qo} \norm{u-u_\ell}{\XX}^2.
\end{align}
\item\label{A:dlr}\textbf{Discrete reliability}:
For all refinements $\widehat\TT\in\T$ of a triangulation $\TT\in\T$,
there exists a subset $\RR(\TT,\widehat\TT)\subseteq\TT$ with $\TT\backslash\widehat\TT \subseteq\RR(\TT,\widehat\TT)$ and
$|\RR(\TT,\widehat\TT)|\le C_{\rm ref}|\TT\backslash\widehat\TT|$ such that
\begin{align*}
 \norm{u_{\widehat\TT}-u_\TT}{\XX}^2
 \leq C_{\rm dlr}^2\sum_{T\in\RR(\TT,\widehat\TT)}\eta_T(\TT)^2.
\end{align*}
\end{enumerate}

In some situations, it might even be possible to prove the stronger form of (A3), namely
\begin{align}\label{eq:qosum}
C_{\rm qo}^{-1} \norm{u-u_\ell}{\XX}^2\leq \sum_{k=\ell}^\infty \norm{u_{k+1}-u_k}{\XX}^2\leq C_{\rm qo} \norm{u-u_\ell}{\XX}^2\quad\text{\rm for all }\ell\in\N.
\end{align}

The main obstacle is the general quasi-orthogonality (A3) and its proof for the particular FEM/BEM coupling below takes up the majority of this work.
The other axioms follow from the combination of techniques for FEM and BEM. 
\begin{lemma}\label{lem:approxclass}
Under (A1)--(A2) and (A4), there holds $\A_s=\A_s^{\rm grad}$.
\end{lemma}
\begin{proof}
The assumption in~\cite[Section~A.3]{l2opt} and~\eqref{eq:graded} are equivalent up to shape regularity.
The definition of $\A_s$ and $\A_s^{\rm grad}$ implies $\A_s^{\rm grad}\subseteq \A_s$.
Let $\TT=\TT^n\in\T$ being generated from $\TT_0$ by iterated refinements $\TT^{j+1}={\rm refine}(\TT^j,\MM^j)$ for $j=0,\ldots,n$.
Since every newest-vertex bisection refinement generates at least two sons, we have $\#\TT\geq \sum_{j=0}^n\#\MM^j$.
 The result~\cite[Theorem~4]{l2opt} shows that by refining all elements in the $\MM^j$ and additionally making sure that~\eqref{eq:graded} holds,
 we find $\TT^{\rm grad}\in\T^{\rm grad}\cap {\rm refine}(\TT)$ with $\#\TT^{\rm grad}\leq C\sum_{j=0}^n\#\MM^j\leq C\#\TT$. The constant $C>0$ depend only on $D_{\rm grad}$ and $\TT_0$.
 From~\cite[Lemma~3.4]{axioms}, we see that (A1)--(A2)\&(A4) imply quasi-monotonicity of $\eta$ in the sense
 \begin{align*}
  \eta(\TT^{\rm grad})\lesssim \eta(\TT).
 \end{align*}
This shows
\begin{align*}
 \min_{\TT\in\T_{\rm grad} \atop \#\TT-\#\TT_0\leq N}N^s\eta(\TT)\lesssim  \min_{\TT\in\T \atop \#\TT-\#\TT_0\leq N/C}N^s\eta(\TT)\leq C^s\norm{u}{\A_s}
\end{align*}
and concludes the proof. 
\end{proof}

\section{General quasi-orthogonality and $LU$-factorization}\label{section:lu0}
In this section, we establish the link between general quasi-orthogonality (A3) and $LU$-factorization of infinite matrices. To that end, we first introduce exponentially decaying matrices.
\subsection{Jaffard class matrices}\label{section:exp}
Jaffard class matrices generalize the notion of matrices which decay exponentially away from the diagonal. 
The generalization allows to replace the distance $|i-j|$ between indices by a general metric $d(i,j)$. This class was introduced and analyzed in~\cite{rescue}.

\begin{definition}[Jaffard class]\label{def:wavelet}
We say that an infinite matrix $M\in\R^{\N\times\N}$ is of Jaffard class, $M\in \JJ(d,\gamma,C)$ for
some metric $d(\cdot,\cdot)\colon \N\times \N\to [0,\infty)$ and some $\gamma>0$ if for all $0<\gamma^\prime<\gamma$ there exists $C({\gamma^\prime})>0$ such that
\begin{align}\label{eq:Mexp}
|M_{ij}|\leq C({\gamma^\prime}) \exp(-\gamma^\prime d(i,j))\quad\text{for all }i,j\in\N.
\end{align}
Moreover, the metric $d(\cdot,\cdot)$ must satisfy for all $\eps>0$ 
\begin{align}\label{eq:Mdist}
	\sup_{i\in\N}\sum_{j\in\N}\exp(-\eps d(i,j))<\infty.
\end{align}
We also write $M\in \JJ$ to state the existence of parameters $d,\gamma,C$ such that $M\in \JJ(d,\gamma,C)$.
\end{definition}

\begin{definition}[banded matrix]\label{def:banded}
We say that an infinite matrix $M\in\R^{\N\times\N}$ is banded with respect to some metric $d(\cdot,\cdot)\colon \N\times \N\to [0,\infty)$ if there exists a bandwidth $b\geq 1$ such that
\begin{align}\label{eq:banded}
d(i,j)> b \quad\implies\quad M_{ij}=0\quad\text{for all }i,j\in\N.
\end{align}
In this case, we write $M\in \BB(d,b)$.
Note that we do not require $d(\cdot,\cdot)$ to satisfy~\eqref{eq:Mdist}.
We also write $M\in \BB$ or $M\in\BB(d)$ to state that the missing parameters exist.
\end{definition}	 

The following technical lemmas prove some straightforward facts about infinite matrices.
\begin{lemma}\label{lem:bandmult}
Let $M^{i,j}\in \BB(d,b_{j})$, $i=1,\ldots,n$, $j=1,\ldots,m$ for some $m,n\in\N$ with respect to some metric $d(\cdot,\cdot)$ and respective bandwidths $b_{j}\in\N$.
Then, there holds 
\begin{align*}
\sum_{i=1}^n \prod_{j=1}^m M^{i,j}\in\BB(d,\sum_{j=1}^m b_j). 
\end{align*}
\end{lemma}
\begin{proof}
Obviously, $\BB(d,b)$ is closed under summation. 
The definition of the matrix product shows 
\begin{align*}
 (M^{i,1}M^{i,2})_{ij} = \sum_{k=1}^\infty (M^{i,1})_{ik}(M^{i,2})_{kj} = \sum_{k\in\N\atop d(i,k)\leq  b_1,\,d(k,j)\leq b_2} (M^{i,1})_{ik}(M^{i,2})_{kj},
\end{align*}
and hence $(M^{i,1}M^{i,2})_{ij}=0$ if $b_1+b_2<d(i,j)\leq d(i,k)+d(k,j)$. Induction on $j$ proves $\prod_{j=1}^m M^{i,j}\in \BB(d,\sum_{j=1}^m b_j)$.
This concludes the proof.
\end{proof}

\begin{lemma}\label{lem:norm}
	Let $M\in\JJ(d,\gamma,C)$, then $|M|\colon \ell_2\to\ell_2$ is a bounded operator
	(the modulus $|M|$ is understood entry wise).
\end{lemma}
\begin{proof}
	Given $0<\gamma^\prime<\gamma$, there holds for all $i\in\N$ with~\eqref{eq:Mexp} and~\eqref{eq:Mdist}
	\begin{align*}
		\sup_{i\in\N}\sum_{j\in\N}|M_{ij}|\lesssim \sup_{i\in\N}	\sum_{j\in\N}\exp(-\gamma^\prime d(i,j))<\infty.
	\end{align*}
	Analogously, we obtain $\sup_{i\in\N}\sum_{j\in\N}|M_{ji}|<\infty$. The standard interpolation estimate $\norm{\cdot}{2}^2\leq \norm{\cdot}{1}\norm{\cdot}{\infty}$ concludes the proof.
\end{proof}

  \begin{lemma}\label{lem:invbound}
	Let $M\in\JJ(d,\gamma,C)$ such that $M$ is additionally elliptic, i.e.,
	\begin{align}\label{eq:matelliptic}
		\sup_{x\in\ell_2}\frac{Mx\cdot x}{\norm{x}{\ell_2}^2}=C_{\rm ell}>0.
	\end{align}
	Then, 
	\begin{align*}
	\overline{M}\in\R^{\N\times\N}\quad\text{with}\quad	\overline{M}_{ij}:= \sup_{n\in\N\atop n\geq \max\{i,j\}}|(M|_{\{1,\ldots,n\}\times \{1,\ldots,n\}})^{-1}|_{ij}
	\end{align*}
	is of exponential class $(d,\widetilde \gamma,\widetilde C)$ and thus a bounded operator $\overline{M}\colon \ell_2\to\ell_2$.
	The constant $\widetilde\gamma$ depends only on $C_{\rm ell}>0$, $d$, and $\gamma$, whereas for all $0<\gamma^\prime< \widetilde \gamma$, $\widetilde C(\gamma^\prime)$ depends only on an upper bound for $C(\gamma^\prime)$
	and on $C_{\rm ell}>0$.
\end{lemma}
\begin{proof} 
	The result~\cite[Proposition~2]{rescue}
	shows that $(M|_{\{1,\ldots,n\}\times \{1,\ldots,n\}})^{-1}\in\JJ(d,\widetilde \gamma,\widetilde C)$. Inspection of the proof reveals that
	$\widetilde\gamma$ depends only on $\gamma$, $d$, and $\widetilde C(\gamma^\prime)$ depends only on an upper bound for $C(\gamma^\prime)$ from Definition~\ref{def:wavelet}
	and on $C_{\rm ell}>0$. Therefore, we have for all $0<\gamma^\prime<\widetilde \gamma$
	\begin{align*}
		|\overline{M}_{ij}| \leq \widetilde C(\gamma^\prime)\exp(-\gamma^\prime d(i,j))\quad\text{for all }i,j\in\N
	\end{align*}
	and hence $\overline{M}\in\JJ(d,\widetilde\gamma,\widetilde C)$. Lemma~\ref{lem:norm} concludes the proof.
\end{proof}

\subsection{LU-factorization}\label{section:lu}
We say that a matrix $M\in\R^{\N\times\N}$ has an $LU$-factorization if
$M=LU$ for matrices $L,U\in\R^{\N\times\N}$ such that
\begin{align*}
	L_{ij}=0\text{ and }U_{ji}=0\quad\text{for all } i,j\in\N,\,i\leq j.
\end{align*}
Given a  block structure in the sense that there exist numbers $n_1,n_2,\ldots\in\N$ with $n_1=1$ and $n_i<n_j$ for all $i\leq j$, we denote matrix blocks by
\begin{align*}
 M(i,j):= M|_{\{n_i,\ldots,n_{i+1}-1\}\times \{n_j,\ldots,n_{j+1}-1\}}\in\R^{(n_{i+1}-n_i)\times (n_{j+1}-n_j)}.
\end{align*}
By $M[k]\in \R^{(n_{k+1}-1) \times (n_{k+1}-1)}$, we denote the restriction of $M$ to the first $k\times k$ blocks.

\begin{lemma}\label{lem:invapprox}
 Let $M\in\R^{\N\times\N}$ such that $M\colon\ell_2\to \ell_2$ is bounded and elliptic~\eqref{eq:matelliptic}. Moreover, let $M\in \BB(d,b_0)$.
 Assume a block structure  $n_1,n_2,\ldots\in\N$.
 Then, given $\eps>0$, there exists a bandwidth $b\in\N$ such that for all $k\in\N$, there exists $R,R_k\in \BB(d,b)$ such that
 \begin{align*}
 \norm{M^{-1}-R}{2}+\sup_{k\in\N} \norm{M[k]^{-1}-R_k}{2}\leq \eps.
 \end{align*}
  If $M$ is additionally block-banded  in the sense $M(i,j)=0$ for all $|i-j|>b_0$, then, $R_k$ and $R$ will additionally be block-banded with bandwidth $b$.
  If $M$ is block-diagonal, also $R$ and $R_k$ will be block-diagonal.
 The bandwidth $b$ depends only on $b_0$, $C_{\rm ell}$, $\norm{M}{2}$, and $\eps$.
\end{lemma}
\begin{proof}
Let $A:=M[k]$ or $A:=M$. Since $A$ is elliptic with constant $C_{\rm ell}$, we obtain for $\alpha:=C_{\rm ell}\norm{M}{2}^{-2}$ and $x\in\ell_2$
 \begin{align*}
  \norm{x-\alpha Ax}{\ell_2}^2 &= \norm{x}{\ell_2}^2 - 2\alpha\dual{x}{Ax}_{\ell_2} + \alpha^2 \norm{Ax}{\ell_2}^2\\
  &\leq (1-2\alpha C_{\rm ell}+\alpha^2\norm{A}{2}^2)\norm{x}{\ell_2}^2\leq(1-C_{\rm ell}^2/\norm{M}{2}^2)\norm{x}{\ell_2}^2.
 \end{align*}
This shows $\norm{I-\alpha A}{2}\leq (1-C_{\rm ell}^2/\norm{M}{2}^2):= q<1$.
We obtain
 \begin{align*}
  A^{-1} = \alpha(\alpha A)^{-1} = \alpha(I-(I-\alpha A))^{-1} =\alpha\sum_{k=0}^\infty (I-\alpha A)^k.
 \end{align*}
We define $R$ ($R_k$) by $R:= \sum_{k=0}^N (I-\alpha A)^k$ for some $N\in\N$ such that $\norm{\sum_{k=N+1}^\infty (I-\alpha A)^k}{2}\leq \sum_{k=N+1}^\infty q^k\leq \eps$.
Since $I-\alpha A\in \BB(d,b_0)$, Lemma~\ref{lem:bandmult} shows that $R$ is banded as well. The bandwidth depends only on $b_0$, $q$, and $N$.
If $M$ is additionally block-banded, also $A$ and $(I-\alpha A)$ will be block-banded with bandwidth $b_0$. Hence $(I-\alpha A)^k$ will be block-banded with bandwidth $kb_0$. 
The same argumentation proves the statement for block-diagonal $M$.
This concludes the proof.
 \end{proof}

The following results prove that block-banded matrices $M$ hand down some structure to their $LU$-factors. This is used in Section~\ref{section:apps} to construct
 suitable hierarchical bases for FEM/BEM coupling.
 
\begin{lemma}\label{lem:blockLU1}
  Let $M\in\R^{\N\times\N}$ such that $M\colon\ell_2\to \ell_2$ is bounded and is elliptic~\eqref{eq:matelliptic}.
  Assume a block structure in the sense that there exist numbers  $n_1,n_2,\ldots\in\N$ with $n_1=1$ and $n_i<n_j$ for all $i\leq j$.
 Moreover, let $M$ be block-banded in the sense $M(i,j)=0$ for $|i-j|>b_0$ for some $b_0\in\N$.
 Then, the block-$LU$-factorization $M=LU$ for block-upper/block-lower triangular matrices $L,U\in\R^{\N\times\N}$ such that $L(i,i)=I$ for all $i\in\N$ exists, is block-banded with bandwidth $b_0$,
 and satisfies
\begin{align*}
 \norm{L}{2}+\norm{U}{2}+\norm{L^{-1}}{2}+\norm{U^{-1}}{2}<\infty.
 \end{align*}
 Moreover, the block-diagonal matrix $D\in\R^{\N\times\N}$, $D(i,i):= U(i,i)$ is bounded and elliptic~\eqref{eq:matelliptic} with bounded and elliptic inverse.
\end{lemma}
\begin{proof}
Since $M[k]$ has invertible principal sub matrices, it is well-known that the (block)-LU-factorization exists. 
The triangular structure of the factors implies that also $M$ has a block-$LU$-factorization. Since $M[k]$ and $L[k]$ are invertible by definition, 
$U[k]$ is invertible for all $k\in\N$. The block-triangular structure guarantees that $L[k]^{-1}=(L^{-1})[k]$, $U[k]^{-1}=(U^{-1})[k]$, and hence existence of $L^{-1},U^{-1}$ as matrices in $\R^{\N\times\N}$.
Moreover, it is well-known that $L$ and $U$ are block-banded with bandwidth $b_0$.
By definition, we have $M[k]^{-1} = U[k]^{-1} L[k]^{-1}$. Since $L$ is lower-triangular with normalized block-diagonal (only identities in the diagonal blocks), the same is true for $L[k]^{-1}$. Therefore, we obtain
\begin{align}\label{eq:luid}
 M[k]^{-1}(i,k) = \sum_{r=1}^kU[k]^{-1}(i,r) L[k]^{-1}(r,k) =U[k]^{-1}(i,k)=U^{-1}(i,k),
\end{align}
where the last identity follows from the fact that $U^{-1}$ is upper-block triangular.

Since $M[k]$ is bounded and elliptic~\eqref{eq:matelliptic}, also $M[k]^{-1}$ is bounded and elliptic. 
Therefore, we see that $U^{-1}(k,k)=U(k,k)^{-1}$ (since $U$ is block-triangular) is bounded and elliptic and thus conclude that $D^{-1}$ and also $D$ are bounded and elliptic. Moreover, we see that 
$\sup_{i,j}\norm{U^{-1}(i,j)}{2}\leq\sup_{k\in\N}\norm{M[k]^{-1}}{2}\leq C_{\rm ell}<\infty$. Hence, there holds
\begin{align*}
 \sup_{i,j\in\N}\norm{L(i,j)}{2}\leq \sup_{i,j\in\N}\sum_{|k-i|\leq b_0} \norm{M(i,k)U^{-1}(k,j)}{2}\leq 2b_0\norm{M}{2}\sup_{i,j}\norm{U^{-1}(i,j)}{2}<\infty.
\end{align*}
Since $L$ is block-banded with bandwidth $b_0$, Lemma~\ref{lem:blockstab} shows $\norm{L}{2}<\infty$. This implies $\norm{U^{-1}}{2}\leq \norm{M^{-1}}{2}\norm{L}{2}<\infty$.

Let $M^T =\widetilde L\widetilde U$ be the analogous block-LU-factorization for the transposed matrix (note that $M^T$ is still elliptic, bounded, and banded).
Since normalized $LU$-factorizations are unique, we see that
\begin{align}\label{eq:lut}
\widetilde L = U^TD^{-T}\quad \text{and}\quad \widetilde U= D^TL^T
\end{align}
Repeating the above arguments shows $\norm{\widetilde L}{2}+\norm{\widetilde U^{-1}}{2}<\infty$. With boundedness and ellipticity of $D$,~\eqref{eq:lut} shows $\norm{U}{2}+\norm{L^{-1}}{2}<\infty$.
Altogether, we prove the statement and conclude the proof.
 \end{proof}
 
\begin{lemma}\label{lem:blockLU2}
  Under the assumptions of Lemma~\ref{lem:blockLU1}, assume that additionally $M\in \BB(d,b_0)$.
  Given $\eps>0$, there exists $b\in\N$ and block-upper triangular $U_\eps^{-1}\in\BB(d,b)$ which is additionally block-banded in the sense $U_\eps^{-1}(i,j)=0$ for $|i-j|>b$ such that
  \begin{align*}
   \norm{U^{-1}-U_\eps^{-1}}{2}\leq \eps.
  \end{align*}
  The approximation $U_\eps^{-1}$ is invertible with bounded inverse such that $\sup_{\eps>0}(\norm{U_\eps}{2}+\norm{U_\eps^{-1}}{2})<\infty$.
  Moreover, there exists block-diagonal $D_\eps\in\BB(d,b)$ which is bounded and elliptic such that $\norm{D-D_\eps}{2}\leq \eps$.
\end{lemma}
\begin{proof}
Lemma~\ref{lem:invapprox} shows that there exist $R,R_k\in \BB(d,b)$ which are block-banded with bandwidth $b=b(\eps)$ such that $\norm{M^{-1}-R}{2}+\norm{M[k]^{-1}-R_k}{2}\leq \eps$ for all $k\in\N$.
Since $M$ is bounded and elliptic, also $M^{-1}$ is bounded and elliptic~\eqref{eq:matelliptic}. Choosing $\eps>0$ sufficiently small, we ensure that also $R$ and $R_k$ are bounded an elliptic with uniform
constant. 

Inspired by~\eqref{eq:luid}, we define a first approximation to $U^{-1}$ by 
\begin{align}\label{eq:defueps}
T(i,j):=\begin{cases}
         0 &\text{for }j<i\text{ or }i<j-b,\\
         R_j(i,j) &\text{for }j-b\leq i\leq j.
\end{cases}
\end{align}
This ensures that $T\in \BB(d,b)$ and  that $T$ is block-banded with bandwidth $b$.
Additionally, we obtain
\begin{align}\label{eq:err1}
 \sup_{i,j\in\N}\norm{T(i,j)-U^{-1}(i,j)}{2}\lesssim \eps.
\end{align}

We define an approximation to $L$ (which is block-banded with bandwidth $b_0$) by 
\begin{align*}
 S(i,j):=\begin{cases}
 0 & \text{for }j<i-b_0\text{ or } j>i,\\
 I &\text{for }j=i,\\
 (M T)(i,j)&\text{for } i-b_0\leq j<i.
\end{cases}
 \end{align*}
The definition and~\eqref{eq:err1} imply 
\begin{align}\label{eq:err2}
 \norm{L(i,j)-S(i,j)}{2}&\leq \norm{\sum_{k=i-b_0}^{i+b_0} M(i,k)(U^{-1}(k,j)-T(k,j))}{2}\\
 &\leq \sum_{k=i-b_0}^{i+b_0}\norm{M}{2}\norm{U^{-1}(k,j)-T(k,j))}{2}\lesssim \norm{M}{2}b_0 \eps.
\end{align}
Since both $L$ and $S$ are block-banded with bandwidth $b_0$, Lemma~\ref{lem:blockstab} shows even
\begin{align}\label{eq:err3}
 \norm{L-S}{2}\lesssim \eps,
\end{align}
where the hidden constant is independent of $\eps$. Moreover, Lemma~\ref{lem:bandmult} shows that $L\in \BB(d,\widetilde b)$, for some $\widetilde b\in\N$ which depends only on $b_0$ and $b$.

Recall $R$ from above with $\norm{M^{-1}-R}{2}\leq \eps$, $R\in\BB(d,b)$ and $R$ is block-banded with bandwidth $b$.
This allows to define $U_\eps^{-1}$ by
\begin{align*}
 U_\eps^{-1}:= RS.
\end{align*}
We obtain from the definition and with~\eqref{eq:err3}
\begin{align}\label{eq:err4}
\begin{split}
 \norm{U_\eps^{-1}-U^{-1}}{2}&\leq \norm{R(S-L)}{2}+\norm{(R-M^{-1})L}{2}\lesssim\norm{R}{2}\eps +\norm{L}{2}\eps\\
 &\leq (C_{\rm ell}+\eps+1)\eps,
 \end{split}
\end{align}
where the hidden constant does not depend on $\eps$.
Moreover, Lemma~\ref{lem:bandmult} shows (since $S$ and $R$ are block-banded), that $U_\eps^{-1}\in\BB(d)$ with bandwidth depending on $\eps$.
Analogously, we see $U_\eps^{-1}$ is block-banded with bandwidth $b_0+b$. Since $U^{-1}$ is invertible with bounded inverse, choosing $\eps>0$ sufficiently small ensures
that $U_\eps^{-1}$ is invertible, with bounded inverse uniformly in $\eps$.

Let $\widetilde D$ denote the block-diagonal of $U_\eps^{-1}$.
Obviously, $\widetilde D\in\BB(d)$ and~\eqref{eq:err4} implies $\norm{\widetilde D-D^{-1}}{2}\lesssim \eps$. Lemma~\ref{lem:blockLU1} shows that $D^{-1}$ is bounded and elliptic, thus sufficiently small 
$\eps>0$ guarantees the same for $\widetilde D$. Hence, Lemma~\ref{lem:invapprox} ensures that there exists block-diagonal $D_\eps\in\BB(d)$ (with bandwidth depending only on $\eps>0$), such that
$\norm{\widetilde D^{-1}-D_\eps}{2}\leq \eps$. From this, we obtain
\begin{align*}
 \norm{D-D_\eps}{2}&\leq \norm{\widetilde D^{-1}-D_\eps}{2}+\norm{\widetilde D^{-1}-D}{2}\leq \eps +\norm{\widetilde D^{-1}}{2}\norm{D}{2}\norm{\widetilde D-D^{-1}}{2}\\
 &\lesssim (1 +\norm{\widetilde D^{-1}}{2}\norm{D}{2})\eps.
\end{align*}
For sufficiently small $\eps>0$, $\norm{\widetilde D^{-1}}{2}$ is bounded in terms of $\norm{D}{2}$. This ensures that the constant above does not depend on $\eps$ and thus concludes the proof.
\end{proof}

\begin{lemma}\label{lem:blockLU3}
Under the assumptions of Lemma~\ref{lem:blockLU1}--\ref{lem:blockLU2},
there exists $b\in\N$ and block-lower triangular $L_\eps^{-1}\in\BB(d,b)$ with $L_\eps^{-1}(i,i)=I$, which is additionally block-banded in the sense $L_\eps^{-1}(i,j)=0$ for $|i-j|>b$ such that
  \begin{align*}
   \norm{L^{-1}-L_\eps^{-1}}{2}\leq \eps.
  \end{align*}
  The approximation $L_\eps^{-1}$ is invertible such that $\sup_{\eps>0}(\norm{L_\eps}{2}+\norm{L_\eps^{-1}}{2})<\infty$.
\end{lemma}
\begin{proof}
Recall that $M^T$ satisfies all the assumptions of Lemma~\ref{lem:blockLU1}--\ref{lem:blockLU2}.
Let $M^T=\widetilde L \widetilde U$. We apply Lemma~\ref{lem:blockLU1} to $M^T$ to obtain an approximation $\widetilde U_\eps^{-1}\in\BB(d,b)$, 
block-banded with bandwidth $b$, bounded with bounded inverse (uniformly in $\eps$) such that
\begin{align*}
 \norm{\widetilde U^{-1}-\widetilde U_\eps^{-1}}{2}\leq \eps.
\end{align*}
The identity~\eqref{eq:luid} shows $L^{-1}=D\widetilde U^{-T}$ and thus motivates the definition
\begin{align*}
 L_\eps^{-1}(i,j):=\begin{cases} (D_\eps\widetilde U_\eps^{-T})(i,j)& i\neq j,\\
                    I &i=j.
                   \end{cases}
\end{align*}
with $D_\eps\in\BB(d,b)$ from Lemma~\ref{lem:blockLU1} applied to $M$. Lemma~\ref{lem:bandmult} shows $L_\eps^{-1}\in \BB(d)$ and $L_\eps^{-1}$ is also block-banded with bandwidth $b$.
We obtain with the approximation estimates from Lemma~\ref{lem:blockLU1}
\begin{align*}
 \norm{L_\eps^{-1}-L^{-1}}{2}&\leq \norm{D_\eps\widetilde U_\eps^{-T} - D\widetilde U^{-T}}{2} + \sup_{i\in\N}\norm{(D_\eps\widetilde U_\eps^{-T})(i,i)-I}{2}
 \end{align*}
 
 The first term on the right-hand side is bounded by
 \begin{align*}
 \norm{D_\eps\widetilde U_\eps^{-T} - D\widetilde U^{-T}}{2}&\leq \norm{(D_\eps-D)\widetilde U_\eps^{-T}}{2}+\norm{D(\widetilde U_\eps^{-T}-\widetilde U^{-T})}{2}\\
 &\leq 
 \eps(\norm{\widetilde U^{-1}}{2}+\eps)+\norm{D}{2}\eps\lesssim \eps.
\end{align*}
The second term satisfies
\begin{align*}
  \sup_{i\in\N}\norm{(D_\eps\widetilde U_\eps^{-T})(i,i)-I}{2}\leq \norm{D_\eps\widetilde U_\eps^{-T}-L}{2}=\norm{D_\eps\widetilde U_\eps^{-T} - D\widetilde U^{-T}}{2}\lesssim \eps.
\end{align*}
Choosing $\eps>0$ sufficiently small ensures that $L_\eps^{-1}$ is invertible with bounded inverse uniformly in $\eps>0$. This concludes the proof.
\end{proof}

\begin{theorem}\label{thm:blockLU}
 Under the assumptions of Lemma~\ref{lem:blockLU1}--\ref{lem:blockLU3}, 
  there exists an approximate block-$LDU$-decomposition $\norm{M-LDU}{2}\leq \eps$ for block-upper/block-lower triangular factors $L$, $U$ such that $L(i,i)=U(i,i)=I$ for all $i\in\N$,
and a block diagonal factor $D$. 
The factors $L,D,U\colon \ell_2\to \ell_2$ are bounded with bounded inverses uniformly in $\eps$ and satisfy $L^{-1},D,U^{-1}\in \BB(d,b)$ for some bandwidth $b$.
Moreover, $L^{-1},U^{-1}$ are block-banded with bandwidth $b$, i.e., $L^{-1}(i,j)=U^{-1}(i,j)=0$ for all $|i-j|>b$. Additionally, $D$ is elliptic~\eqref{eq:matelliptic}.
The constant $b$ depends only on $M$, $b_0$, and $\eps$.
\end{theorem}
\begin{proof}
To avoid confusion, we will denote the $LU$-factorization of $M$ from Lemma~\ref{lem:blockLU1} by $\widetilde L$ and $\widetilde U$, with diagonal matrix $\widetilde D$.
 With Lemma~\ref{lem:blockLU2}--\ref{lem:blockLU3}, we set $D:=D_\eps$ and $L^{-1}:= L_\eps^{-1}$. This ensures $L^{-1},D\in\BB(d)$ and that $L^{-1}$ is block-banded. Moreover, $D$ is bounded and elliptic.
This motivates the definition 
\begin{align*}
 U^{-1}(i,j):=\begin{cases} ( U_\eps^{-1}D_\eps)(i,j)&i\neq j,\\
               I&i=j.
              \end{cases}
\end{align*}
Lemma~\ref{lem:bandmult} shows that $U^{-1}\in\BB(d)$ with bandwidth depending on $\eps$ and moreover $U^{-1}$ is block-banded.
We obtain
\begin{align*}
 \norm{M^{-1}-U^{-1}D^{-1}L^{-1}}{2}&\leq\norm{M^{-1}-U_\eps^{-1}L_\eps^{-1}}{2} +\sup_{i\in\N}\norm{ ( U_\eps^{-1}D_\eps)(i,i)-I}{2}\\
 &\leq \norm{M^{-1}-U_\eps^{-1}L_\eps^{-1}}{2}+\norm{U_\eps^{-1}D_\eps -\widetilde U^{-1} \widetilde D}{2}.
\end{align*}
The first term on the right-hand side can be bounded by use of Lemma~\ref{lem:blockLU2}--\ref{lem:blockLU3} by
\begin{align*}
 \norm{M^{-1}-U_\eps^{-1}L_\eps^{-1}}{2}\leq \norm{\widetilde U^{-1}}{2}\norm{\widetilde L^{-1}-L_\eps^{-1}}{2}+ \norm{\widetilde U^{-1}-U_\eps^{-1}}{2}\norm{L_\eps^{-1}}{2}\lesssim \eps,
\end{align*}
where the hidden constant does not depend on $\eps>0$.
The second term can be bounded in a similar fashion by
\begin{align*}
 \norm{U_\eps^{-1}D_\eps -\widetilde U^{-1} \widetilde D}{2}\leq \norm{U_\eps^{-1}}{2}\norm{D_\eps - \widetilde D}{2}+\norm{U_\eps^{-1} -\widetilde U^{-1}}{2}\norm{\widetilde D}{2}\lesssim \eps
\end{align*}
with $\eps$-independent hidden constant.
Altogether we proved
\begin{align*}
  \norm{M^{-1}-U^{-1}D^{-1}L^{-1}}{2}\lesssim \eps.
\end{align*}
There holds
\begin{align*}
  \norm{M-LDU}{2}&\leq \norm{M}{2}\norm{LDU}{2} \norm{M^{-1}-U^{-1}D^{-1}L^{-1}}{2}\\
  &\leq(\norm{M}{2}^2+\norm{M}{2}\norm{LDU-M}{2}) \norm{M^{-1}-U^{-1}D^{-1}L^{-1}}{2}\\
  &  \lesssim   \norm{M}{2}^2\eps
  + \norm{M}{2}\norm{M-LDU}{2} \eps.
\end{align*}
Sufficiently small $\eps>0$ shows
\begin{align*}
  \norm{M-LDU}{2}\lesssim \eps,
\end{align*}
where the hidden constant does not depend on $\eps$. This concludes the proof.
\end{proof}

The following results establishes existence of a bounded $LU$-factorization for particular Jaffard class matrices.
\begin{theorem}\label{thm:LU}
Let $M\in\R^{\N\times \N}\in\JJ(d,\gamma,C)$ and additionally be elliptic~\eqref{eq:matelliptic}.
Then, $M$ has an $LU$-factorization such that $|L|,|U|,|L^{-1}|,|U^{-1}|\colon \ell_2\to\ell_2$ are bounded operators.
\end{theorem}
\begin{proof}
Lemma~\ref{lem:invbound} together with~\cite[Theorem~2]{LU} show that $|L^{-1}|,|U^{-1}|\colon \ell_2\to\ell_2$ are bounded operators.
From this and Lemma~\ref{lem:norm}, we infer the estimates $\norm{|L|}{2}$ $=$ $\norm{|MU^{-1}|}{2}$ $\leq$ $\norm{|M|}{2}\norm{|U^{-1}|}{2}<\infty$ as well as 
$\norm{|U|}{2}=\norm{|L^{-1}M|}{2}\leq \norm{|L^{-1}|}{2}\norm{|M|}{2}<\infty$ and conclude the proof.
\end{proof}

The following three theorems connect existence of bounded $LU$-factors with general quasi-orthogonality.
\begin{theorem}\label{thm:luqo}
Let there exist a Riesz basis $(w_n)_{n\in\N}$ of $\XX$ and a constant $C>0$ such that all $x=\sum_{n\in\N} \lambda_n w_n\in \XX$ satisfy
\begin{align}\label{eq:schauder}
C^{-1}\norm{x}{2}^2\leq \sum_{n\in\N} \lambda_n^2\leq C\norm{x}{2}^2
\end{align}
and there holds $\XX_\ell={\rm span}\set{w_n}{n\in\{1,\ldots,N_\ell\}}$ for some constants $N_\ell\in\N$ with $N_{\ell}< N_{\ell+1}$. 
If $M_{ij}:=a(w_j,w_i)$ and $M\in\JJ$, then there holds general quasi-orthogonality~\eqref{eq:qosum}.
The constant $C_{\rm qo}$ depends only on the basis $(w_n)$, $a(\cdot,\cdot)$, $C$, $c_0$, the Jaffard class $\JJ$, and $\XX$.
\end{theorem}
\begin{proof}
The $N_\ell$ induce a block structure.
By~\eqref{eq:schauder}, the matrix $M$ is bounded and elliptic~\eqref{eq:matelliptic}.
Since $M\in\JJ$, Theorem~\ref{thm:LU} shows that there exists a bounded $LU$-factorization $M=LU$. Let $u=\sum_{n=1}^\infty \lambda_n w_n$ and $u_\ell=\sum_{n=1}^\infty \lambda(\ell)_n w_n$.
With $\lambda:=(\lambda_1,\lambda_2,\ldots)$ and $\lambda(\ell):=(\lambda(\ell)_1,\lambda_2(\ell),\ldots,\lambda(\ell)_{N_\ell},0,\ldots)$, $F:=(\dual{f}{w_1},\dual{f}{w_2},\ldots)\in\R^\N$,
there holds
\begin{align}\label{eq:l2sol}
M\lambda=F\quad\text{\rm and}\quad M[\ell]\lambda(\ell)=F[\ell]\quad\text{\rm for all }\ell\in\N,
\end{align}
where $F[\ell]:=(F_1,\ldots,F_{N_\ell},0,\ldots)$. Moreover, there holds $M[\ell]=L[\ell]U[\ell]$ for any $LU$-factorization.
Due to the triangular structure of $L$, there holds for all $1\leq i\leq N_\ell$ that
\begin{align*}
(L[\ell]U[\ell]\lambda(\ell))_i=(M[\ell]\lambda(\ell))_i=F[\ell]_i=F_i=(M\lambda)_i=(LU\lambda)_i=(L[\ell]U\lambda)_i.
\end{align*}
Since $L$ and hence also $L[\ell]$ is regular, this shows that $(U\lambda)_i=(U[\ell]\lambda(\ell))_i$ for all $1\leq i\leq N_\ell$.
Moreover, there holds $U[\ell]\lambda(\ell)=U\lambda(\ell)$ due to the upper triangular structure of $U$. Altogether, this proves
\begin{align*}
	(U\lambda)_i=(U\lambda(\ell))_i\quad\text{for all }1\leq i\leq N_\ell\quad\text{and}\quad (U\lambda(\ell))_i=0\quad\text{for all }i>N_\ell.
\end{align*}
Hence, we have, by use of the boundedness of $U$ and $U^{-1}$ and~\eqref{eq:schauder}, that
\begin{align*}
\norm{u_{k+1}-u_k}{\XX}\simeq \norm{\lambda(k+1)-\lambda(k)}{\ell_2}\simeq \norm{U\lambda(k+1)-U\lambda(k)}{\ell_2}=
\norm{(U\lambda)|_{\{N_k,\ldots,N_{k+1}-1\}}}{\ell_2}.
\end{align*}
This shows
\begin{align*}
\sum_{k=\ell}^\infty \norm{u_{k+1}-u_k}{\XX}^2&\simeq \sum_{k=\ell}^\infty \norm{(U\lambda)|_{\{N_k,\ldots,N_{k+1}-1\}}}{\ell_2}^2= \norm{(U\lambda)|_{\{N_\ell,\ldots\}}}{\ell_2}^2\\
&=\norm{U\lambda - U\lambda(\ell)}{\ell_2}^2\simeq \norm{\lambda - \lambda(\ell)}{\ell_2}^2\simeq \norm{u-u_\ell}{\XX}^2.
\end{align*}
Hence, we conclude the proof.
\end{proof}

\begin{theorem}\label{thm:luqo2}
	With the spaces and basis functions from Theorem~\ref{thm:luqo} assume that for some $\eps>0$, there exists $M^\eps\in\R^{\N\times\N}$ such that $M_{ij}:=a(w_j,w_i)$,
	$M\in \R^{\N\times\N}$ satisfies $\norm{M-M^\eps}{2}\leq \eps$.
	If $M^\eps$ is elliptic~\eqref{eq:matelliptic} and $M^\eps\in\JJ$, then there holds general quasi-orthogonality~\eqref{eq:qosum2}.
	The constant $C_{\rm qo}>0$ depends only on the basis $(w_n)$, $a$, $C$, the Jaffard class $\JJ$, and $\XX$.
\end{theorem}
\begin{proof}
Note that $M^\eps\in\JJ$ implies $\norm{M^\eps}{2}<\infty$ (Lemma~\ref{lem:norm}).
	With the notation from the proof of Theorem~\ref{thm:luqo}, we apply Theorem~\ref{thm:luqo} to the bilinear form $a^\eps\colon \ell_2\times\ell_2\to\R$, $a^\eps(x,y):=\dual{M^\eps x}{y}_{\ell_2}$
	and $f^\eps:=M^\eps \lambda\in\ell_2$. Ellipticity of and boundedness of $M^\eps$ implies boundedness and ellipticity~\eqref{eq:aelliptic} of $a^\eps(\cdot,\cdot)$.
	We use the $\ell_2$ unit vectors as the Riesz basis to obtain with Theorem~\ref{thm:luqo}
	\begin{align}\label{eq:lambdaqo}
		\sum_{k=\ell}^\infty \norm{\lambda^\eps(k+1)-\lambda^\eps(k)}{\ell_2}^2 \simeq \norm{\lambda - \lambda^\eps(\ell)}{\ell_2}^2.
	\end{align}
	Here, we used that $\lambda^\eps=\lambda$ by definition of $f^\eps$.
	
	We identify vectors in $\R^n$ with vectors in $\R^\N$ by adding zeros. Then, there holds
	\begin{align*}
		\norm{\lambda^\eps(k)-\lambda(k)}{\ell_2}^2&\lesssim a^\eps(\lambda^\eps(k)-\lambda(k),\lambda^\eps(k)-\lambda(k))\\
		&= a^\eps(\lambda-\lambda(k),\lambda^\eps(k)-\lambda(k))\\
		&= \dual{(M^\eps-M)(\lambda-\lambda(k))}{\lambda^\eps(k)-\lambda(k)}_{\ell_2}\\
		&\leq \norm{M^\eps-M}{2}\norm{\lambda-\lambda(k)}{\ell_2}\norm{\lambda^\eps(k)-\lambda(k)}{\ell_2}.
	\end{align*}
	Hence, we have
	\begin{align*}
		\big|\norm{\lambda^\eps(k+1)-\lambda^\eps(k)}{\ell_2}-\norm{\lambda(k+1)-\lambda(k)}{\ell_2}\big|\lesssim\eps(\norm{\lambda-\lambda(k)}{\ell_2}+\norm{\lambda-\lambda(k+1)}{\ell_2}).
	\end{align*}
	With~\eqref{eq:lambdaqo}, this concludes the proof.
\end{proof}
\begin{theorem}\label{thm:luqo3}
  With the spaces and basis function from Theorem~\ref{thm:luqo} assume that there exists another Riesz basis $(v_n)_{n\in\N}$ which satisfies the same conditions as $(w_n)$ in Theorem~\ref{thm:luqo}.
  Assume that for some $\eps>0$, there exists $M^\eps\in\R^{\N\times\N}$ such that $M_{ij}:=a(v_j,w_i)$, $M\in \R^{\N\times\N}$ satisfies $\norm{M-M^\eps}{2}\leq \eps$.
	If $M$ and $M^\eps$ are elliptic~\eqref{eq:matelliptic} and $M^\eps\in\JJ$, then there holds general quasi-orthogonality~\eqref{eq:qosum2}.
	The constant $C_{\rm qo}>0$ depends only on the basis $(w_n)$, $a$, $C_{\rm ell}$, $C$, the Jaffard class $\JJ$, and $\XX$.
\end{theorem}
\begin{proof}
	With the notation from the proof of Theorem~\ref{thm:luqo}, we apply Theorem~\ref{thm:luqo2} to the bilinear form  $\widetilde a\colon \ell_2\times\ell_2\to \R$
	defined by $\widetilde a(x,y):=\dual{Mx}{y}_{\ell_2}$. Let $M^\eps$ as in the statement and choose the $\ell_2$-unit vectors as the Riesz bases.
	Note that the Riesz bases condition~\eqref{eq:schauder} ensures that $M$ and $M^\eps$ are bounded operators in $\ell_2$ and thus Theorem~\ref{thm:luqo2} is applicable.
	Thus, we obtain for all $\ell,N\in\N$
	\begin{align}\label{eq:lambdaqo2}
		\sum_{k=\ell}^{\ell+N}\norm{\lambda(k+1)-\lambda(k)}{\ell_2}^2-C_{\rm qo}^\prime\eps\norm{\lambda-\lambda(k)}{\ell_2}^2 \leq C_{\rm qo}^\prime \norm{\lambda - \lambda(\ell)}{\ell_2}^2.
	\end{align}
	By definition, the vectors $\lambda$ and $\lambda(k)$ satisfy the equations~\eqref{eq:l2sol}.
	Definition of $M$ implies
	\begin{align*}
	 a(\sum_{j=1}^\infty \lambda_j v_j,w_i)=\dual{f}{w_i}\quad\text{for all }i\in\N.
	\end{align*}
	Hence, we know $\sum_{j=1}^\infty \lambda_j v_j=u$. Since $(v_n)$ and $(w_n)$ span the same subspaces $\XX_\ell$, we get analogously
	\begin{align*}
	  a(\sum_{j=1}^{N_\ell} \lambda(\ell)_j v_j,w_i)=\dual{f}{w_i}\quad\text{for all }1\leq i\leq N_\ell.
	\end{align*}
	This shows $\sum_{j=1}^{N_\ell} \lambda(\ell)_j v_j=u_\ell$. 
	Thus, by use of~\eqref{eq:schauder} for $(v_n)$, we rewrite~\eqref{eq:lambdaqo2} and conclude
	\begin{align*}
	 \sum_{k=\ell}^{\ell+N}\norm{u_{k+1}-u_k}{\XX}^2&-C^2C_{\rm qo}^\prime\eps\norm{u-u_k}{\XX}^2\\
	 &\leq \sum_{k=\ell}^{\ell+N}C\norm{\lambda(k+1)-\lambda(k)}{\ell_2}^2-CC_{\rm qo}^\prime\eps\norm{\lambda-\lambda(k)}{\ell_2}^2\\
	 &\leq CC_{\rm qo}^\prime\norm{\lambda - \lambda(\ell)}{\ell_2}^2\leq C^2C_{\rm qo}^\prime\norm{u-u_\ell}{\XX}^2.
	\end{align*}
	We conclude the proof with $C_{\rm qo}:= C^2C_{\rm qo}^\prime$ which is independent of $\eps>0$.
\end{proof}

\section{Metrics on hierarchical function spaces}\label{section:metric}
To connect the theory of Jaffard class/banded matrices from Section~\ref{section:exp} to the applications in Section~\ref{section:apps}, we require several metrics
on suitable functions spaces.

\medskip

Assume a set of functions with simply connected supports $B=\bigcup_{\ell\in\N} B_\ell$ on $\Omega$ with the following properties:
\begin{enumerate}
\item[(i)] $B_\ell\cap B_k=\emptyset$ and $\bigcup_{v\in B_\ell}{\rm supp}(v) =\overline{\Omega}$ for all $\ell,k\in\N$,
 \item[(ii)]$\#B_\ell\leq C_{\rm base}C_{\rm mesh}^{2\ell}$ and $\#\set{v\in B_\ell}{v|_{\Omega\setminus\Gamma}=0}\leq C_{\rm base}C_{\rm mesh}^{\ell}$
 for all $\ell\in\N$,
 \item[(iii)] $C_{\rm base}^{-1} C_{\rm mesh}^{-\ell}\leq {\rm diam}({\rm supp}(v))\leq C_{\rm base} C_{\rm mesh}^{-\ell}$ 
 for all $v\in B_\ell$, all $\ell\in \N$,
 \item[(iv)] ${\rm diam}({\rm supp}(v))\leq C_{\rm base} |{\rm supp}(v)|^{1/2}$  for all $v\in B_\ell$ with $v|_{\Omega\setminus\Gamma}\neq 0$ and all $\ell\in \N$.
\end{enumerate}
Assume an arbitrary but one-to-one numbering of all basis functions in $B$, i.e., $B=\{w_1,w_2,\ldots\}$. We define $L(w_i):= \ell$ if $w_i\in B_\ell$.
\begin{definition}\label{def:metric}
 Define the following functions:
 \begin{itemize}
 \item  $\delta\colon B\times B\to \{0,1\}$ is defined by $\delta(v,w)=1$ if $v\neq w$ and $\delta(v,w)=0$ if $v=w$.
 \item  $\delta_k\colon B\times B\to \N$ is defined by 
   \begin{align*}
   \delta_k(v,w):=\min\set{n\in\N}{&\exists T_1,\ldots,T_{n}\in\widehat\TT_{k},\,{\rm mid}(v)\cap T_1\neq \emptyset,\\
   &\,{\rm mid}(w)\cap T_n\neq \emptyset,\,T_i\cap T_{i+1}\neq \emptyset,\,i=1,\ldots,n-1},
  \end{align*}
  where ${\rm mid}(\cdot)$ denotes the barycenter of the support of the function.
  \item $d_1\colon B\times B\to \N$ is defined by
  \begin{align*}
   d_1(v,w):=\delta_{\min\{L(v),L(w)\}}(v,w).
  \end{align*}
\item Given $\beta>0$, $d_2\colon B\times B\to [0,\infty)$ is defined by
\begin{align*}
 d_2(v,w):= \delta(v,w)+\beta|L(v)-L(w)| + \log(\delta(v,w)+d_1(v,w)).
\end{align*}
\item Given $\gamma>0$, $d_3\colon B\times B\to [0,\infty)$ is defined by
\begin{align*}
 d_3(v,w):= \begin{cases}
             \gamma^{\max\{L(v),L(w)\}} & L(v)\neq L(w),\\
             \delta(v,w)+d_1(v,w)-1& L(v) = L(w).
            \end{cases}
\end{align*}
 \end{itemize}

\end{definition}

In the following, we prove certain properties for the functions defined above.

\begin{lemma}\label{lem:dkmetric}
 Let $k,n\in\N$  and let $u,v,w\in B$. Then, there holds
\begin{align}\label{eq:dequiv}
 \delta_k(v,w)\leq\begin{cases}
  C_{\TT_0}C_{\rm base}^2C_{\rm mesh}^{k-n}\delta_n(v,w)&\text{for }k>n,\\
  \delta_n(v,w)&\text{else,}
                  \end{cases}
\end{align}
where $C_{\TT_0}>0$ depends only on $\TT_0$, 
 as well as
\begin{align*}
 \delta_k(u,w)\leq \delta_k(u,v)+\delta_k(v,w).
\end{align*}
\end{lemma}
\begin{proof}
 The triangle inequality follows directly from the definition since for $u,v,w\in B$,
 each pair of chains of elements $\TT_1,\ldots,\TT_n$ connecting ${\rm mid}(v)$ and ${\rm mid}(w)$ and $\widetilde \TT_1,\ldots,\widetilde\TT_m$ connecting ${\rm mid}(w)$ and ${\rm mid}(u)$
 can be combined to a chain of length $n+m$ connecting ${\rm mid}(u)$ and ${\rm mid}(v)$.
 The estimate~\eqref{eq:dequiv} can be seen as follows:
 First, assume $k\leq n$ and  let $T_1,\ldots,T_r\in \widehat\TT_n$ denote the minimizer from the definition of $\delta_n(v,w)$ with $r=\delta_n(v,w)$.
 By replacing each $T_i$ with its father in $\widehat\TT_k$, we obviously obtain $\delta_k(v,w)\leq \delta_n(v,w)$.
 Second, assume $k>n$: Obviously, 
 the $T_i$ touch each other at most at corners (otherwise we could delete an element and shorten
 the sequence). Shape regularity ensures that two corners of $T_i\in\widehat\TT_n$ can be connected with less than $\mathcal{O}(C_{\rm base}^2C_{\rm mesh}^{k-n})$ elements of $\widehat\TT_k$.
 This shows that $\delta_k(v,w)\lesssim C_{\rm base}^2C_{\rm mesh}^{k-n}r=C_{\rm base}^2C_{\rm mesh}^{k-n}\delta_n(v,w)$, where the hidden constant depends only on shape regularity of $\TT_0$.
 This concludes the proof.
%
%
\end{proof}

\begin{lemma}\label{lem:d2metric}
 For sufficiently large $\beta>0$, the function $d_2$ is a metric.
\end{lemma}
\begin{proof}
 Symmetry of $d_2$ is obvious. Moreover, $d_2=0$ is equivalent to $v=w$. It remains to prove the triangle inequality. 
 To that end, we observe with Lemma~\ref{lem:dkmetric}, $u,v,w\in B$, and $a:=\min\{L(u),L(v)\}$, $b:=\min\{L(v),L(w)\}$, $c:=\min\{L(u),L(w)\}$ that
 \begin{align*}
  d_1(u,w)&=\delta_c(u,w)\\
  &\leq \delta_c(u,v) +\delta_c(v,w)\leq (C_{\rm mesh}C)^{\max\{c-a,0\}}d_1(u,v) + (C_{\rm mesh}C)^{\max\{c-b,0\}}d_1(v,w),
\end{align*}
where $C:=\max\{1,C_{\TT_0}C_{\rm base}^2\}$.
With Lemma~\ref{lem:log} and $CC_{\rm mesh}\geq 1$, we have (assuming $u\neq v$, $v\neq w$, $u\neq w$)
\begin{align*}
 & \log(\delta(u,w)+d_1(u,w))\\
 &\leq \log(\delta(u,v) +(C_{\rm mesh}C)^{\max\{c-a,0\}}\delta_1(u,v))+  \log(\delta(v,w) +(C_{\rm mesh}C)^{\max\{c-b,0\}}\delta_1(v,w))\\ 
 &\leq (\max\{c-a,0\} + \max\{c-b,0\})\log(C_{\rm mesh}C) \\
 &\qquad + \log(\delta(u,v) +\delta_1(u,v))+\log(\delta(v,w) +\delta_1(v,w)).
\end{align*}
Note that each of the conditions $c-a\leq 0$ or $c-b\leq 0$ implies that 
\begin{align}\label{eq:bool}
L(u)\leq \min\{L(v),L(w)\}\quad\text{or}\quad L(w)\leq \min\{L(u),L(v)\}.
\end{align} 
The condition~\eqref{eq:bool}, however, implies $c-a\leq 0$ and $c-b\leq 0$.
Hence, we have $c-a\leq 0$ if and only if $c-b\leq 0$ and thus 
\begin{align*}
\max\{c-a,0\} + \max\{c-b,0\}= \max\{2c-a-b,0\}.
\end{align*}
For $\beta\geq \log(C_{\rm mesh}C)$ and with $|L(u)-L(w)|=L(u)+L(w)-2c$, we therefore obtain
\begin{align*}
 d_2(u,w)&\leq \delta(u,v)+\delta(v,w) + \beta(\max\{L(u)+L(w)-a-b,|L(u)-L(w)|\}) \\
 &\qquad+ \log(\delta(u,v) +\delta_1(u,v))+\log(\delta(v,w) +\delta_1(v,w)).
\end{align*}
With $L(u)-a\leq |L(u)-L(v)|$ and $L(w)-b\leq |L(v)-L(w)|$ as well as $|L(u)-L(w)|\leq |L(u)-L(v)|+|L(v)-L(w)|$, this yields the triangle inequality $d_2(u,w)\leq d_2(u,v)+d_2(v,w)$ under the assumption $u\neq v$, $v\neq w$, $u\neq w$.
If this assumption is violated, the triangle inequality follows trivially, since one of the three terms is zero. This concludes the proof.
\end{proof}

The following result is an auxiliary lemma which estimates the number of basis functions in a given annulus. This used in Section~\ref{section:banded} below to show that the FEM/BEM coupling matrices
are close to banded ones.
\begin{lemma}\label{lem:annulus}
Given $w\in B$ and $n,b,r\in\N$, define the sets
 \begin{align*}
  R_1&:=\set{v\in B_\ell}{v|_{\Omega\setminus\Gamma}=0,\,r-b\leq \delta_n(v,w)\leq r},\\
  R_2&:=\set{v\in B_\ell}{v|_{\Omega\setminus\Gamma}\neq 0,\,r-b\leq \delta_n(v,w)\leq r}.
 \end{align*}
Then, there holds $\#R_d \leq C_{\rm geo}b (r+1)^{d-1} C_{\rm mesh}^{d\max\{\ell-n,0\}}$ for $d=1,2$, where $C_{\rm geo}$ depends only on $C_{\rm mesh}$ and $C_{\rm base}$. 
\end{lemma}
\begin{proof}
 First, consider $d=2$.
 Uniform shape regularity of $(\widehat\TT_\ell)_{\ell\in\N}$ implies that 
 \begin{align*}
 \delta_n(v,w)\simeq {\rm dist}({\rm mid}(v),{\rm mid}(w))C_{\rm mesh}^{n}+1.
 \end{align*}
 For this, and since $v\in B_\ell$, we conclude that the supports of functions in $R_2$ is contained in an annulus around the barycenter of ${\rm supp}(w)$ 
 with radii 
 \begin{align*}
 r_1&\simeq (r-b-1)C_{\rm mesh}^{-n}-C_{\rm base}C_{\rm mesh}^{-\ell},\\
 r_2&\simeq rC_{\rm mesh}^{-n}+C_{\rm base}C_{\rm mesh}^{-\ell}.
 \end{align*}
 Moreover, conditions~(ii)--(iv) on $B$ show that the supports of $v\in B_\ell$ cover $\Omega$ with finite overlap depending only on $C_{\rm base}$. This ensures that
 \begin{align*}
 \#R_2 &\lesssim \frac{\text{area of annulus}}{\text{minimal area of support}}\\
 &\lesssim (r_2^2-r_1^2)/C_{\rm mesh}^{-2\ell}\lesssim r(b+1)C_{\rm mesh}^{2\ell-2n}+r C_{\rm mesh}^{\ell-n}\\
 &\lesssim r(b+1)C_{\rm mesh}^{2\max\{\ell-n,0\}}.
\end{align*}
This concludes the proof for $d=2$. For $d=1$, we use that $|{\rm supp}(v)|\simeq C_{\rm mesh}^{-\ell}$ and that the area of the 1D annulus is just $2(r_2-r_1)$.
 \end{proof}

\begin{lemma}\label{lem:d3metric}
 For sufficiently large $\gamma>0$, the  function $d_3$ is a metric. In this case, $d_3$ satisfies~\eqref{eq:Mdist} (assuming that $B$ is bijectively identified with $\N$ and $d_3(i,j):=d_3(v_i,v_j)$).
\end{lemma}
\begin{proof}
 Symmetry of $d_3$ is obvious. Moreover, $d_3=0$ is equivalent to $v=w$ since $d_1(v,v)=1$. It remains to prove the triangle inequality.
 For $L(v)=L(w)=L(u)$, Lemma~\ref{lem:dkmetric} shows
 \begin{align*}
  d_3(u,w)\leq d_3(u,v)+d_3(v,w).
 \end{align*}
If $L(v)\neq L(u)=L(w)$, we use that according to Lemma~\ref{lem:dkmetric}, 
\begin{align*}
d_3(u,w)\leq \delta_0(u,w)(C_{\rm mesh}C)^{L(u)}\leq \#\widehat\TT_0 (C_{\rm mesh}C)^{L(u)}\leq {\gamma^{\max\{L(u),L(v)\}}}
\end{align*}
for sufficiently large $\gamma$ {$C=\max\{1,C_{\TT_0}C_{\rm base}^2\}$}.
If $L(u)\neq L(w)$, we have
\begin{align*}
 d_3(u,w)={\gamma^{\max\{L(u),L(w)\}}}\leq d_3(u,v)+d_3(v,w)
\end{align*}
by distinguishing cases. This proves the triangle inequality for $d_3$.

To see~\eqref{eq:Mdist} assume a one-to-one numbering of $B=\{v_1,v_2,\ldots\}$ and let $d_3(i,j):=d_3(v_i,v_j)$. Let $i\in\N$, $\eps>0$, and compute
\begin{align*}
 \sum_{j=1}^\infty \exp(-\eps d_3(i,j)) &= \sum_{v\in B}\exp(-\eps d_3(v_i,v)) \\
 &\lesssim  \sum_{v\in B_{L(v_i)}}\exp(-\eps\delta_{L(v_i)}(v_i,v)) +  \sum_{v\in B\setminus B_{L(v_i)}}\exp(-\eps\gamma^{{\max\{L(v_i),L(v)\}}}).
\end{align*}
The last term is bounded by
\begin{align*}
 \sum_{v\in B\setminus B_{L(v_i)}}\exp(-\eps\gamma^{{\max\{L(v_i),L(v)\}}})\leq\sum_{\ell\in\N}\sum_{v\in B_\ell} \exp(-\eps \gamma^\ell)
 \lesssim \sum_{\ell\in\N} C_{\rm mesh}^{2\ell} \exp(-\eps \gamma^\ell)<\infty,
\end{align*}
where we used that $\#B_\ell\simeq C_{\rm mesh}^{2\ell}$ and that $\exp(-\eps\gamma^\ell)\to 0$ faster than $C_{\rm mesh}^{-2\ell}\to 0$.
The other term is bounded by
\begin{align*}
  \sum_{v\in B_{L(v_i)}}\exp(-\eps\delta_{L(v_i)}(v_i,v))\leq\sum_{r=1}^\infty \sum_{v\in B_{L(v_i)}\atop \delta_{L(v_i)}(v_i,v)=r} \exp(-\eps r)\lesssim 
  \sum_{r=1}^\infty r \exp(-\eps r)<\infty,
\end{align*}
where we used that due to Lemma~\ref{lem:annulus}, the inner sum contains $\mathcal{O}(r)$ summands.
Since all the bounds do not depend on $i\in\N$, this concludes the proof.
\end{proof}

\section{Scott-Zhang projection on Steroids}\label{section:szb}
In the following, we develop two versions of the classical Scott-Zhang projection.
The first one (Lemma~\ref{lem:gensz}) is just designed slightly differently to obtain a moment condition on the residual.
The second one (Theorem~\ref{thm:szb}) is constructed such that projections on different levels commute. As an interesting side note,
this particularly implies that there exists an equivalent Hilbert norm for with these projections are orthogonal projections and thus self-adjoint.

We use the standard definitions of Sobolev spaces $H^k(\omega)$ and note that for $s>0$, $H^s(\omega)$ is defined via real interpolation.

\begin{definition}\label{def:s2}
	Given a triangulation $\TT$, define the hat functions $v_z\in \SS^1(\TT)$ associated with a certain node $z$ of $\TT$. 
	For an edge $E$ of $\TT$ with endpoints $z_1$ and $z_2$, define the edge bubble $v_E:= \alpha_E v_{z_1}v_{z_2}$ with $\alpha_E>0$ such that $\norm{v_E}{L^\infty(\Omega)}=1$.
	For an element $T\in\TT$ with nodes $z_1$, $z_2$, and $z_3$, define the element bubble $v_T:= \alpha_T v_{z_1}v_{z_2}v_{z_3}$ with $\alpha_T>0$ such that $\norm{v_T}{L^\infty(\Omega)}=1$.
	Let $\SS^{2+}_B(\TT)$ denote the set of all hat-, edge bubble-, and element bubble-functions defined on $\TT$ and let $\SS^{2+}(\TT)$ define the linear span of $\SS^{2+}_B(\TT)$.
	For a refinement $\widehat\TT$ of $\TT$, we denote by $\SS^{2+}_B(\widehat\TT\setminus\TT)$ all hat functions $v_z$ associated with new nodes $z\in\NN(\widehat\TT)\setminus\NN(\TT)$, 
	all edge bubble functions $v_E$ associated with new edges $E\in\EE(\widehat\TT)\setminus \EE(\TT)$, and all element bubble functions $v_T$ associated with new elements $T\in\widehat\TT\setminus\TT$.
\end{definition}

In the following, we define a particular basis of $\SS^{2+}(\TT)$ with a certain moment condition.
\begin{definition}
	Given a triangulation $\TT$, consider the following basis of $\SS^{2+}(\TT)$: Let $\BB_B(\TT)\subseteq  \SS^{2+}_B(\TT)$ 
	denote all the element bubble functions $v_T$ and let $\EE_B(\TT)\subseteq  \SS^{2+}_B(\TT)$ denote all the edge bubble functions $v_E$. Given $v_E\in \EE_B(\TT)$, define
	\begin{align*}
		v_{E,0}:= v_E-\sum_{T\in\TT\atop T\subseteq {\rm supp}(v_{E,0})} \alpha_{T,E}v_T
	\end{align*}
	with $\alpha_{T,E}\in\R$ such that $\int_T v_{E,0}\,dx=0$ for all $T\in\TT$.
	Given $v_z\in \SS^{2+}_B(\TT)\setminus(\BB_B(\TT)\cup\EE_B(\TT))$ (a hat function), define
	\begin{align*}
		v_{z,0}:= v_z-\sum_{T\in\TT\atop T\subseteq {\rm supp}(v_{z,0})} \alpha_{T,z}v_T-\sum_{E\in\EE(\TT)\atop E\subseteq {\rm supp}(v_{z,0})} \beta_{E,z}v_{E,0}
	\end{align*}
	with $\alpha_{T,z},\beta_{E,z}\in\R$
	such that $\int_T v_{z,0}\,dx=0$ for all $T\in\TT$ and $\int_E v_{z,0}\,dx=0$ for all $E\in\EE(\TT)$.
	Note that the number of terms as well as the magnitude of the coefficients $\alpha_{T,E},\alpha_{T,z},\beta_{E,z}$ in the above definitions 
	is bounded in terms of the shape regularity of $\TT$ and hence in terms of $\TT_0$.
	Then, with $\NN_{sz}(\TT):=\set{v_{z,0}}{v_z\in\SS^{2+}_B(\TT)\setminus (\BB_B(\TT)\cup\EE_B(\TT)) }$ and
	$\EE_{sz}(\TT):=\set{v_{E,0}}{v_E\in \EE_B(\TT)}$, the set
	$\SS_{sz}(\TT):=\NN_{sz}(\TT)\cup \EE_{sz}(\TT)\cup \BB_B(\TT)$ is a basis of $\SS^{2+}(\TT)$.
\end{definition}
To define the Scott-Zhang projection, we define the dual basis functions.
\begin{definition}\label{def:sz}	
	For each $T\in\TT$, let $w^\star_T\in\SS^{2+}(T)$ denote the dual basis functions of $w|_T$ for all $w\in \SS_{sz}(\TT)$ with $w|_T$ is non-zero, i.e., all $v\in \SS^{2+}(\TT)$ satisfy
	\begin{align*}
	 \int_T w^\star_T v\,dx=\begin{cases} 1 & \text{for }v|_T = w,\\
	                         0 &\text{else.}
	                        \end{cases}
	\end{align*}	
	Analogously, for all $E\in\EE(\TT)$, let $w^\star_E\in\SS^{2}(E)$ denote the dual basis functions of $w|_E$ for all $w\in \SS_{sz}(\TT)$ with $w|_E$ is non-zero, i.e., all
	$v\in \SS^{2+}(\TT)$ satisfy
	\begin{align*}
	 \int_E w^\star_E v\,dx=\begin{cases} 1 & \text{for }v|_E = w,\\
	                         0 &\text{else.}
	                        \end{cases}
	\end{align*}
	Moreover, for each $w\in \NN_{sz}(\TT)$ choose some $E_w\in\EE(\TT)$ with $w|_{E_w}\neq 0$ such that if $w|_\Gamma\neq 0$ also $E_v\subseteq \Gamma$. For $w\in \EE_{sz}(\TT)$ choose the one $E_w\in\EE(\TT)$ with $w|_{E_w}\neq 0$, 
	and for each $w\in \BB_B(\TT)$, let $T_w\in\TT$ denote the element on which $w$ is supported.
	With this, define the modified Scott-Zhang operator as
	\begin{align*}
		J_\TT v &:= \sum_{w\in \BB_B(\TT)}w\dual{w^\star_{T_w}}{v}_{T_w}+ \sum_{w\in \NN_{sz}(\TT)\cup \EE_{sz}(\TT)}w\dual{w^\star_{E_w}}{v}_{E_w}. 
	\end{align*}
	With $\SS_{sz}(T)$, $\BB_B(T)$, $\NN_B(T)$, and $\EE_B(T)$, we denote the respective subsets of functions which are non zero on $T$. Note that the cardinality of those sets is bounded in terms of $\TT_0$.
\end{definition}

\begin{definition}\label{def:sobslob}
 Define the Sobolev-Slobodeckij semi norm 
 \begin{align*}
 |v|_{H^s(\omega)}^2:= \int_\omega\int_\omega |v(x)-v(y)|/|x-y|^{2+2s}\,dx\,dy\quad\text{for } 0<s<1.
 \end{align*}
 For $s=\nu+r\in\R$ with $\nu\in\N$ and $s\in (0,1)$, define $|\cdot|_{H^s(\omega)}:=|\nabla^\nu(\cdot)|_{H^r(\Omega)}$, where $\nabla^\nu$ 
 denotes the tensor of all partial derivatives of order  $\nu$. 
 As shown {in~\cite{heu}}, $\norm{\cdot}{H^\nu(\omega)}+|\cdot|_{H^r(\omega)}$  is equivalent to the $H^s$-norm obtain  via (real) interpolation.
 The norm equivalence constants depend only on the shape of $\omega$.
\end{definition}

\begin{lemma}\label{lem:gensz}
	The Scott-Zhang operator $J_\TT$ from Definition~\ref{def:sz} is a projection which preserves homogeneous Dirichlet values, i.e., $v|_\Gamma=0$ implies $J_\TT v|_\Gamma=0$.
	There holds for all $1/2<s< 3/2$ and all  $v\in H^s(\Omega)$
	\begin{align}
		\norm{J_\TT v}{H^s(T)}&\leq C_{\rm sz} \norm{v}{H^s(\cup\omega(T,\TT))},\label{eq:szhs}\\
		\norm{J_\TT v}{H^s(\Omega)}&\leq C_{\rm sz} \norm{v}{H^s(\Omega)},\label{eq:szOhs}
	\end{align}
	as well as for all $v\in H^s(\Omega)$ and all $0\leq r\leq s$, $r<3/2$, $1/2<s\leq 2$
	\begin{align}
		\norm{(1-J_\TT) v}{H^r(T)}&\leq C_{\rm sz} {\rm diam}(T)^{s-r} |v|_{H^s(\cup\omega(T,\TT))},\label{eq:sza}\\
		\norm{(1-J_\TT) v}{H^r(\Omega)}&\leq C_{\rm sz} \norm{h_\TT^{s-r}\nabla^s v}{L^2(\Omega)}\quad\text{for }0\leq r\leq 1\text{ and }s\in\{1,2\}.\label{eq:szaO}
	\end{align}
	The constant $C_{\rm sz}>0$ depends only on the shape regularity of $\TT$, the fact that $\TT$ is generated from $\TT_0$ by newest vertex bisection, on a lower bound on $s>1/2$ and on an upper bound of $r<3/2$. The function $(J_\TT v)|_T$ depends   only on $v|_{\cup\omega(T,\TT)}$. Moreover, $\int_T v\,dx =0$ for some $T\in\TT$ implies $\int_{T}J_\TT v\,dx=0$ and $\int_E v\,dx =0$ for some $E\in\EE(\TT)$ implies $\int_E J_\TT v\,dx=0$. This particularly implies
	\begin{align}\label{eq:szzero}
		\int_E (1-J_\TT) v\,dx =0 = \int_T(1-J_\TT)v\,dx\quad\text{for all }E\in\EE(\TT)\text{ and all } T\in\TT.
	\end{align}
\end{lemma}

\begin{proof}
	For the projection property, let $v\in \SS^{2+}(\TT)$ with $v=\sum_{w\in\SS_{sz}(\TT)}\alpha_w w$ for some coefficients $\alpha_w\in\R$. Then, there holds
	\begin{align*}
		J_\TT v &= \sum_{w\in \BB_B(\TT)}\sum_{\widetilde w\in \SS_{sz}(\TT)}\alpha_{\widetilde w}w\dual{w^\star_{T_w}}{\widetilde w}_{T_w} 
		+  \sum_{w\in \NN_{sz}(\TT)\cup \EE_{sz}(\TT)}\sum_{\widetilde w\in \SS_{sz}(\TT)}\alpha_{\widetilde w}w\dual{w^\star_{E_w}}{\widetilde w}_{E_w}\\
		& = 
		\sum_{w\in \SS_{sz}(\TT)}\alpha_{w}w = v.
	\end{align*}
	Obviously, $(J_\TT v)|_T$ depends only on $v|_{T_w}$ and $v|_{E_w}$ for all $w\in\SS_{sz}(\TT)$ such that $w|_T$ is non zero. 
	Since all the basis functions have support within one patch of $\TT$, this shows that $(J_\TT v)|_T$ depends only on $v|_{\cup\omega(T,\TT)}$. 
	Particularly, there holds $J_\TT v =v$ on $T$, if $v\in \SS^{2+}(\omega(T,\TT))$. Since $E_w\subseteq \Gamma$ for all $w\in \NN_{sz}(\TT)\cup \EE_{sz}(\TT)$ with $w|_\Gamma\neq 0$,
	we have $J_\TT v|_\Gamma =0$ if $v|_\Gamma=0$.
	
	To see the moment condition~\eqref{eq:szzero}, note that the dual basis function $v_T^\star$ of an element bubble $v_T\in\BB_B(T)$ is uniquely determined.
	The function $1$ satisfies $\dual{w}{1}_{T}=0$ for all $w\in \SS_{sz}(\TT)\setminus \BB_B(\TT)$ by definition. 
	Moreover, there clearly holds $\dual{v_T}{1}_T\neq 0$. Thus, $v_T^\star\in \SS^{2+}(T)$ is a constant function.
	Hence, by definition of $J_\TT$, we have
	\begin{align*}
		\int_T v\,dx=0\quad\implies\quad (J_\TT v)|_T &:= \sum_{w\in \NN_{sz}(T)\cup\EE_{sz}(T)}w\dual{w^\star_{T_w}}{v}_{E_w}. 
	\end{align*}
	Since all $w\in \NN_{sz}(\TT)\cup\EE_{sz}(\TT)$ have element wise zero integral mean, we obtain that $\int_T v\,dx =0$ implies $\int_T J_\TT v\,dx =0$. Analogously, we find that the dual basis function of $w|_{E_w}$ for $w\in \EE_{sz}(\TT)$ must be constant. Thus, if $\int_E v\,dx =0$, there holds
	\begin{align*}
		(J_\TT v)|_E =\sum_{w\in \NN_{sz}(\TT)}w|_E\dual{w^\star_{T_w}}{v}_{E_w},
	\end{align*}
	where we used that exactly one $w\in\EE_{sz}(\TT)$ satisfies $E_w=E$ and the fact that $w\in\BB_B(\TT)$ satisfy $w|_E=0$. Since all $w\in \NN_{sz}(\TT)$ have zero integral mean on $E$, we obtain $\int_E J_\TT v\,dx =0$.
	This particularly implies that for $v\in H^1(\Omega)$ and $v_0\in\SS^{2+}(\TT)$ such that $v-v_0$ has zero integral mean on all edges and elements, we obtain
	\begin{align*}
		\int_\omega (1-J_\TT)v\,dx = \int_\omega (1-J_\TT)(v-v_0)\,dx =0\quad\text{for all } \omega\in \TT\cup\EE(\TT)
	\end{align*}
	and hence prove~\eqref{eq:szzero}.
	
	It remains to prove the estimates~\eqref{eq:szhs}--\eqref{eq:szaO}. To that end, note that for each $T\in\TT$, there exists one of finitely many reference patches $\widehat\omega$
	such that $\omega(T,\TT) = \phi(\widehat\omega)$ and $T= \phi(\widehat T)$ for some affine function $\phi\colon \widehat\omega\to \omega(T,\TT)$ with Jacobian $D\phi = \alpha I$ for some $\alpha>0$ and identity matrix $I$.
	Define $\widehat J v= (J_\TT (v\circ\phi^{-1}))\circ \phi$. There holds for $w\in \SS_{sz}(\TT)$
	\begin{align*}
		\dual{w^\star_{T_w}}{v\circ\phi^{-1}}_{T_w} = \dual{|T_w|w^\star_{T_w}\circ\phi}{v}_{\phi^{-1}(T_w)}\quad\text{for all }w\in\BB_B(T),\\
		\dual{w^\star_{E_w}}{v\circ\phi^{-1}}_{E_w} = \dual{|E_w|w^\star_{E_w}\circ\phi}{v}_{\phi^{-1}(E_w)}\quad\text{for all }w\in\NN_{sz}(T)\cup\EE_{sz}(T).
	\end{align*}
	Since composition $\widehat{(\cdot)}\colon w\mapsto w\circ \phi$ maps $\SS_{sz}(T)$ to $\SS_{sz}(\widehat T)$ in a one-to-one manner, it is obvious from the above that 
	$|T_w|w^\star_{T_w}\circ\phi= \widehat{w}^\star_{\widehat T_w}$ with $\widehat T_w = \phi^{-1}(T_w)\in \widehat\omega$
	and $|E_w|w^\star_{E_w}\circ\phi= \widehat{w}^\star_{\widehat E_w}$ with $\widehat E_w = \phi^{-1}(E_w)\in \EE(\widehat\omega)$. Thus, $\widehat J$ is one of finitely many (analogously defined) Scott-Zhang
	operators.
	
	A standard scaling argument shows for $T\in\TT$ and $s>1/2$ 
	\begin{align}\label{eq:semicont}
	\begin{split}
		|J_\TT v|_{H^s(T)}^2&\simeq |T|{\rm diam}(T)^{-2s}|\widehat J (v\circ\phi)|_{H^s(\widehat T)}^2\\
		&\lesssim |T|{\rm diam}(T)^{-2s}|v\circ\phi|_{H^s(\cup\widehat\omega)}^2\simeq |v|_{H^s(\cup\omega(T,\TT))}^2,
	\end{split}
	\end{align}
	where we used that $\widehat J$ has finite dimensional range as well as that $|v|_{H^s(\cup\omega(T,\TT))}=0$ implies that $v$ is constant and hence 
	$J_\TT v = v$ on $T$, which in turn implies $|\widehat J (v\circ\phi)|_{H^s(\widehat T)}=0$.
	
	To see~\eqref{eq:sza} recall that $(1-J_\TT)v$ has element wise zero integral mean.
	If $r\leq 1$ consider $t:=\max\{r,\min\{s,1\}\}>1/2$. We have with~\eqref{eq:semicont} and a Poincar\'e inequality
	\begin{align*}
		\norm{&(1-J_\TT)v}{H^r(T)}\\
		&\lesssim {\rm diam}(T)^{t-r}|(1-J_\TT)v)|_{H^t(T)}=  {\rm diam}(T)^{t-r}\inf_{v_0\in\SS^{2+}(\TT)}|(1-J_\TT)(v-v_0)|_{H^t(T)}\\
		&\lesssim 
		{\rm diam}(T)^{t-r}\inf_{v_0\in\SS^{2+}(\TT)}|v-v_0|_{H^t(\cup\omega(T,\TT))}\lesssim {\rm diam}(T)^{s-r}|v|_{H^s(\cup\omega(T,\TT))}.
	\end{align*}
	If $r>1$, we obtain again with~\eqref{eq:semicont}
	\begin{align*}
		\norm{&(1-J_\TT)v}{H^r(T)}\\
		&\lesssim |(1-J_\TT)v)|_{H^1(T)}+|(1-J_\TT)v)|_{H^r(T)}\\
		&\lesssim 
		\inf_{v_0\in\SS^{2+}(\TT)}\big(|v-v_0|_{H^1(\cup\omega(T,\TT))}+|v-v_0|_{H^r(\cup\omega(T,\TT))}\big)\\
		&\leq \inf_{v_0\in\SS^{2+}(\TT)}\norm{v-v_0}{H^r(\cup\omega(T,\TT))}\lesssim {\rm diam}(T)^{s-r}|v|_{H^s(\cup\omega(T,\TT))}.
	\end{align*}
	This proves~\eqref{eq:sza}.
	
	By choosing $r=s$,~\eqref{eq:sza} also proves~\eqref{eq:szhs}. For $s=1$, we obtain~\eqref{eq:szOhs} by summation.
	For $1/2<s<1$, we use the equivalence of interpolation norms with Sobolev-Slobodeckij norms as discussed in Definition~\ref{def:sobslob}.
	Then, the result~\cite[Lemma~3.2]{faermann2} shows together with~\eqref{eq:sza}
	\begin{align*}
	 \norm{v-J_\TT v}{H^s(\Omega)}^2&\lesssim \sum_{T\in\TT}\Big(|v-J_\TT v|_{H^s(\cup\omega(T,\TT))}^2 + (1+{\rm diam}(T)^{-2s})\norm{v-J_\TT v}{L^2(T)}^2\Big)\\
	 &\lesssim \sum_{T\in\TT}|v|_{H^s(\cup\omega(T,\TT))}^2\leq |v|_{H^s(\Omega)}^2.
	\end{align*}
	This implies immediately $\norm{J_\TT v}{H^s(\Omega)}\lesssim \norm{v}{H^s(\Omega)}$ and thus~\eqref{eq:szOhs} for all $1/2<s<1$. For $1<s<3/2$, a similar argument shows
	\begin{align*}
	 \norm{v-J_\TT v}{H^s(\Omega)}^2&\lesssim \sum_{T\in\TT}\Big(|\nabla(v-J_\TT v)|_{H^{s-1}(\cup\omega(T,\TT))}^2 + (1+{\rm diam}(T)^{-2s})\norm{v-J_\TT v}{H^1(T)}^2\Big)\\
	 &\lesssim \sum_{T\in\TT}|v|_{H^s(\cup\omega(T,\TT))}^2\leq |v|_{H^s(\Omega)}^2
	\end{align*}
	and thus concludes the proof of~\eqref{eq:szOhs}.
	
	 Moreover, from~\eqref{eq:sza}, we conclude~\eqref{eq:szaO} for $r=0$ and $r=1$ by summation.
	Hence, interpolation proves~\eqref{eq:szaO} for all $0\leq r \leq 1$.

	This concludes the proof.
\end{proof}

The following theorem establishes the extended Scott-Zhang projections such that operators on different levels commute. The author is confident, that this definition is not
restricted to uniform refinements as shown below and a more careful analysis would work for general triangulations $\TT\in\T$. This could be useful for other applications, as for example
the Aubin-Nitsche trick relies on the fact that $L^2$-projections on different levels commute.
\begin{theorem}\label{thm:szb}
	Recall $\widehat\TT_\ell$ and $C_{\rm mesh}$ from Definition~\ref{def:widehat}.
	With $J_\ell:= J_{\widehat\TT_\ell}$, define
	\begin{align*}
		S_\ell v := \lim_{N\to\infty} (J_\ell J_{\ell+1}\ldots J_{\ell+N})v\in S^{2+}(\widehat \TT_\ell)
	\end{align*}
	for all $v\in H^1(\Omega)$.
	Then, the operator $S_\ell\colon H^1(\Omega)\to \SS^{2+}(\widehat \TT_\ell)$ is well-defined and satisfies for all $1/2<\sigma< 3/2$, $\mu>1/2$, and $0\leq \nu\leq 1$, and $\nu+1/4\leq \mu\leq 2$
	\begin{align}
		\norm{(1-S_\ell) v}{H^\nu(\Omega)}&\leq C_{\rm S} C_{\rm mesh}^{-\ell(\mu-\nu)}\norm{v}{H^\mu(\Omega)}\quad\text{for all }v\in H^\mu(\Omega),\label{eq:sapprox}\\
		\norm{S_\ell v}{H^\sigma(\Omega)}&\leq C_{\rm S} \norm{ v}{H^\sigma(\Omega)}\quad\text{for all }v\in H^\sigma(\Omega). \label{eq:scont}
	\end{align}
	 Moreover, there holds $S_\ell S_k=S_{\min\{\ell,k\}}$ for all $\ell,k\in\N_0$ as well as
	 \begin{align}\label{eq:Sid}
	  S_\ell v = J_\ell J_{\ell+1}\cdots J_{\ell+k} v\quad\text{for all }v\in \SS^{2+}(\widehat\TT_{\ell+k})
	 \end{align}
	for all $\ell,k\in\N$.
	$S_\ell$ preserves locality in the sense that $(S_\ell u)|_T$
	for some $T\in\TT_\ell$ depends only on $u$ on $\bigcup\omega^r(T,\TT_\ell)$ for some $r\in\N$. 
	Finally, $v|_\Gamma =0$ implies $(S_\ell v)|_\Gamma =0$.
	The constant $C_{\rm S}>0$ depends only on $C_{\rm sz}$ and $C_{\rm mesh}$, whereas the constant $r$ depends only on $\TT_0$.
\end{theorem}
\begin{proof}
	From~\eqref{eq:szaO}, we obtain for $r\leq s$ and $0\leq r\leq 1$ and $s\in\{1,2\}$
	\begin{align}\label{eq:szaOO}
		\norm{(1-J_\ell)v}{H^r(\Omega)}\lesssim C_{\rm mesh}^{s-r}\norm{v}{H^{s}(\Omega)}\quad\text{for all }v\in H^s(\Omega).
	\end{align}
	Moreover, from~\eqref{eq:szOhs}, we even get $\norm{(1-J_\ell)v}{H^r(\Omega)}\lesssim \norm{v}{H^{r}(\Omega)}$ for all $1/2<r<3/2$.
	Interpolation arguments prove that~\eqref{eq:szaOO} holds for all $s$ with $r\leq s \leq 2$.
	Let $u\in H^{\mu}(\Omega)$. Since $J_\ell\colon H^\nu(\Omega)\to H^\nu(\Omega)$ is continuous for $1/2< \nu<3/2$ (see~\eqref{eq:szOhs}), 
	there holds for $1/2< \nu\leq 1$ and $\nu+1/4\leq \mu \leq 2$ with $\mu>1/2$ and by use of~\eqref{eq:szaOO} that
	\begin{align}\label{eq:sz1}
		\begin{split}
			\norm{(1-&(J_\ell J_{\ell+1}\ldots J_{\ell+N}))u}{H^{\nu}(\Omega)}\\
			&\leq \norm{(1-J_{\ell})u}{H^{\nu}(\Omega)}+\sum_{k=0}^{N-1}\norm{( (J_{\ell}\cdots J_{\ell+k})-(J_{\ell}\cdots J_{\ell+k+1}))u}{H^{\nu}(\Omega)}\\
			&\leq \norm{(1-J_{\ell})u}{H^{\nu}(\Omega)}+\sum_{k=0}^{N-1}\norm{(J_{\ell}\cdots J_{\ell+k})}{H^{\nu}(\Omega)\to H^{\nu}(\Omega)}\norm{(1-J_{\ell+k+1})u}{H^{\nu}(\Omega)}\\
			&\lesssim \sum_{k=0}^{N}C_{\rm sz}^{k}C_{\rm mesh}^{-(\mu-\nu)(\ell+k)}\norm{u}{H^\mu(\Omega)}\lesssim C_{\rm mesh}^{-(\mu-\nu)\ell}\norm{u}{H^\mu(\Omega)},
		\end{split}
	\end{align}
	where we used $C_{\rm mesh}^{(\mu-\nu)}\geq C_{\rm mesh}^{1/4}>C_{\rm sz}$ from Definition~\ref{def:widehat}.
	Hence,~\eqref{eq:sz1} for $ (\nu_1,\mu_1)= (3/4,1)$ and $(\nu_2,\mu_2)=(1,5/4)$ implies for $M\leq N$
	\begin{align*}
		\norm{((J_\ell J_{\ell+1}\ldots J_{\ell+M})&-(J_\ell J_{\ell+1}\ldots J_{\ell+N}))u}{H^{1}(\Omega)}\\
		&=\norm{(J_\ell\ldots J_{\ell+M})((1-(J_{\ell+M+1} J_{\ell+1}\ldots J_{\ell+N}))u}{H^{1}(\Omega)}\\
		&\leq \norm{((1-(J_{\ell+M+1} J_{\ell+1}\ldots J_{\ell+N}))u}{H^{1}(\Omega)}\\
		&\qquad+\norm{(1-(J_\ell\ldots J_{\ell+M}))((1-(J_{\ell+M+1} J_{\ell+1}\ldots J_{\ell+N}))u}{H^{1}(\Omega)}\\
		&\lesssim \norm{((1-(J_{\ell+M+1} J_{\ell+1}\ldots J_{\ell+N}))u}{H^{5/4}(\Omega)}\\
		&\lesssim C_{\rm mesh}^{-(\ell+M)/4}\norm{u}{H^{5/4}(\Omega)}.
	\end{align*}
	This shows, that $(J_\ell J_{\ell+1}\ldots J_{\ell+N})u$ is a Cauchy-sequence in $\SS^{2+}{\widehat\TT_\ell}$ with respect to the $H^{3/4}(\Omega)$-norm as $N\to\infty$.
	Thus, for $u\in H^{5/4}(\Omega)$, the limit $S_\ell u\in \SS^{2+}(\widehat\TT_\ell)$ exists.  
	
	We prove $S_k S_\ell u = S_k u$ for all $k\leq \ell$ and all $u\in H^{5/4}(\Omega)$. Note that $J_n S_\ell = S_\ell$ for all $n\geq \ell$ 
	due to the projection property of $J_n$. Since $(J_k\cdots J_{k+N})v$ converges in $H^{3/4}(\Omega)$ for all $v\in H^{5/4}(\Omega)$
	and $\SS^{2+}(\widehat\TT_\ell)\subseteq H^{5/4}(\Omega)$, we have
	\begin{align*}
		\norm{(S_k S_\ell-S_k)u}{H^{3/4}(\Omega)} &=\lim_{N\to\infty}\norm{((J_kJ_{k+1}\ldots J_{k+N}) S_\ell-S_k)u}{H^{3/4}(\Omega)}\\ 
		&=\norm{((J_kJ_{k+1}\ldots J_{\ell-1}) S_\ell-S_k)u}{H^{3/4}(\Omega)}\\
		&= \lim_{N\to\infty}\norm{((J_kJ_{k+1}\ldots J_{\ell+N}) -(J_kJ_{k+1}\ldots J_{\ell+N}))u}{H^{3/4}(\Omega)}=0.
	\end{align*}
	This shows $S_k S_\ell u = S_k u$ for $k\leq \ell$ and all $u\in H^{5/4}(\Omega)$. Particularly, we have $S_\ell S_\ell u=S_\ell u$.
	For $u\in \SS^{2+}(\widehat\TT_\ell)\subseteq H^{5/4}(\Omega)$, we obtain
	\begin{align*}
	 \norm{u-S_\ell u}{H^{3/4}(\Omega)} = \lim_{N\to\infty}\norm{u-(J_\ell J_{\ell+1}\ldots J_{\ell+N})u}{H^{3/4}(\Omega)}=0
	\end{align*}
	and hence $S_\ell (\SS^{2+}(\widehat\TT_\ell))=\SS^{2+}(\widehat\TT_\ell)$. This particularly implies $S_\ell S_k u=S_k u$ by use of
	the nestedness of the spaces. Analogously, we see~\eqref{eq:Sid}. By continuity of the trace operator in $H^{3/4}$ and with Lemma~\ref{lem:gensz}, we obtain from the above that $u\in H^{5/4}(\Omega)$ with $u|_\Gamma =0$ 
	implies $(S_\ell u)|_\Gamma =0$.	
	
	For the proof of~\eqref{eq:sapprox}, we derive from~\eqref{eq:sz1} with $1/2< \nu\leq 1 $ and $\nu+1/4\leq \mu\leq 2$ with $\mu>1/2$ for all $u\in H^2(\Omega)$
	\begin{align}\label{eq:sapprox1}
	\begin{split} 
		\norm{(1-S_\ell) u}{H^\nu(\Omega)}&=\lim_{N\to\infty}\norm{(1-(J_\ell J_{\ell+1}\ldots J_{\ell+N}))u}{H^\nu(\Omega)}\\
		&\lesssim C^{-\ell(\mu-\nu)}\norm{u}{H^\mu(\Omega)}.
		\end{split}
	\end{align}
	Density proves the statement for all $u\in H^{\mu}(\Omega)$ and particularly defines $S_\ell \colon H^\mu(\Omega) \to \SS^{2+}(\widehat\TT_\ell)$
	for all $\mu>3/4$ by continuous extension. The remaining case $0\leq \nu\leq 1/2$ and $\nu+1/4\leq \mu$ of~\eqref{eq:sapprox} can be proved as follows: First, note that
	\begin{align*}
	 (1-S_\ell)u= \lim_{N\to \infty} \Big( (1-J_{\ell+N})u + (1- J_{\ell+N-1})J_{\ell+N}u +\ldots +(1- J_\ell)(J_{\ell+1}\cdots J_{\ell+N})u\Big)
	\end{align*}
	in $H^{3/4}(\Omega)$. Moreover,~\eqref{eq:szzero} shows that $(1-J_k)u$ has element wise integral mean zero on $\widehat\TT_k$. Together, this shows
	that $(1-S_\ell)u$ has element wise integral mean zero on $\widehat\TT_\ell$. This, and a Poincar\'e inequality show
	\begin{align}\label{eq:sapprox2}
	 \norm{(1-S_\ell) u}{H^\nu(\Omega)}&\lesssim C^{-\ell(\min\{\mu,1\}-\nu)}\norm{u}{H^{\min\{\mu,1\}}(\Omega)}.
	\end{align}
	The combination of~\eqref{eq:sapprox1}--\eqref{eq:sapprox2} with $(1-S_\ell)=(1-S_\ell)(1-S_\ell)$ shows~\eqref{eq:sapprox}.
	
	This, and inverse estimates imply immediately for $1/2<\mu<3/2$
	\begin{align*}
		\norm{S_\ell u}{H^\mu(\Omega)}&\lesssim \norm{(S_\ell -J_\ell) u}{H^\mu(\Omega)}+\norm{ u}{H^\mu(\Omega)}
		\\
		&\lesssim C_{\rm mesh}^{\ell\mu}\norm{(S_\ell -J_\ell) u}{L^2(\Omega)}+\norm{ u}{H^\mu(\Omega)}\\
		&\lesssim C_{\rm mesh}^{\ell\mu}\big(\norm{(1-S_\ell) u}{L^2(\Omega)}+\norm{(1-J_\ell) u}{L^2(\Omega)}\big)+\norm{ u}{H^\mu(\Omega)}\lesssim \norm{ u}{H^\mu(\Omega)},
	\end{align*}
	where we used~\eqref{eq:szaOO}.
	The above estimate allows to continuously extend $S_\ell\colon H^\mu(\Omega)\to \SS^{2+}(\widehat\TT_\ell)\subseteq H^\mu(\Omega)$. The properties proved above for $u\in H^{5/4}(\Omega)$ follow
	for general $u\in H^\mu(\Omega)$ by density arguments.
	
	Finally, by the locality properties of $J_\ell$, there holds that given $T\in\TT_\ell$, $((J_\ell\ldots J_{\ell+N})u)|_T$ depends only on $u$ in the domain
	\begin{align*}
		\omega:=\bigcup\omega(\omega(\ldots\omega(\omega(T,\TT_\ell),\TT_{\ell+1}),\ldots,\TT_{\ell+N-1}),\TT_{\ell+N}).
	\end{align*}
	Due to the geometrically decreasing element size, it follows that $\omega$ is contained in some generalized patch $\bigcup\omega^r(T,\TT_\ell)$ for some $r\in\N$ which depends only on $C_{\rm sz}$.
\end{proof}

\section{Riesz bases}\label{section:basis}
This section constructs suitable Riesz bases of $H^1(\Omega)$, $H^{1/2}(\partial\Omega)$, and $H^{-1/2}(\partial\Omega)$ for Section~\ref{section:apps}.
To that end, we first have to prove that $(H^{3/4}(\Omega),H^1(\Omega),H^{5/4}(\Omega))$ form a Gelfand triple with $H^1(\Omega)$ as its pivot space.

\begin{lemma}\label{lem:ellreg}
Let $v\in H^1(\Omega)$ and define
\begin{align*}
 H^1(v,w):=\dual{\nabla v}{\nabla w}_\Omega +\dual{v}{w}_\Omega\quad\text{for all }v,w\in H^1(\Omega).
\end{align*}
Then, for $0\leq s<1/2$, there holds
\begin{align*}
 \norm{H^1(v,\cdot)}{\widetilde H^{-1-s}(\Omega)}\geq C_s \norm{v}{H^{1-s}(\Omega)}
\end{align*}
for all $v\in H^{1-s}(\Omega)$,
where $C_s>0$ depends only on $\Omega$ and $s$.
\end{lemma}
\begin{proof}
Let $0\leq s<1/2$. 
Inspection of the proof of~\cite[Theorem~4]{ellreg2} shows that for all $L\in \widetilde H^{-1+s}(\Omega)$, the unique solution $w\in H^1(\Omega)$ of
\begin{align*}
 H^1(v,w)=\dual{L}{v}_\Omega\quad\text{for all }v\in H^1(\Omega)
\end{align*}
satisfies $w\in H^{1+s}(\Omega)$ such that
\begin{align*}
	\norm{w}{H^{1+s}(\Omega)}\lesssim \norm{L}{\widetilde H^{-1+s}(\Omega)}.
\end{align*}
Given $v\in H^{1}(\Omega)$ define $L\in \widetilde H^{-1+s}(\Omega)$ such that $\norm{L}{\widetilde H^{-1+s}(\Omega)}=1$ and
	$L(v)=\norm{v}{H^{1-s}(\Omega)}$. The above shows
\begin{align*}
	H^1(v,w)=\dual{L}{v}_\Omega = \norm{v}{H^{1-s}(\Omega)}\gtrsim \norm{v}{H^{1-s}(\Omega)}\norm{w}{H^{1+s}(\Omega)}.
\end{align*}
Hence, we obtain $\norm{H^1(v,\cdot)}{\widetilde H^{-1-s}(\Omega)}\gtrsim \norm{v}{H^{1-s}(\Omega)}$ for all $v\in H^1(\Omega)$.
Density concludes the proof.
\end{proof}

\begin{lemma}\label{lem:gelfand}
For $0<s<1/2$, the interpolation spaces $H^{1-s}(\Omega)$ and $H^{1+s}(\Omega)$ form a Gelfand triple in the sense $H^{1+s}(\Omega)\subset H^1(\Omega)\subset H^{1-s}(\Omega)$.
\end{lemma}
\begin{proof}
We need to prove that $H^{1-s}(\Omega)$ is the dual space of $H^{1+s}(\Omega)$ with respect to the $H^1$-scalar product (and vice versa).

Obviously, $\nabla\colon H^s(\Omega)\to H^{{s-1}}(\Omega)$ for $1\leq s\leq 2$ is a bounded operator. We prove that $\nabla\colon L^2(\Omega)\to \widetilde H^{-1}(\Omega)$ is also bounded. 
To see this, consider a function $v \in L^2(\Omega)\cap H^1_0(\Omega)$
and
\begin{align*}
\norm{\nabla v}{\widetilde H^{-1}(\Omega)}&:= \sup_{w\in (H^1(\Omega))^2}\frac{\dual{\nabla v}{w}_\Omega}{\norm{w}{H^1(\Omega)}}=
\sup_{w\in (H^1(\Omega))^2}\frac{-\dual{v}{{\rm div} (w)}_\Omega}{\norm{w}{H^1(\Omega)}}\leq
\norm{v}{L^2(\Omega)}.
\end{align*}
Density of $H^1_0(\Omega)$ in $L^2(\Omega)$ and continuous extension prove boundedness $\nabla\colon L^2(\Omega)\to \widetilde H^{-1}(\Omega)$. Real interpolation concludes that $\nabla\colon H^s(\Omega)\to \widetilde H^{s-1}(\Omega)$ is bounded for all $0\leq s \leq 2 $.

Recall the $H^1$-scalar product $H^1(\cdot,\cdot)$ from Lemma~\ref{lem:ellreg}. 
The above together with duality of $H^s(\Omega)^\star = \widetilde H^{-s}(\Omega)$
shows that $H^1(\cdot,\cdot)\colon H^{1-s}(\Omega)\times H^{1+s}(\Omega)\to \R$ is bounded.
 Let $H^{s,\star}(\Omega)$ denote the dual space of $H^s(\Omega)$ with respect to $H^1(\cdot,\cdot)$. Then, the boundedness of $H^1(\cdot,\cdot)$ shows immediately  
 \begin{align*}
 H^{1+s}(\Omega)\subseteq H^{1-s,\star}(\Omega)\quad\text{and}\quad H^{1-s}(\Omega)\subseteq H^{1+s,\star}(\Omega).
 \end{align*}
 To obtain the remaining inclusions, let $v\in H^{1-s}(\Omega)$. Lemma~\ref{lem:ellreg} proves
 \begin{align*}
 \norm{v}{H^{1+s,\star}(\Omega)}=\norm{H^1(v,\cdot)}{\widetilde H^{-1-s}(\Omega)}\gtrsim \norm{v}{H^{1-s}(\Omega)}.
 \end{align*}
This confirms $H^{1+s,\star}(\Omega)= H^{1-s}(\Omega)$.
A duality argument shows $H^{1+s}(\Omega)= H^{1-s,\star}(\Omega)$ and thus concludes the proof.
\end{proof}

{
\begin{lemma}\label{lem:graded0}
 Let $\TT\in\T_{\rm grad}$ and $\ell\in\N$. Let $ T\in \TT\setminus \widehat \TT_\ell$ such that $\TT|_{T}$ is a strict local refinement of $\widehat \TT_\ell$.
 Then, $\TT|_{\omega^{D_{\rm grad}}(T,\widehat\TT_\ell)}$ is a local refinement of $\widehat\TT_\ell$.
\end{lemma}
\begin{proof}
Let $L\in\N$ denote the level of the elements in $\widehat\TT_\ell$. By assumption, there holds ${\rm level}(T)>L$. Assume there exists $T^\prime\in \TT|_{\omega^{D_{\rm grad}}(T,\widehat\TT_\ell)}$
with ${\rm level}(T^\prime)<L$. Assumption~\eqref{eq:graded} implies ${\rm level}(T^{\prime\prime})\leq L$ for all $T^{\prime\prime}\in \omega^{D_{\rm grad}}(T^\prime,\TT)$.
This shows that there holds 
\begin{align*}
 T\in \bigcup \omega^{D_{\rm grad}}(T^\prime,\widehat\TT_\ell)\subseteq \bigcup\omega^{D_{\rm grad}}(T^\prime,\TT).
\end{align*}
With~\eqref{eq:graded}, this implies ${\rm level}(T)\leq  L$, which contradicts the assumption that $\TT|_{T}$ is a strict local refinement of $\widehat \TT_\ell$.
\end{proof}}

The following theorem establishes the Riesz basis.
\begin{theorem}\label{thm:stableBbase}
With the spaces from Definition~\ref{def:s2}, define 
\begin{align*}
 B_0^1:=\set{\frac{v_0}{\norm{v_0}{H^1(\Omega)}}}{v_0\in\SS^{2+}_B(\TT_0)}
\end{align*}
and for $\ell\geq 1$
\begin{align*}
B_\ell^1:=\set{\frac{(1-J_{\ell-1})v_0}{\norm{(1-J_{\ell-1})v_0}{H^1(\Omega)}}}{v_0\in\Big(\bigcup_{k\in\N}\SS^{2+}_B(\TT_k\setminus\TT_{k-1})\Big)\cap (\SS^{2+}(\widehat\TT_\ell)\setminus\SS^{2+}(\widehat\TT_{\ell-1}))}.
\end{align*}
Moreover, define $B_\ell^{1/2}:=\set{v|_\Gamma}{v\in B^1_\ell}\setminus\{0\}$ for all $\ell\geq 0$ as well as $B_\ell^{-1/2}:=\set{v^\prime}{v\in B_\ell^{1/2}}$ for all $\ell\geq 1$ with $B_0^{-1/2}= \set{v^\prime}{v\in B_0^{1/2}}\cup \{1\}$. 
Define $B^s:=\bigcup_{\ell\in\N}B_\ell^s$ for all $s\in\{1,1/2,-1/2\}$.  Then, $B^1$, $B^{1/2}$, and $B^{-1/2}$ form Riesz bases of ${\overline{\bigcup_{\ell\in\N}\SS^{2+}(\TT_\ell)}\subseteq }H^1(\Omega)$,
${\overline{\bigcup_{\ell\in\N}\SS^{2}(\TT_\ell\cap\Gamma)}\subseteq }H^{1/2}(\Gamma)$, and ${\overline{\bigcup_{\ell\in\N}\PP^{1}(\TT_\ell\cap\Gamma)}\subseteq }H^{-1/2}(\Gamma)$ respectively. There holds
\begin{align}\label{eq:rieszB}
\begin{split}
\norm{\sum_{v\in B^1} \alpha_v v}{H^{1}(\Gamma)}&\simeq \Big(\sum_{v\in B^1} \alpha_v^2\Big)^{1/2},\\
\norm{\sum_{v\in B^{1/2}} \alpha_v v}{H^{1/2}(\Gamma)}&\simeq \Big(\sum_{v\in B^{1/2}} \alpha_v^2\Big)^{1/2},\\
\norm{\sum_{v\in B^{-1/2}} \alpha_v v}{H^{-1/2}(\Omega)}&\simeq \Big(\sum_{v\in B^{-1/2}} \alpha_v^2\Big)^{1/2}.
\end{split}
\end{align}
Moreover, ${\rm diam}({\rm supp}(v))\simeq C_{\rm mesh}^{-\ell} $ for all $v\in B_\ell^1\cup B_\ell^{1/2}\cup B_\ell^{-1/2}$ and there holds  
\begin{subequations}\label{eq:scaling}
\begin{align}\label{eq:scaling1}
\norm{v}{H^{s}({\rm supp}(v))}&\simeq C_{\rm mesh}^{-s\ell}\norm{v}{L^2({\rm supp}(v))}\quad\text{\rm for all } v\in B_\ell^1\text{ and all } -1\leq s<3/2,\\
\norm{v}{H^{s}({\rm supp}(v))}&\simeq C_{\rm mesh}^{-s\ell}\norm{v}{L^2({\rm supp}(v))}\quad\text{\rm for all } v\in B_\ell^{1/2}\text{ and all } -1\leq s<3/2,
\label{eq:scaling12}\\
\norm{v}{H^{s}({\rm supp}(v))}&\simeq C_{\rm mesh}^{-s\ell}\norm{v}{L^2({\rm supp}(v))}\quad\text{\rm for all } v\in B_\ell^{-1/2}\text{ and all } -2\leq s<1/2.\label{eq:scalingm12}
\end{align}
\end{subequations}
Finally, with~\eqref{eq:graded} for $D_{\rm grad}\geq 1$, there holds for all $\ell\in\N$
\begin{align}\label{eq:span}
\begin{split}
 \SS^{2+}(\TT_\ell)&={\rm span}\set{v\in B^1}{v\in \SS^{2+}(\TT_\ell)},\\
 \PP^{1}(\TT_\ell|_\Gamma)&={\rm span}\set{v\in B^{-1/2}}{v\in \PP^1(\TT_\ell|_\Gamma)}.
 \end{split}
\end{align}
\end{theorem}
\begin{proof}
We aim to employ~\cite{Dahmen} with the operators $(S_\ell)_{\ell\in\N_0}$ from Theorem~\ref{thm:szb}. The $S_\ell$ are uniformly $H^1(\Omega)$ bounded and satisfy $S_\ell S_k= S_\ell$ for all $\ell\leq k$. 
Moreover, their ranges $\SS^{2+}(\widehat \TT_\ell)$ form a dense and nested sequence of subspaces of $H^1(\Omega)$. Lemma~\ref{lem:gelfand} confirms that $H^{3/4}(\Omega)$ is the dual space of $H^{5/4}(\Omega)$ with respect to the $H^1(\Omega)$-scalar product. 
Theorem~\ref{thm:szb} confirms the approximation estimates
\begin{align*}
\norm{(1-S_\ell) u}{H^{3/4}(\Omega)}&\lesssim C_{\rm mesh}^{-\ell/4}\norm{u}{H^1(\Omega)},\\
\norm{(1-S_\ell) u}{H^{1}(\Omega)}&\lesssim C_{\rm mesh}^{-\ell/4}\norm{u}{H^{5/4}(\Omega)}
\end{align*}
as well as uniform boundedness $S_\ell\colon H^{3/4}(\Omega)\to H^{3/4}(\Omega)$. Standard inverse estimates prove
\begin{align*}
\norm{S_\ell u}{H^{5/4}(\Omega)}&\lesssim C_{\rm mesh}^{\ell/4}\norm{S_\ell u}{H^1(\Omega)},\\
\norm{S_\ell u}{H^{1}(\Omega)}&\lesssim C_{\rm mesh}^{\ell/4}\norm{S_\ell u}{H^{3/4}(\Omega)}.
\end{align*}
Therefore, we may apply~\cite[Theorems~3.1\&3.2]{Dahmen} to prove
\begin{align}\label{eq:Did}
\norm{u}{H^1(\Omega)}^2\simeq \sum_{\ell=0}^\infty \norm{(S_\ell-S_{\ell-1})u}{H^1(\Omega)}^2,
\end{align}
where we define $S_{-1}:=0$.
The identity~\eqref{eq:Sid} implies that for $v_0\in \SS^{2+}(\widehat\TT_k)\setminus\SS^{2+}(\widehat\TT_{k-1})$, there holds
\begin{align*}
 (1-J_{\ell-1})v_0= (1-S_{\ell-1})v_0
\end{align*}
and Theorem~\ref{thm:szb} shows
\begin{align*}
(S_\ell-S_{\ell-1})(1-S_{k-1})v_0&=S_\ell v_0 - S_\ell S_{k-1}v_0 - S_{\ell-1}v_0+ S_{\ell-1}S_{k-1}v_0 \\
&=\begin{cases}
S_\ell v_0- S_{\ell-1}v_0 =0 &k<\ell,\\ 
(1-S_{\ell-1}) v_0& k=l,\\ 
0 & k>\ell.
\end{cases}
\end{align*}
Thus, writing $w=\sum_{\ell\in\N_0}\sum_{v\in B^1_\ell} \alpha_v v$, we get with~\eqref{eq:Did}
\begin{align}\label{eq:firstriesz}
\norm{w}{H^{1}(\Omega)}^2&\simeq \sum_{\ell\in\N_0}\norm{(S_\ell-S_{\ell-1})w}{H^1(\Omega)}^2\simeq \sum_{\ell\in\N_0}\norm{\sum_{v\in B^1_\ell} \alpha_v v}{H^1(\Omega)}^2.
\end{align}
We define
\begin{align*}
	\widetilde B_\ell^1:=\Big(\bigcup_{k\in\N}\SS^{2+}_B(\TT_k\setminus\TT_{k-1})\Big)\cap (\SS^{2+}(\widehat\TT_\ell)\setminus\SS^{2+}(\widehat\TT_{\ell-1})).
\end{align*}
Each $v\in B^1_\ell$ is of the form $v=(1-J_{\ell-1})v_0$ for some $v_0\in\widetilde B_\ell^1$.

Let $\ell\geq 2$.
First, we prove that $\widetilde B_\ell^1|_T\setminus\{0\}$ is linearly independent for all $T\in\widehat\TT_{\ell-2}$.
To that end, assume $g:=\sum_{i=1}^n\alpha_i v_{0,i}|_T =0$ for $v_{0,i}\in \widetilde B_\ell^1$ with $v_{0,i}|_T\neq 0$ and $\alpha_i\in\R$.
Since $v_{0,i}\in\SS^{2+}_B(\TT_{k_i}\setminus\TT_{k_i-1})\cap\SS^{2+}(\widehat\TT_\ell)\setminus\SS^{2+}(\widehat\TT_{\ell-1})$ for some minimal $k_i\in\N$, 
we know that ${\rm supp}(v_{0,i})$ contains at least one element $T^\prime\in\TT_{k_i}$ such 
that $T^\prime\notin \widehat\TT_{\ell-1}$. This together with~\eqref{eq:graded} for $D_{\rm grad}\geq 1$ shows that $\TT_k|_{{\rm supp}(v_{0,i})}$ is a local refinement of $\widehat\TT_{\ell-2}$.
This and $v_{i,0}|_T\neq 0$, then implies the following:
\begin{itemize}
 \item If $v_{0,i}$ is a hat function, its unique corresponding node $z_i$ satisfies $z_i\in T$.
  \item If $v_{0,i}$ is an edge bubble function, its unique corresponding edge $E_i$ satisfies $E_i\subset T$.
   \item If $v_{0,i}$ is an element bubble function, its unique corresponding element $T_i$ satisfies $T_i\subseteq T$.
\end{itemize}
Hence, with $g=0$, we see $g(z_i)=0$ and hence all $\alpha_i$ corresponding to hat-functions are zero. Next, we find $f(E_i)=0$ and hence all $\alpha_i$ corresponding to edge-bubble functions are zero.
Finally, we see $f(T_i)=0$, which concludes that all $\alpha_i$ are zero. This proves that $\widetilde B_\ell^1|_T\setminus\{0\}$ is linearly independent.

Moreover, since newest-vertex bisection produces only finitely many shapes, and the number of refinements between $T$ and $T_i$, $E$ and $E_i$ is bounded in terms of $C_{\rm mesh}$,
we see that $\widetilde B_\ell^1|_T\setminus\{0\}$ belongs to a finite family of sets (up to
scaling).
Thus, a scaling argument shows 
for all  $T\in\widehat \TT_{\ell-2}$ 
\begin{align}\label{eq:szhelp}
\begin{split}
\norm{(1-S_{\ell-1})&\sum_{v_0\in \widetilde B_\ell^1}\alpha_{v_0}v_0}{H^1(T)}^2{\gtrsim}\norm{\nabla\sum_{v_0\in \widetilde B_\ell^1\atop {\rm supp}(v_0)\cap T\neq \emptyset}\alpha_{v_0}v_0}{L^2(T)}^2\\
&\simeq\sum_{v_0\in \widetilde B_\ell^1\atop {\rm supp}(v_0)\cap T\neq \emptyset}\alpha_{v_0}^2\norm{\nabla v_0}{L^2(T)}^2
\simeq\sum_{v_0\in \widetilde B_\ell^1\atop {\rm supp}(v_0)\cap T\neq \emptyset}\alpha_{v_0}^2\norm{v_0}{H^1(T)}^2,
\end{split}
\end{align}
where in the first estimate we used that {$\sum_{v_0\in \widetilde B_\ell^1}\alpha_{v_0}v_0 \notin \SS^{2+}(\widehat\TT_{\ell-1})$ as well as norm equivalence on finite dimensional spaces
(note that the hidden constants depend only on the shape regularity).} For $\ell\in\{0,1\}$, we obtain 
\begin{align*}
\norm{(1-S_{\ell-1})\sum_{v_0\in \widetilde B_\ell^1}\alpha_{v_0}v_0}{H^1(\Omega)}^2{\lesssim}
\sum_{v_0\in \widetilde B_\ell^1}\alpha_{v_0}^2\norm{v_0}{H^1(\Omega)}^2
\end{align*}
by stability of $S_{\ell-1}$.
Summing up, this means with~\eqref{eq:firstriesz}
\begin{align*}
\norm{w}{H^{1}(\Omega)}^2\simeq \sum_{\ell\in\N_0}\sum_{v\in B^1_\ell} \alpha_v^2\norm{ v_0}{H^1(\Omega)}^2.
\end{align*}
By~\eqref{eq:szhelp}, we obtain particularly $\norm{ v_0}{H^1(\Omega)}\lesssim \norm{v}{H^1(\Omega)}$. Continuity of $J_\ell$ implies $\norm{ v}{H^1(\Omega)}\lesssim \norm{v_0}{H^1(\Omega)}$. Altogether, we have
\begin{align*}
\norm{w}{H^{1}(\Omega)}^2\simeq \sum_{\ell\in\N_0}\sum_{v\in B^1_\ell} \alpha_v^2\norm{ v}{H^1({\rm supp}(v))}^2.
\end{align*}
Therefore, the operator $\iota\colon \ell_2(B^1)\to H^{1}(\Omega)$, $\iota(\alpha):=\sum_{v\in B^1}\alpha_v v$ is
bounded and has a bounded inverse on its closed range ${\overline{\bigcup_{\ell\in\N}\SS^{2+}(\TT_\ell)}\subseteq H^1(\Omega)}$. Obviously, the range is dense and hence $\iota$ is bijective.
This concludes that $B^1$ is a Riesz basis of ${\overline{\bigcup_{\ell\in\N}\SS^{2+}(\TT_\ell)}\subseteq }H^1(\Omega)$.

From this, we immediately deduce that $B^{1/2}$ is a Schauder generating set of $H^{1/2}(\Gamma)$ 
(since all $w\in H^{1/2}(\Gamma)$ have an extension $W\in H^1(\Omega)$). We need to show that the representation is unique also on the boundary. To that end, assume $W=\sum_{v\in B^1}\alpha_v v\in H^1(\Omega)$ with $W|_\Gamma=0$.
Lemma~\ref{lem:gensz} shows that $v|_\Gamma=0$ for all $v\in B^1_\ell$ for which the corresponding $v_0\in \widetilde B_\ell^1$ satisfies $v_0|_\Gamma=0$. 
Moreover, Theorem~\ref{thm:szb} shows that $(S_\ell-S_{\ell-1})W)|_\Gamma =0$. Hence, we have for all $\ell\in\N$
\begin{align}\label{eq:linind}
 0=\sum_{v\in B_\ell^1}\alpha_v ((S_\ell-S_{\ell-1})v)|_\Gamma = \sum_{v_0\in \widetilde B_\ell^1\atop v_0|_\Gamma\neq 0}\alpha_v(v_0|_\Gamma- (J_{\ell-1}v_0)|_\Gamma). 
\end{align}
Obviously, the set $\SS:=\set{v_0|_\Gamma}{v_0\in \widetilde B_\ell^1,\, v_0|_\Gamma\neq 0}$ is linearly independent (it consist of hat and bubble functions).
By definition, $\SS\cap \SS^{2+}(\widehat\TT_{\ell-1})|_\Gamma =\emptyset$. Hence,~\eqref{eq:linind} implies
\begin{align}\label{eq:traceid}
 W|_\Gamma =0 \quad\implies \quad \alpha_v=0\quad\text{for all }v\in B^1\text{ with } v|_\Gamma\neq 0.
\end{align}
Moreover, there holds for all $w\in H^{1/2}(\Gamma)$
\begin{align*}
\norm{w}{H^{1/2}(\Gamma)}^2=\inf_{W\in H^1(\Omega)\atop W|_\Gamma=w}\norm{W}{H^1(\Omega)}^2\simeq\inf_{W\in H^1(\Omega)\atop W|_\Gamma=w}
\sum_{\ell\in\N_0}\sum_{v\in B^1_\ell} (\iota^{-1}(W)_v)^2\norm{ v}{H^1({\rm supp}(v))}^2
\end{align*}
Since $B^{1/2}$ is a generating set of $H^{1/2}(\Gamma)$, we may represent $w$ as $w=\sum_{\ell\in\N_0}\sum_{v\in B^1_\ell} \alpha_v v|_\Gamma$, 
where we assume that all $v\in B^1$ which are zero on $\Gamma$ satisfy $\alpha_v=0$. Due to~\eqref{eq:traceid}, $W|_\Gamma=w$ implies $(\iota^{-1}(W))_v=\alpha_v$ for all $v\in B^1$ which are not zero on $\Gamma$.
This and the above shows
\begin{align*}
\norm{w}{H^{1/2}(\Gamma)}^2&=\inf_{W\in H^1(\Omega)\atop W|_\Gamma=w}\norm{W}{H^1(\Omega)}^2=
\sum_{\ell\in\N_0}\sum_{v\in B^1_\ell\atop v|_\Gamma\neq 0} \alpha_v^2\norm{ v}{H^1({\rm supp}(v))}^2.
\end{align*}
Let $v\in B^1_\ell$ with $v|_\Gamma$ non-zero. Let $\widehat\omega$ be one of finitely many reference patches
such that ${\rm supp}(v)$ and $\widehat\omega$ have the same shape. Let $\widehat\Sigma$ be one of finitely many
reference patches on the boundary, such that ${\rm supp}(v)\cap\Gamma$ and $\widehat\Sigma$ have the same shape.
Then, for $c\in\R$ a scaling argument shows
\begin{align*}
\norm{ v}{H^1({\rm supp}(v))}^2\simeq \norm{ \nabla  v}{L^2({\rm supp}(v))}^2\simeq 
\norm{ \nabla  \widehat v}{L^2(\widehat\omega)}^2 \simeq \norm{ \widehat v-c}{L^2(\widehat \Sigma)}^2\simeq C_{\rm mesh}^{\ell}
\norm{ \widehat v-c}{L^2({\rm supp}(v)\cap \Gamma)}^2.
\end{align*}
Choosing $c\in\R$ is the integral mean of $v$, a Poincar\'e inequality proves
\begin{align*}
\norm{ v}{H^1({\rm supp}(v))}^2\lesssim \norm{ v}{H^{1/2}({\rm supp}(v)\cap \Gamma)}^2.
\end{align*}
The continuity of the trace operator shows $\norm{ v}{H^1({\rm supp}(v))}=\norm{ v}{H^1(\Omega)}\gtrsim \norm{ v}{H^{1/2}(\Gamma)}\gtrsim\norm{ v}{H^{1/2}({\rm supp}(v)\cap \Gamma)} $, and hence
\begin{align*}
\norm{w}{H^{1/2}(\Gamma)}^2&=\inf_{W\in H^1(\Omega)\atop W|_\Gamma=w}\norm{W}{H^1(\Omega)}^2\simeq 
\sum_{\ell\in\N_0}\sum_{v\in B^1_\ell\atop v|_\Gamma\neq 0} \alpha_v^2\norm{ v}{H^{1/2}({\rm supp}(v)\cap\Gamma)}^2\\
&=
\sum_{\ell\in\N_0}\sum_{v\in B^{1/2}_\ell} \alpha_v^2\norm{ v}{H^{1/2}({\rm supp}(v))}^2.
\end{align*}
This concludes that $B^{1/2}$ is a Riesz basis of ${\overline{\bigcup_{\ell\in\N}\SS^{2}(\TT_\ell\cap\Gamma)}\subseteq }H^{1/2}(\Gamma)$.

 Note that all $v\in B^{1/2}\setminus B^{1/2}_0$ have zero integral mean (by Lemma~\ref{lem:gensz}).
Moreover, $\partial_\Gamma v$ has zero integral mean on ${\rm supp}(v)$ (since $v$ has zero boundary values on ${\rm supp}(v)$).
Thus, Lemma~\ref{lem:derivative}, Poincar\'e estimates, and inverse estimates show for $v\in B_\ell^{1/2}$, $\ell\geq 1$
\begin{align*}
\norm{\partial_\Gamma v}{\widetilde H^{-1/2}({\rm supp}(v))}\simeq C_{\rm mesh}^{\ell/2}\norm{\partial_\Gamma v}{\widetilde H^{-1}({\rm supp}(v))}=C_{\rm mesh}^{\ell/2} \norm{v}{L^2({\rm supp}(v))}\simeq \norm{v}{H^{1/2}({\rm supp}(v))}.
\end{align*}
For $v\in B_0^{1/2}$, the equivalence holds due to the fact that the number of functions in $B_0^{1/2}$ is bounded in terms of $\TT_0$.
Thus, again with Lemma~\ref{lem:derivative},
we derive for $ w = \sum_{\ell=1}^\infty\sum_{v\in B^{1/2}_\ell} \alpha_v  v$ 
\begin{align*}
\norm{\partial_\Gamma w}{H^{-1/2}(\Gamma)}^2&\simeq
\norm{ w}{H^{1/2}(\Gamma)}^2
\simeq 
\sum_{\ell=1}^\infty\sum_{v\in B^{1/2}_\ell} \alpha_v^2\norm{ v}{H^{1/2}({\rm supp}(v))}^2\\
&\simeq
\sum_{\ell=1}^\infty\sum_{v\in B^1_\ell} \alpha_v^2\norm{ \partial_\Gamma v}{\widetilde H^{-1/2}({\rm supp}(v))}^2.
\end{align*}
Therefore, the operator $\iota_1\colon \ell_2\to H^{-1/2}(\Gamma)$, $\iota(\alpha):=\sum_{v\in B^{-1/2}\setminus B_0^{-1/2}}\alpha_v v$ is
bounded and has a bounded inverse on its closed range. The operator $\iota_0\colon \ell_2\to H^{-1/2}(\Gamma)$, $\iota(\alpha):=\sum_{v\in  B_0^{-1/2}}\alpha_v v$ has finite dimensional range and trivial kernel, and is thus bounded with bounded inverse.
Obviously, the direct sum $\iota_0(\ell_2) \oplus \iota_1(\ell_2)\subseteq H^{-1/2}(\Gamma)$ is dense, and since the first summand is finite dimensional, the sum is even closed.
Thus, we have that $\iota:=\iota_0 + \iota_1$, is surjective, injective, and thus bijective. The open mapping theorem concludes that $\iota$ is bounded with bounded inverse.
This proves that $B^{-1/2}$ is a Riesz basis of ${\overline{\bigcup_{\ell\in\N}\PP^{1}(\TT_\ell\cap\Gamma)}\subseteq }H^{-1/2}(\Gamma)$.
Particularly, we proved~\eqref{eq:rieszB}.

The local support claims follow from the local support of the $v_0\in \SS^{2+}_B(\widehat\TT_\ell)$ and the fact that ${\rm supp}(1-J_{\ell-1})v_0\subseteq \omega(\TT_{\ell-1},{\rm supp}(v_0))$.

The scaling estimates~\eqref{eq:scaling1}--\eqref{eq:scalingm12} can be proved as follows: Let $v=(1-J_{\ell-1})v_0\in B_\ell^1$ for some $v_0\in \widetilde B_\ell^1$ and let $\omega:={\rm supp}(v)$. 
Lemma~\ref{lem:gensz} shows that
$v$ has element wise zero integral mean. Hence, we have for $0\leq s\leq 1$
\begin{align*}
\norm{w}{\widetilde H^{-s}(\omega)}\lesssim C_{\rm mesh}^{-s\ell}\norm{w}{L^2(\omega)}.
\end{align*}
The approximation property~\eqref{eq:sza} and the projection property of $J_{\ell-1}$ show
\begin{align*}
\norm{w}{L^2(\omega)}\lesssim C_{\rm mesh}^{-s\ell}\norm{w}{H^s(\omega)}
\end{align*}
for all $0\leq s<3/2$.
The converse estimates $\norm{w}{\widetilde H^{-s}(\omega)}\gtrsim C_{\rm mesh}^{-s\ell}\norm{w}{L^2(\omega)}$ as well as $\norm{w}{ L^2(\omega)}\gtrsim C_{\rm mesh}^{-s\ell}\norm{w}{H^s(\omega)}$
for $0\leq s<3/2$ follow from  standard inverse estimates. 
This concludes~\eqref{eq:scaling1}.
The estimate~\eqref{eq:scaling12} for $-1\leq s\leq 1$ follows from the fact that all $v\in B^{1/2}\setminus B_0^{1/2}$ have zero integral mean, Poincar\'e and inverse estimates.
For $1<s<3/2$,~\eqref{eq:scaling12} follows from $v\in H^1_0({\rm supp}(v))$ with a Friedrich's inequality together with inverse estimates. 
Moreover, there holds for $v=\partial_\Gamma w\in B_\ell^{-1/2}$ and an affine function $A\colon {\rm supp}(v)\to \R$
\begin{align*}
\norm{v}{\widetilde H^{-2}({\rm supp}(v))} &= \sup_{u\in H^2({\rm supp}(v))}\frac{\dual{\partial_\Gamma w}{u}_{{\rm supp}(v)}}{\norm{u}{H^2({\rm supp}(v))}}= \sup_{u\in H^2({\rm supp}(v))}\frac{-\dual{w}{\partial_\Gamma(u-A)}_{{\rm supp}(v)}}{\norm{u}{H^2({\rm supp}(v))}}
\\
&\leq\norm{w}{L^2({\rm supp}(v))} \sup_{u\in H^2({\rm supp}(v))}\frac{\norm{u-A}{H^1({\rm supp}(v))}}{\norm{u}{H^2({\rm supp}(v))}},
\end{align*}
where we used that $w\in H^1_0({\rm supp}(w))$ and $\int_{{\rm supp}(w)} w\,dx =0$ (by Lemma~\ref{lem:gensz}).
Since the above holds for all affine $A$, we get 
\begin{align*}
\norm{v}{\widetilde H^{-2}({\rm supp}(v))} \lesssim C_{\rm mesh}^{-\ell}\norm{w}{L^2({\rm supp}(v))} \simeq C_{\rm mesh}^{-\ell}\norm{v}{\widetilde H^{-1}({\rm supp}(v))}
\end{align*}
by an approximation estimate and Lemma~\ref{lem:derivative}. By use of the fact that $v$ has zero integral mean, 
we also obtain $\norm{v}{\widetilde H^{-2}({\rm supp}(v))} \lesssim C_{\rm mesh}^{-2\ell}\norm{v}{L^2({\rm supp}(v))}$. Inverse estimates finally prove~\eqref{eq:scalingm12} for all $-2\leq s\leq 0$. 
Again, with Poincar\'e and inverse estimates, we prove~\eqref{eq:scalingm12} for the remaining $0<s<1/2$.
This concludes the proof of the scaling estimates. 

{Finally, to see~\eqref{eq:span}, we note that $\SS^{2+}(\TT_j)={\rm span}\set{v_0\in\bigcup_{\ell=0}^\infty\widetilde B_\ell^1}{v_0\in\SS^{2+}(\TT_j)}$.
For each $v_0 \in \SS^{2+}(\TT_j)\cap \widetilde B_\ell^1$, we note that $J_{\ell-1}v_0$ is supported on $\omega({\rm supp}(v_0),\widehat\TT_{\ell-1})$.
Since $\TT_j|_{{\rm supp}(v_0)}$ is a strict local refinement of $\widehat\TT_{\ell-1}$ (at least one element is finer than $\widehat\TT_{\ell-1}$, Lemma~\ref{lem:graded0} shows
that $\TT_j|_{\omega({\rm supp}(v_0),\widehat\TT_{\ell-1} )}$ is a local refinement of $\widehat\TT_{\ell-1}$. 
This implies that $J_{\ell-1}v_0\in \SS^{2+}(\TT_j)$ and thus shows $\SS^{2+}(\TT_j)={\rm span}\set{v\in B^1}{v\in\SS^{2+}(\TT_j)}$. The argument works analogously for $\PP^1(\TT_j|_\Gamma)$
and thus concludes the proof.}
\end{proof}

\section{Application}\label{section:apps}
\subsection{Model problem}
Although the framework developed above seems to be fairly general, our main goal here is to prove general quasi-orthogonality 
for the non-symmetric Johnson-N\'ed\'elec coupling. This FEM/BEM coupling method stems from a transmission problem of the form
\begin{align}\label{eq:transmission}
\begin{split}
 -\Delta u &= F\quad\text{in }\Omega,\\
 -\Delta u &= 0\quad\text{in }\R^2\setminus\overline{\Omega},\\
 [u]&= U\quad\text{on }\Gamma,\\
 [\partial_n u]&=\Phi\quad\text{on }\Gamma,\\
 |u(x)|&=a+b\log|x|\quad\text{as }|x|\to \infty,
\end{split}
 \end{align}
for given functions $F\in L^2(\Omega)$, $U\in H^{1/2}(\Gamma)$, $\Phi\in L^2(\Gamma)$, and constants $a,b\in\R$. Here, $[\cdot]$ denotes the jump over $\Gamma$ and $\partial_n$ is the normal derivative on $\Gamma$.

The first FEM/BEM coupling approach for this problem was Costabel`s symmetric coupling~\cite{costabel}. While this coupling method induces an operator which is symmetric, it lacks positive definiteness. 
Reformulation of the method into a positive definite one destroys the symmetry. Therefore, rate optimality of the adaptive algorithm is also open for this \emph{symmetric} method. In principle, the methods
developed here can be used directly to prove optimality for Costabel`s symmetric coupling. In this work, however, we focus on another coupling method (which is very poplar amongst engineers due to its simpler
implementation) called one-equation coupling or Johnson-N\'ed\'elec coupling first proposed in~\cite{johned} (see also~\cite{fembem} for further details).

We define $\XX:=H^1(\Omega)\times H^{-1/2}(\Gamma)$. With the integral operators
\begin{align}\label{eq:integralops}
(V\phi)(x):=-\frac{1}{2\pi}\int_\Gamma \log|x-y|\phi(y)\, dy \quad \text{ and } \quad 
(K g)(x):=-\frac{1}{2\pi}\int_\Gamma \partial_{n(y)} \log|x-y|g(y)\, dy 
\end{align}
for all $x\in\Gamma$, we may consider the weak form of the problem above
\begin{align*}
 \widetilde a((\uint,\phi),(\vint,\psi))=\widetilde f(\vint,\psi)\quad\text{for all }(\vint,\psi)\in \XX
\end{align*}
with
\begin{align*}
 \widetilde a((\uint,\phi),(\vint,\psi)):= \dual{\nabla \uint}{\nabla \vint}_\Omega -\dual{\phi}{\vint}_\Gamma + \dual{(1/2-K)\uint}{\psi}_\Gamma +\dual{V\phi}{\psi}_\Gamma
\end{align*}
and
\begin{align*}
 \widetilde f(\vint,\psi):= \dual{F}{\vint}_\Omega +\dual{\Phi}{\vint}_\Gamma + \dual{\psi}{(1/2-K)U}_\Gamma.
\end{align*}

The connection to the transmission problem~\eqref{eq:transmission} is given by
\begin{align*}
 u|_\Omega = u^{\rm int}, \quad \partial_n u|_{\R^2\setminus\Omega} = -\phi,\quad\text{and } u|_{\R^2\setminus\Omega} = V\phi + K u^{\rm int}.
\end{align*}

Existence of unique solutions of the above method was first proved in~\cite{johned} for the case of smooth $\Gamma$. 
Almost three decades later, Sayas~\cite{sayas} proved existence of unique solutions also for the case of polygonal boundaries $\Gamma$. This work was extended in~\cite{fembem} to
nonlinear material parameters and other coupling methods.

Given a Galerkin solution in the sense of~\eqref{eq:solutions}, i.e., $(\uint_\TT,\phi_\TT)\in \XX_\TT:=\SS^{2+}(\TT)\times \PP^1(\TT|_\Gamma)$,
the corresponding residual-based error estimator (see e.g.~\cite{afp,fembem} for the derivation) reads element wise for all $T\in\TT$
 \begin{align*}
\eta_T(\TT)^2&:={\rm diam}(T)^2\norm{F+\Delta \uint_\TT)}{L^2(T)}^2 +{\rm diam}(T)\norm{[\partial_n \uint_\TT]}{L^2(\partial T\cap\Omega)}^2\\
 &\qquad + {\rm diam}(T)\norm{\Phi+\phi_\TT-\partial_n \uint_\TT}{L^2(\partial T\cap\Gamma)}^2\\
 &\qquad+{\rm diam}(T)\norm{\partial_\Gamma((\tfrac12-K)(U-\uint_\TT)-V\phi_\TT)}{L^2(\partial T\cap\Gamma)}^2,
\end{align*}
where $\partial_\Gamma$ denotes the arc-length derivative on $\Gamma$.
Note that the exterior problem affects the estimator only on elements $T\in\widehat\TT$ with $T\cap\Gamma\neq \emptyset$. 
The overall estimator reads
\begin{align*}
 \eta(\TT):=\Big(\sum_{T\in\TT}\eta_T(\TT)^2\Big)^{1/2}\quad\text{for all }\TT\in\T.
\end{align*}

We define $\XX_\ell:=\SS^{2+}(\TT_\ell)\times \PP^1(\TT_\ell|_\Gamma)$ and to fit the problem into our abstract framework, we choose the following Riesz basis from Theorem~\ref{thm:stableBbase}: 
\begin{align*}
	\bB:=\set{(v,0)}{v\in B^1}\cup \set{(0,w)}{w\in B^{-1/2}}.
\end{align*}
We recall that $\XX_\ell\subseteq \XX_{\ell+1}\subset \XX$ are nested finite dimensional spaces generated 
by the adaptive algorithm described in Section~\ref{section:mesh}.
We order the functions in $\bB$ such that $\XX_\ell={\rm span}\{\bw_1,\bw_2,\ldots,\bw_{N_\ell}\}$ for particular $N_\ell\in\N$ and all $\ell\in\N$ (note that this is possible due to~\eqref{eq:span}).

\begin{remark}
 We restrict to the spaces $\SS^{2+}(\TT_\ell)\times \PP^1(\TT_\ell|_\Gamma)$ for technical reasons only. As it becomes obvious in the proof of Theorem~\ref{thm:stableBbase}, it is essential
 to exploit a number of moment conditions of the basis functions. The author is convinced that this could also be done for the lowest order case $\SS^1(\TT_\ell)\times\PP^0(\TT_\ell|_\Gamma)$ at the expense
 of greater care for technical details. Similarly, the techniques should transfer straightforwardly to higher order discretizations.
\end{remark}

Since $\widetilde a(\cdot,\cdot)$ is not elliptic, we have to seek an elliptic reformulation. This was first demonstrated in~\cite[Section~4.2]{fembem}. The forms $\widetilde a$ and $\widetilde f$
can be replaced by equivalent forms which produce the same solutions~\eqref{eq:solutions} as long as $\bigcap_{\ell\in\N}\XX_\ell$ contains
a function $\xi$ which satisfies $\dual{\xi}{1}_\Gamma\neq 0$. This is obviously the case for our choice of $\XX_\ell$, and therefore, we may theoretically consider
\begin{align*}
 a((\uint,\phi),(\vint,\psi)):=  \widetilde a&((\uint,\phi),(\vint,\psi))\\
 &+\dual{\xi}{(1/2-K)\uint + V\phi}_\Gamma\dual{\xi}{(1/2-K)\vint+V\psi}_\Gamma
\end{align*}
and
\begin{align*}
 f(\vint,\psi):= \widetilde f(\vint,\psi) + \widetilde f(0,\xi)\dual{\xi}{(1/2-K)\vint+V\psi}_\Gamma.
\end{align*}
The result~\cite[Theorem~14]{fembem} states that $a(\cdot,\cdot)$ is elliptic~\eqref{eq:aelliptic}.
With $u:=(\uint,\phi)$ and $v:=(\vint,\psi)$, this fits the abstract from~\eqref{eq:solutions} and we define $u_\ell:=(\uint_\ell,\phi_\ell)$ for all $\ell\in\N$
with $\XX_\ell:= \SS^{2+}(\TT_\ell)\times \PP^1(\TT_\ell|_\Gamma)$. {Standard arguments (see, e.g.,~\cite[Corollary~4.8]{axioms}) prove that the exact solution satisfies
$u\in\overline{\bigcup_{\ell\in\N}\XX_\ell}\subseteq \XX$. Therefore, we can restrict the problem to the space $\XX^\prime:= \overline{\bigcup_{\ell\in\N}\XX_\ell}$.}

\begin{remark}
 Note that $a$ and $f$ produce the same solutions~\eqref{eq:solutions} as $\widetilde a$ and $\widetilde f$. Thus, for actual computation it is more convenient to use $\widetilde a$ and $\widetilde f$,
 whereas for our theoretical considerations,
 it is mandatory to use $a$ and $f$.
\end{remark}

\subsection{Main result}

The following result shows rate optimality of the adaptive algorithm and is the main result of the paper. 

\begin{theorem}[Optimality of the adaptive algorithm]\label{thm:opt}
Given sufficiently small $\theta>0$ and sufficiently large $D_{\rm grad}\geq 1$, 
Algorithm~\ref{algorithm} applied to the Johnson-N\'ed\'elec FEM-BEM coupling as described above guarantees rate-optimal convergence, i.e., there exists a constant $C_{\rm opt}>0$ such that
\begin{align*}
 C_{\rm opt}^{-1}\norm{u}{\A_s}\leq\sup_{\ell\in\N_0}\frac{\eta(\TT_\ell)}{(\#\TT_\ell-\#\TT_0+1)^{-s}}\leq C_{\rm opt}\norm{u}{\A_s},
\end{align*}
for all $s>0$ with $\norm{u}{\A_s}<\infty$.
\end{theorem}
\begin{proof}
We have to specify a mesh-refinement strategy which ensures~\eqref{eq:graded} and fits into the framework of~\cite[Section~2.4]{axioms}.
To that end, we use the strategy specified in~\cite[Section~A.3]{l2opt} (note that the condition in~\cite[Section~A.3]{l2opt} and~\eqref{eq:graded} are equivalent up to shape regularity).
Then, the result follows immediately from~\cite[Theorem~4.1]{axioms} and Lemma~\ref{lem:approxclass}, after we prove the axioms (A1)--(A4) in the sections below.
\end{proof}

\subsection{Proof of (A3)}
The main innovation of this paper is the proof of general quasi-orthogonality~(A3).
\begin{lemma}\label{lem:graded}
 Let $\TT\in\T$ such that $\v\in \SS^{2+}(\TT)\times \PP^1(\TT|_\Gamma)$ for some $\v\in \B_\ell$.
Let $D_{\rm grad}>0$ and assume~\eqref{eq:graded}.
Then, $k<\ell$ implies
\begin{align*}
 \big(\w\in \B_k,\,\delta_k(\v,\w)\leq D_{\rm grad}-C_{\rm grad}\big)\quad\implies\quad \w\in \SS^{2+}(\TT)\times \PP^1(\TT|_\Gamma),
\end{align*}
with $\delta_k(\cdot,\cdot)$ from Definition~\ref{def:metric} and $C_{\rm grad}\geq 1$ depends only on $\TT_0$ and $C_{\rm mesh}$. 
\end{lemma}
{\begin{proof}
Define 
\begin{align*}
 C:=\sup_{k\in\N}\max_{\w\in\bigcup_{j=k}^\infty\B_j}\#\set{T\in\widehat\TT_k}{T\subseteq {\rm supp}(\w)}.
\end{align*}
By definition of $\B$, we see that $C<\infty$ in terms of $\TT_0$ and $C_{\rm mesh}$.
Assume $T,T^\prime\in \TT$ with $T\cap {\rm supp}(\w)\neq \emptyset$ and $T^\prime\cap {\rm supp}(\v)\neq \emptyset$.
Let $L_\ell$ denote the uniform level of elements in $\widehat\TT_\ell$. Assume  ${\rm level}(T)< L_{k}$.
Assumption~\eqref{eq:graded} implies that all $T^{\prime\prime}\in \omega^{D_{\rm grad}}(T,\TT)$
satisfy ${\rm level}(T^{\prime\prime})\leq L_{k}$. This means that $\widehat\TT_{k}|_{\omega^{D_{\rm grad}}(T,\TT)}$ is a local refinement of $\TT$.
With $\chi_T$ denoting the function which is one on $T$ and zero elsewhere, we get
\begin{align*}
 \delta_k(\chi_T,\chi_{T^\prime})\leq 2C + \delta_k(\v,\w).
\end{align*}
Thus, $\delta_k(\v,\w)\leq {D_{\rm grad}}-2C-1$ implies $\delta_k(\chi_{T},\chi_{T^\prime})\leq D_{\rm grad}$ and,
since $\widehat\TT_k|_{\omega^{D_{\rm grad}}(T,\TT)}$ is a local refinement of $\TT$, also $T^\prime \in \omega^{{D_{\rm grad}}}(T,\TT)$.
This shows ${\rm level}(T^\prime)\leq L_k\leq L_{\ell-1}$ and contradicts the assumption $\v\in B_\ell$.
Thus, we proved ${\rm level}(T)\geq  L_{k}$ for all $T\in\TT$ with $T\cap {\rm supp}(\w)\neq \emptyset$. This implies that $\TT|_{{\rm supp}(\w)}$ is a local refinement of $\widehat\TT_{k}$ and thus
$\w\in \SS^{2+}(\TT)\cup \PP^1(\TT|_\Gamma)$. We conclude the proof with $C_{\rm grad}:=2C+1$.
\end{proof}}

\begin{theorem}\label{thm:qosum}
 Given $\eps>0$ and under all previous assumptions, there exists $D_{\rm grad}>0$ sufficiently large such that 
 the solutions~\eqref{eq:solutions} of the transmission problem~\eqref{eq:transmission} satisfy general quasi-orthogonality~\eqref{eq:qosum2}.
\end{theorem}

\begin{remark}
 The idea of the proof of the theorem is to consider the matrix $A_{ij}:= a(\bw_i,\bw_j)$. The lemmas from Section~\ref{section:banded} below show that $A$ is close to a matrix $M^\eps$ which is banded with
 respect to the metric $d_2(\cdot,\cdot)$ from Section~\ref{section:metric}. Ideally, we would like to apply Theorem~\ref{thm:luqo2} directly. This, however, is not possible since $d_2(\cdot,\cdot)$ does
 not satisfy~\eqref{eq:Mdist}. Thus, $M^\eps$ is not of any Jaffard class and therefore Theorem~\ref{thm:luqo2} does not apply.
 The remedy is to first consider a permutation of $A$ which is block-banded. We then prove that there exists a block banded $LDU$-factorization which is used to define two new Riesz bases of $\XX$.
 The corresponding matrix $M$ turns out to be close to the block diagonal matrix $D$, which satisfies $D\in\JJ$. Thus, Theorem~\ref{thm:luqo3} applies and concludes the proof. 
 The redefinition of the Riesz bases is the reason for the grading condition~\eqref{eq:graded} on the adaptive meshes. This condition ensures that the spaces $\XX_\ell$ are still spanned by
 the new Riesz bases.
\end{remark}

\begin{remark}
One might ask the question if there exists a better metric $d_2(\cdot,\cdot)$ such that $A$ is still close to a  matrix in $\BB(d_2,\cdot)$, and that $d_2(\cdot,\cdot)$ satisfies~\eqref{eq:Mdist}.
This would simplify the proof and remove the grading condition~\eqref{eq:graded}. We argue that such a metric is likely not to exists.

Given $\bw_i,\bw_j\in\bB_\ell$, we choose $\bw_{ij,k}\in \bigcup_{\nu\leq \ell}\bB_\nu$, $k=1,\ldots,n$ with $\bw_{ij,1}=\bw_i$ and $\bw_{ij,n}=\bw_j$ as follows:
We require that $\bw_{ij,k}\in \bB_{\ell-k}$ for all $k\leq n/2$
and that their supports overlap. Moreover, we choose the sequence such that at some point
the support of $\bw_{ij,n/2}$ overlaps
with $\bw_j$. The second half of the $\bw_{ij,k}\in \bB_{\ell-n/2+k}$ is chosen analogously.
Thus, if the distance of functions with supporting overlap and level difference one is uniformly bounded, the triangle inequality implies
\begin{align*}
 d_2(\bw_i,\bw_j)\leq d_2(\bw_i,\bw_{ij,1})+d_2(\bw_{ij,1},\bw_{ij,2}) +\ldots+ d_2(\bw_{ij,n},\bw_j)\lesssim n,
\end{align*}
where $n\simeq \log_{C_{\rm mesh}}(\delta_\ell(\bw_i,\bw_j))$. ( Since the supports grow/shrink in size by the factor of $C_{\rm mesh}$ from level to level, 
we see the condition on $n$.) With $d_2(\bw_i,\bw_j)\lesssim \log(\delta_\ell(\bw_i,\bw_j))$, condition~\eqref{eq:Mdist} cannot hold.
Thus, the distance of functions with supporting overlap and level difference one must not be uniformly bounded. This, however, implies that any matrix being close to $\BB(d_2)$ or $\JJ(d_2)$ is
essentially diagonal as row and column numbers approach infinity. This would imply a much stronger condition on the chosen Riesz bases and it is not clear how to construct such a basis.
\end{remark}

\begin{proof}
With the basis $\bB$, define the matrix $A\in\R^{\N\times \N}$ by
\begin{align*}
 A_{ij}:=a(\bw_j,\bw_i). 
\end{align*}
Since $\bB$ is a Riesz basis, boundedness and ellipticity of $a(\cdot,\cdot)$ imply that also $A$ is elliptic and bounded.
By reordering $\bB$ such that $L(\bw_{k_i})\leq L(\bw_{k_j})$ for all $i\leq j$, we obtain a permuted matrix
\begin{align*}
 \widetilde A_{ij}:= A_{k_ik_j}\quad\text{or}\quad  \widetilde A=P^TAP
\end{align*}
for some permutation matrix $P\in \{0,1\}^{\N\times\N}$ defined by $P_{ij}=1$ if and only if $i=k_j$.
We introduce a block-structure on $B$ with $n_1,n_2,\ldots$ such that $\set{\bw_{k_i}}{i=n_r,\ldots,n_{r+1}-1}=\set{\bw\in\bB}{L(\bw)=r}$.
Then, Lemmas~\ref{lem:Mscaling}--\ref{lem:Ascaling} below show that there exists $ \widetilde A^\eps$ which is block-banded for some bandwidth $b$ such that
\begin{align*}
 \norm{ \widetilde A- \widetilde A^\eps}{2}\leq \eps.
\end{align*}
Moreover, if we identify $i\mapsto \bw_i$, there holds $ \widetilde A^\eps\in\BB(d_2)$ with the metric $d_2(\cdot,\cdot)$ from Definition~\ref{def:metric} (we verify the conditions (i)--(iv) posed on $\bB$
in Section~\ref{section:metric} with standard arguments). Choosing $\eps>0$ sufficiently small, we ensure that $\widetilde A^\eps$ is elliptic and bounded.
Thus, Theorem~\ref{thm:blockLU} shows that there exists an approximate block-$LDU$-factorization with $\norm{ \widetilde A-LDU}{2}\leq 2\eps$. Moreover, the factors $L^{-1},D,U^{-1}\in\BB(d_2)$ are block-banded 
with bandwidth $b$.
Since $D$ is block-diagonal, we also have $D\in\BB(d_3)$ for the metric $d_3(\cdot,\cdot)$ from Definition~\ref{def:metric}.
Together with Lemma~\ref{lem:d3metric}, this shows $D\in \JJ(d_3)$.
We define a new basis $B^U$ by
\begin{align*}
 B^L:=\bigcup_{\ell\in\N}B^U_\ell,\quad\text{and}\quad B^U_\ell:=\set{v_{k_j}:=\sum_{i=1}^\infty (U^{-1})_{ij}\bw_{k_i}}{j\in\N,\,L(\bw_{k_j})=\ell}.
\end{align*}
and we define $B^L$ by
\begin{align*}
 B^L:=\bigcup_{\ell\in\N}B^L_\ell,\quad\text{and}\quad B^L_\ell:=\set{w_{k_j}:=\sum_{i=1}^\infty (L^{-1})_{ji}\bw_{k_i}}{j\in\N,\,L(\bw_{k_j})=\ell}.
\end{align*}
Since $L^{-1},U^{-1}$ are inversely bounded uniformly in $\eps$, we see that also $B^U$ and $B^L$ are a Riesz basis of $\XX$.
Moreover, since $L^{-1},U^{-1}\in\BB(d_2,b^\prime)$ and block-banded with a certain bandwidth $b^\prime$, and since the diagonal blocks of $L^{-1}$ and $U^{-1}$ are just identities, we have
\begin{align*}
 v_{k_j}&= \bw_{k_j} + \sum_{\ell=L(\bw_{{k_j}})-b^\prime}^{L(\bw_{{k_j}})-1} \sum_{\bw_{k_i}\in B_\ell\atop\delta_\ell(\bw_{k_j},\bw_{k_i})\leq n}(U^{-1})_{ij}\bw_{k_i},\\
  w_{k_j}&= \bw_{k_j} + \sum_{\ell=L(\bw_{{k_j}})-b^\prime}^{L(\bw_{{k_j}})-1} \sum_{\bw_{k_i}\in B_\ell\atop \delta_\ell(\bw_{k_j},\bw_{k_i})\leq n}(L^{-1})_{ji}\bw_{k_i} 
\end{align*}
for some $n\in\N$ which depends only on $b^\prime$.
Lemma~\ref{lem:graded} shows that since all $\TT_\ell$ satisfy~\eqref{eq:graded} for sufficiently large $D_{\rm grad}\geq n+C_{\rm grad}$, we have that $\bw_j\in\XX_\ell$ implies $v_j,w_j\in\XX_\ell$.
Therefore, we have
\begin{align*}
 \XX_\ell={\rm span}\{v_1,v_2,\ldots,v_{N_\ell}\}={\rm span}\{w_1,w_2,\ldots,w_{N_\ell}\}.
\end{align*}
Moreover, with $C_{ij}:= a(v_{k_j},w_{k_i})$, we see $C=L^{-1}\widetilde A U^{-1}$ and hence
\begin{align*}
\norm{C-D}{2}&=\norm{L^{-1}(\widetilde A-LDU) U^{-1}}{2}\lesssim \eps,
\end{align*}
where we used that $L^{-1}$ and $U^{-1}$ are bounded uniformly in $\eps$ (see Theorem~\ref{thm:blockLU}).
Considering $M:= PCP^T$ and $M^\eps:= PDP^T$, we obtain with the above $\norm{M-M^\eps}{2}\lesssim \eps$ as well as $M^\eps\in \JJ(d_3)$.
Moreover, Theorem~\ref{thm:blockLU} shows that $D$ and thus $M^\eps$ are elliptic~\eqref{eq:matelliptic}. Choosing $\eps>0$ sufficiently small, we ensure that also $M$ is elliptic.
Thus, Theorem~\ref{thm:luqo3} {(for $\XX^\prime=\overline{\bigcup_{\ell\in\N}\XX_\ell}$ instead of $\XX$)} applies and concludes the proof.
\end{proof}

\subsection{Proof of (A1), (A2), and (A4)}
The proofs of the properties (A1), (A2), and (A4) are combinations of techniques from the FEM case and from the BEM case (mainly from~\cite{fkmp}). While no expert will be surprised by the following
proofs, they cannot be found in the literature and we included them for completeness.
\begin{proof}[Proof of (A1)--(A2)]
 The statements~(i) and~(ii) are part of the proof of~\cite[Theorem~25]{fembem} and follow from the triangle inequality and 
 local inverse estimates for the non-local operators $ V$ and $ K$ from~\cite{invest}. The constants $C_{\rm stab},C_{\rm red},q_{\rm red}$ depend only on $\Gamma$ and the shape regularity of $\TT$
 and $\widehat\TT$.
\end{proof}

\begin{proof}[Proof of (A4)]
 The proof is essentially the combination of the corresponding proofs for FEM in~\cite{stevenson07,ckns} and BEM in~\cite{fkmp}. There holds with ellipticity~\cite[Theorem~14]{fembem} and 
  $v_{\widehat{\TT}}:=(\uint_{\widehat{\TT}},\psi_{\widehat{\TT}}):=u_{\widehat{\TT}}-u_{{\TT}}\in\SS^{2+}(\widehat\TT)\times \PP^1(\widehat\TT|_\Gamma)$ by use of Galerkin orthogonality
 \begin{align*}
  \norm{u_{\widehat{\TT}}-u_\TT}{\XX}^2&\lesssim a(u_{\widehat{\TT}}-u_{{\TT}},v_{\widehat{\TT}})=\widetilde a(u_{\widehat{\TT}}-u_{{\TT}},v_{\widehat{\TT}})\\
  &=f(v_{\widehat{\TT}}-v_{{\TT}})-\widetilde a(u_{{\TT}},v_{\widehat{\TT}}-v_{{\TT}})\quad\text{for all }v_{{\TT}}\in\XX_\ell.
 \end{align*}
Recall the Scott-Zhang operator $J_\TT: H^1(\Omega)\to \SS^{2+}(\TT)$ from Section~\ref{section:szb} as well as the $L^2(\Gamma)$-orthogonal projection 
$\Pi_\TT:L^2(\Gamma) \to \PP^{1}(\TT|_\Gamma)$. With this, define
\begin{align*}
 v_{{\TT}}:=(J_\TT\uint_{\widehat{\TT}},\Pi_\TT \psi_{\widehat{\TT}})\in\SS^{2+}(\TT)\times \PP^1(\TT|_\Gamma).
\end{align*}
This implies
\begin{align}\label{eq:johneddrelstart}
\begin{split}
  \norm{u_{\widehat{\TT}}-u_{{\TT}}}{\XX}^2&\lesssim \dual{F}{(1-J_\TT)\uint_{\widehat{\TT}}}_{L^2(\Omega)}-\dual{\nabla \uint_\TT}{\nabla (1-J_\TT)\uint_{\widehat{\TT}}}_{L^2(\Omega)}\\
  &\qquad  +  \dual{\Phi+\phi_\TT}{(1-J_\TT)\uint_{\widehat{\TT}}}_{L^2(\Gamma)}\\
  &\qquad + \dual{(1/2- K)(U-\uint_\TT)- V\phi_\TT}{(1-\Pi_\TT)\psi_{\widehat{\TT}}}_{L^2(\Gamma)}.
  \end{split}
\end{align}
$\TT$-piecewise integration by parts shows
\begin{align*}
 \dual{&F}{(1-J_\TT)\uint_{\widehat{\TT}}}_{L^2(\Omega)}-\dual{\nabla \uint_\TT}{\nabla (1-J_\TT)\uint_{\widehat{\TT}}}_{L^2(\Omega)}+\dual{\Phi+\phi_\TT}{(1-J_\TT)\uint_{\widehat{\TT}}}_{L^2(\Gamma)}\\
 &\lesssim\sum_{T\in\TT}\normLtwo{F+\Delta\uint_\TT}{T}\normLtwo{(1-J_\TT)\uint_{\widehat{\TT}}}{T}\\
 &\qquad+\sum_{T\in\TT}\Big(\normLtwo{[\partial_n\uint_\TT]}{\partial T\cap\Omega}
 + \normLtwo{\Phi+\phi_\TT-\partial_n\uint_\TT }{\partial T\cap\Gamma}\Big)\normHeh{(1-J_\TT)\uint_{\widehat{\TT}}}{T}.
\end{align*}
Since all $T\in\TT$ with $T\notin \omega(\TT\setminus\widehat\TT,\TT)$ satisfy $((1-J_\TT)\uint_{\widehat{\TT}})|_T = 0$ 
 and by use of the first-order approximation properties of $J_\TT$, the above estimate implies
\begin{align}\label{eq:johnedhelp}
 \dual{&F}{(1-J_\TT)\uint_{\widehat{\TT}}}_{L^2(\Omega)}-\dual{\nabla \uint_\TT}{\nabla (1-J_\TT)\uint_{\widehat{\TT}}}_{L^2(\Omega)}+\dual{\Phi-\phi_\TT}{(1-J_\TT)\uint_{\widehat{\TT}}}_{L^2(\Gamma)}\nonumber\\
 &\lesssim\sum_{T\in\omega(\TT\setminus\widehat\TT,\TT)}\Big({\rm diam}(T)\normLtwo{F+\Delta\uint_\TT}{T}+{\rm diam}(T)^{1/2} \normLtwo{[\partial_n\uint_\TT]}{\partial T\cap\Omega}\nonumber\\
 &\qquad+ {\rm diam}(T)^{1/2}\normLtwo{\Phi+\phi_\TT-\partial_n\uint_\TT}{\partial T\cap\Gamma}\Big)\normLtwo{\nabla\uint_{\widehat{\TT}}}{T},
\end{align}
where the hidden constant depends only on the shape regularity of $\TT$ and $\Omega$.
Consider a partition of unity of $\Gamma$ in the sense
\begin{align*}
 \sum_{z\in\Gamma\atop z\text{ node of }\TT} \xi_z = 1\quad\text{on }\Gamma
\end{align*}
with the nodal hat functions $\xi_z\in\NN(\TT|_\Gamma)$.
Since $(1-\Pi_\TT)\psi_{\widehat{\TT}} = 0 $ on $\TT\cap\widehat\TT$, the last term on the right-hand side of~\eqref{eq:johneddrelstart} satisfies
\begin{align*}
  \dual{(1/2- K)(U-\uint_\TT)&- V\phi_\TT}{(1-\Pi_\TT)\psi_{\widehat{\TT}}}_{L^2(\Gamma)}\\
  &=
   \dual{\sum_{z\in\bigcup(\TT\setminus\widehat\TT)\cap\Gamma\atop z\text{ node of }\TT}\xi_z\big((1/2- K)(U-\uint_\TT)- V\phi_\TT\big)}{(1-\Pi_\TT)\psi_{\widehat{\TT}}}_{L^2(\Gamma)}.
\end{align*}
Galerkin orthogonality shows $ \dual{1}{(1/2- K)(U-\uint_\TT)- V\phi_\TT}_{L^2(T\cap\Gamma)}=0$ for all $T\in\TT$ allows to follow the arguments of the proof of~\cite[Proposition~5.3]{fkmp} resp.~\cite[Proposition~4]{ffkmp:part1}. This shows
\begin{align}\label{eq:johnedhelp2}
 \dual{(1&-\Pi_\TT)\psi_{\widehat{\TT}}}{(1/2- K)(U-\uint_\TT)- V\phi_\TT}_{L^2(\Gamma)}\\\nonumber
 &\lesssim \Big(\sum_{T\in\omega(\TT\setminus\widehat\TT,\TT)}{\rm diam}(T)^{1/2}\normLtwo{\nabla_\Gamma\big((1/2- K)(U-\uint_\TT)- V\phi_\TT\big)}{T\cap\Gamma}\Big)\normHmeh{\psi_{\widehat{\TT}}}{\Gamma}.
\end{align}
The combination of~\eqref{eq:johnedhelp}--\eqref{eq:johnedhelp2} with~\eqref{eq:johneddrelstart} concludes the proof of the discrete reliability~(A4) with $\RR(\TT,\widehat\TT):=\omega(\TT\setminus\widehat\TT,\TT)$
and $C_{\rm ref}$ depending only on shape regularity. 
\end{proof}

\subsection{Discretization of integral operators}\label{section:disc}
In view of the application in Section~\ref{section:apps}, we look at matrices stemming from discretizations of certain integral operators.
\begin{lemma}\label{lem:intop1}
Given $r\in\N\cup\{0\}$, let $G\colon \Gamma\times\Gamma \to \R$ denote an integral kernel. Let $v,w\in L^2(\Gamma)$ 
respectively have connected support with 
\begin{align*}
\dist(v,w):=\inf_{x\in {\rm supp}(v)\atop y \in {\rm supp}(w)} |x-y|>0.
\end{align*}
Let $\nu,\sigma\in\N_0$ 
such that $\partial_x^\nu\partial_y^\sigma G\in L^\infty({\rm supp}(v)\times {\rm supp}(w))$,
where $\partial_x,\partial_y$ denote the arc-length derivatives on $\Gamma$ with respect to the arguments $x$ and $y$ of $G$.
Then, there holds
\begin{align*}
\Big|&\int_\Gamma\int_\Gamma G(x,y)w(x)v(y)\,dx\,dy\Big|\\
&\leq 
C_{\rm far}\norm{v}{\widetilde H^{-\sigma}({\rm supp}(v))}\norm{w}{\widetilde H^{-\nu}({\rm supp}(v))}
\frac{\sup_{0\leq \alpha\leq \nu\atop 0\leq \beta\leq \sigma}\norm{\partial_x^\alpha \partial_y^\beta G}{L^\infty({\rm supp}(v)\times {\rm supp}(w))}}{|{\rm supp}(w)|^{-1/2}|{\rm supp}(w)|^{-1/2}}.
\end{align*}
The constant $C_{\rm far}>0$ depends only on $\Gamma$, $\sigma$, and $\nu$.
\end{lemma}
\begin{proof}
There holds 
\begin{align}\label{eq:start}
\begin{split}
\Big|\int_\Gamma\int_\Gamma G(x,y)w(x)v(y)\,dx\,dy\Big|&\leq \norm{w}{\widetilde H^{-\nu}({\rm supp}(w))}\norm{\int_\Gamma v(y)G(\cdot,y)\,dy}{H^{\nu}({\rm supp}(w))}.
\end{split}
\end{align}
We obtain that 
\begin{align*}
\norm{\int_\Gamma v(y)G(\cdot,y)\,dy}{H^{\nu}({\rm supp}(w))}^2=\sum_{\alpha=0}^\nu\norm{\int_\Gamma v(y)\partial_x^\alpha G(\cdot,y)\,dy}{L^2({\rm supp}(w))}^2.
\end{align*}
There holds
\begin{align*}
\norm{\int_\Gamma &v(y)\partial_x^\alpha G(\cdot,y)\,dy}{L^2( {\rm supp}(w))}\\
&\leq \norm{v}{\widetilde H^{-\sigma}({\rm supp}(v))}
\Big(\int_{ {\rm supp}(w)}\norm{\partial_x^\alpha G(x,\cdot)}{H^\sigma({\rm supp}(v))}^2\,dx\Big)^{1/2}\\
&\leq \norm{v}{\widetilde H^{-\sigma}({\rm supp}(v))}\sup_{x\in{\rm supp}(w)}\norm{\partial_x^\alpha G(x, \cdot)}{H^\sigma({\rm supp}(v))}|{\rm supp}(w)|^{1/2}.
\end{align*}
Together with the above, this shows
\begin{align*}
\norm{\int_\Gamma v(y)&G(\cdot,y)\,dy}{H^{\nu}({\rm supp}(w))}^2\\
&\leq \sum_{\alpha=0}^\nu\norm{v}{\widetilde H^{-\sigma}({\rm supp}(v))}\sup_{x\in{\rm supp}(w)}\norm{\partial_i^\alpha G(x, \cdot)}{H^\sigma({\rm supp}(v))}|{\rm supp}(w)|^{1/2}.
\end{align*}
For $0\leq\alpha\leq \nu$, we get
\begin{align*}
\sup_{x\in{\rm supp}(w)}\norm{\partial_x^\alpha G(x, \cdot)}{H^\sigma({\rm supp}(v))}&\lesssim
|{\rm supp}(v)|^{1/2}\sup_{x\in{\rm supp}(w)\atop 0\leq \beta\leq \sigma}\norm{\partial_x^\alpha\partial_y^\beta G(x, \cdot)}{L^\infty({\rm supp}(v))}\\
&=|{\rm supp}(v)|^{1/2}\sup_{ 0\leq \beta\leq \sigma}\norm{\partial_x^\alpha\partial_y^\beta G}{L^\infty({\rm supp}(v)\times {\rm supp}(w))}.
\end{align*}
The combination of the above estimates concludes the proof.
\end{proof}

Recall the definition of a level of a function $L(\cdot)$ from Section~\ref{section:metric} as well as the Riesz bases $B^1$, $B^{1/2}$, $B^{-1/2}$ from Theorem~\ref{thm:stableBbase}. In this situation, we have 
\begin{align*}
L(v)=\ell\quad\text{for all } v\in B^1_\ell\cup B^{1/2}_\ell\cup B^{-1/2}_\ell
\end{align*}
for all $\ell\in\N\cup\{0\}$.

\begin{lemma}\label{lem:Vbase}
Let $G(x,y):=-1/(2\pi)\log|x-y|$ denote the kernel of the single-layer potential.
Then, there holds for $w,v\in B^{-1/2}$ with $\dist(v,w)>0$ that
\begin{align}\label{eq:Vred}
\Big|\int_\Gamma\int_\Gamma G(x,y)w(x)v(y)\,dx\,dy\Big|\leq C_{\rm B} \frac{C_{\rm mesh}^{-2(L(v)+L(w))}}{{\rm dist}(v,w)^{4}}.
\end{align}
The constant $C_{\rm B}>0$ depends only on $\TT_0$.
\end{lemma}
\begin{proof}
Since, for $\sigma=\nu=2$, there holds
\begin{align*}
 \sup_{0\leq \alpha\leq \nu\atop 0\leq \beta\leq \sigma}\norm{\partial_x^\alpha \partial_y^\beta G}{L^\infty({\rm supp}(v)\times {\rm supp}(w))}
 \lesssim
 \sup_{x\in {\rm supp}(v)\atop y\in {\rm supp}(w)} |x-y|^{-4} = {\rm dist}(v,w)^4,
\end{align*}
Lemma~\ref{lem:intop1} shows 
\begin{align*}
\Big|\int_\Gamma\int_\Gamma G(x,y)w(x)v(y)\,dx\,dy\Big|\lesssim\norm{v}{\widetilde H^{-2}({\rm supp}(v))}\norm{w}{\widetilde H^{-2}({\rm supp}(w))}\frac{C_{\rm mesh}^{(-L(w)-L(v))/2}}{\dist(v,w)^{4}}.
\end{align*}
The scaling estimates~\eqref{eq:scalingm12} show $\norm{v}{\widetilde H^{-2}({\rm supp}(v))}\simeq C_{\rm mesh}^{-3L(v)/2}$ for all $v\in B^{-1/2}$ and thus conclude the proof.
\end{proof}

\begin{lemma}\label{lem:Kbase}
Let $G(x,y):=-1/(2\pi)(x-y)\cdot n(y)/|x-y|^2$ denote the kernel of the double-layer potential.
Then, there holds for $v\in B^{1/2}$, $w\in B^{-1/2}$ with $\dist(v,w)>0$ that
\begin{align}\label{eq:Kred}
\Big|\int_\Gamma\int_\Gamma G(x,y)w(x)v(y)\,dx\,dy\Big|\leq C_{\rm B} \frac{C_{\rm mesh}^{-2(L(v)+L(w))}}{{\rm dist}(v,w)^{4}}.
\end{align}
The constant $C_{\rm B}>0$ depends only on $\TT_0$.
\end{lemma}
\begin{proof}
First, we assume $v\in B_\ell^{1/2}$ for some $\ell>0$. Lemma~\ref{lem:gensz} ensures that $\int_E v\,dx=0$ for all $E\in\widehat\TT_\ell|_\Gamma$.
For all $x\in\Gamma$, we may define $G_0(x,\cdot)\in\PP^0(\widehat\TT_\ell|_\Gamma)$ by $G_0(x,\cdot)|_E:=|E|^{-1}\int_E G(x,y)\,dy$ for all $E\in\widehat\TT_\ell|_\Gamma$.
We define the new kernel $\widetilde G(x,y):= G(x,y)-G_0(x,y)$ and observe that
\begin{align*}
 \int_\Gamma\int_\Gamma G(x,y)w(x)v(y)\,dx\,dy=\int_\Gamma\int_\Gamma \widetilde G(x,y)w(x)v(y)\,dx\,dy.
\end{align*}
Note that for $\sigma=0$ and $\nu=2$, there holds
\begin{align*}
 \sup_{0\leq \alpha\leq \nu\atop 0\leq \beta\leq \sigma}&\norm{\partial_x^\alpha \partial_y^\beta \widetilde G}{L^\infty({\rm supp}(v)\times {\rm supp}(w))}\\
  &=
 \sup_{0\leq \alpha\leq 2}\sup_{E\in\widehat\TT_\ell|_\Gamma}
 \sup_{x\in {\rm supp}(v)\cap E\atop y\in {\rm supp}(w)} |\partial_x^\alpha G(x,y)-|E|^{-1}\int_E\partial_x^\alpha G(x,y)\,dy|\\
 &\lesssim C_{\rm mesh}^{-L(v)}
 \sup_{0\leq \alpha\leq 2}\sup_{E\in\widehat\TT_\ell|_\Gamma}
 \sup_{x\in {\rm supp}(v)\cap E\atop y\in {\rm supp}(w)} |\partial_y\partial_x^\alpha G(x,y)|,
\end{align*}
where we used element wise continuity of $y\mapsto G(x,y)$ and a Poincar\'e inequality in the last step.
(Note that $\partial_y\partial_x^\alpha G(x,y)$ is well-defined since $n(y)$ jumps only at corners of $\Gamma$, which are resolved by $\widehat\TT_0$.)
Repeating the steps, we also get the same estimate for $v\in B^{1/2}_0$ by choosing $G_0(x)=0$ since $L(v)=0$.

With this, and the fact that for $x,y$ not being a corner of $\Gamma$ there holds
\begin{align*}
 |\partial_y\partial_x^\alpha G(x,y)|\lesssim |x-y|^4,
\end{align*}
Lemma~\ref{lem:intop1} shows  
\begin{align*}
\Big|\int_\Gamma\int_\Gamma G(x,y)w(x)v(y)\,dx\,dy\Big|\lesssim\norm{v}{L^2({\rm supp}(v))}\norm{w}{\widetilde H^{-2}({\rm supp}(w))}\frac{C_{\rm mesh}^{-L(v)}C_{\rm mesh}^{(-L(w)-L(v))/2}}{\dist(v,w)^{4}}.
\end{align*}
The scaling estimates~\eqref{eq:scalingm12} and~\eqref{eq:scaling12} show $\norm{v}{L^2({\rm supp}(v))}\simeq C_{\rm mesh}^{-L(v)/2}$ for all $v\in B^{1/2}$ 
as well as $\norm{w}{\widetilde H^{-2}({\rm supp}(v))}\simeq C_{\rm mesh}^{-3L(w)/2}$ for all $v\in B^{-1/2}$ and thus conclude the proof.
\end{proof}
The following lemma shows that a matrix whith decay properties as above can be approximated by a banded matrix.
\begin{lemma}\label{lem:l2approx}
For $X,Y\in \{B^{-1/2},B^{1/2}\}$, let $M=(M_{vw})_{v\in X, w\in Y}$ denote an infinite matrix with
\begin{align*}
|M_{vw}|\leq  C_{\rm B} \frac{C_{\rm mesh}^{-2(L(v)+L(w))}}{{\rm dist}(v,w)^{4}}
\end{align*}
for all $v\in X$, $w\in Y$ with ${\rm dist}(v,w)>0$. For $c\in\N$ define
\begin{align*}
M^c:=(M^c_{vw})_{v\in X,w\in Y}\quad\text{\rm with}\quad M^c_{vw}:=\begin{cases}
M_{vw}&\text{\rm if }d_1(v,w)\leq c,\\
0&\text{\rm else,}
\end{cases}
\end{align*}
where $d_1(\cdot,\cdot)$ is defined in Section~\ref{section:metric}.
Then, there holds $\lim_{c\to\infty}\norm{M-M^c}{2}=0$.
\end{lemma}
\begin{proof}
Note that $\norm{\cdot}{2}$ is permutation invariant, so that the ordering of the indices in $X$ and $Y$ does not matter. Let $X_\ell$, $Y_\ell$ denote the corresponding subspaces $B_\ell^{-1/2}$ and $B_\ell^{1/2}$.
Define for $\ell,n\in\N$ and $v\in X$
\begin{align*}
P_\ell^n(v):=\set{w\in Y_\ell}{c+n-1\leq d_1(v,w)\leq c+n}.
\end{align*}
Lemma~\ref{lem:annulus} shows (note that we use the $R_1$ case of Lemma~\ref{lem:annulus})
\begin{align}\label{eq:pnum}
\#P_\ell^n(v)\lesssim C_{\rm mesh}^{\ell-\min\{L(v),\ell\}}.
\end{align}
Note that $w\in P_\ell^n(v)$ implies that at least $c+n-1$ elements of $\widehat\TT_{\min\{L(v),\ell\}}$ are needed to connect ${\rm mid}(v)$ and ${\rm mid}(w)$. 
Since both $w$ and $v$ have local support on $\widehat\TT_{\min\{L(v),\ell\}}$ (depending on $C_{\rm mesh}$), we obtain that at least $c+n-\gamma$ elements of
$\widehat\TT_{\min\{L(v),\ell\}}$ are needed to connect ${\rm supp}(v)$ and ${\rm supp}(w)$, where $\gamma\in\N$ depends only on $C_{\rm mesh}$ and $\TT_0$.
Let $S$ denote the shortest line connecting ${\rm supp}(v)$ and ${\rm supp}(w)$. Shape regularity implies that $S$ can be covered with $\mathcal{O}({\rm dist}(v,w)C_{\rm mesh}^{\min\{L(v),\ell\}})$
elements of $\widehat\TT_{\min\{L(v),\ell\}}$. This shows that ${\rm dist}(v,w)\gtrsim C_{\rm mesh}^{-\min\{L(v),\ell\}}(c+n-\gamma)$, where the hidden constant depends only on $\Omega$ and $\TT_0$.
For $c\geq \gamma$, this allows to bound the row sum of $(M-M^c)$ by
\begin{align*}
\sum_{w\in Y}|M-M^c|_{vw}&=\sum_{n=0}^\infty\sum_{\ell=0}^\infty\sum_{w\in P_\ell^{n}(v)}|M|_{vw}\\
&\lesssim \sum_{n=0}^\infty\sum_{\ell=0}^\infty\sum_{w\in P_\ell^{n}(v)}\frac{C_{\rm mesh}^{-2(L(v)+\ell)}}{{\rm dist}(v,w)^{4}}\\
&\lesssim \sum_{n=0}^\infty\sum_{\ell=0}^\infty\sum_{w\in P_\ell^n(v)}\frac{C_{\rm mesh}^{-2(L(v)+\ell)}}{(c+n-\gamma)^{4}C_{\rm mesh}^{-4\min\{L(v),\ell\}}}\\
&\lesssim \sum_{n=0}^\infty\sum_{\ell=0}^\infty C_{\rm mesh}^{\ell-\min\{L(v),\ell\}}\frac{C_{\rm mesh}^{-2(L(v)+\ell)}}{(c+n-\gamma)^{4}C_{\rm mesh}^{-4\min\{L(v),\ell\}}}\\
&\lesssim \sum_{n=0}^\infty\sum_{\ell=0}^\infty(c+n-\gamma)^{-{4}}C_{\rm mesh}^{3\min\{L(v),\ell\}}C_{\rm mesh}^{-2L(v)-\ell}\\
&\leq \sum_{n=0}^\infty\sum_{\ell=0}^\infty(c+n-\gamma)^{-{4}}C_{\rm mesh}^{-|L(v)-\ell|}\lesssim 
\sum_{n=0}^\infty(c+n-\gamma)^{-{4}}.
\end{align*}
The analogous result holds for the column sums. Thus, we may estimate
\begin{align*}
\norm{M-M^c}{2}^2\leq \norm{M-M^c}{1}\norm{M-M^c}{\infty}\lesssim \Big(\sum_{n=0}^\infty(c+n-\gamma)^{-{4}}\Big)^2.
\end{align*}
The sum on the right-hand side is finite and tends to zero for $c\to\infty$. This concludes the proof.
\end{proof}

\subsection{Almost bandedness of differential/integral operator matrices}\label{section:banded}

This section uses the established techniques to prove that the FEM/BEM-coupling matrices are close to banded matrices in the sense of Definition~\ref{def:banded}.
To that end, recall $d_1(\cdot,\cdot)$ from Section~\ref{section:metric}.
\begin{lemma}\label{lem:Mscaling}
Let $M_{ij}:=\dual{v_i}{w_j}_\Gamma$ for all $i,j\in\N$ with $v_i\in B^{-1/2}$ and $w_j\in B^{1/2}$.
Given $\eps>0$, there exists $M^\eps\in\R^{\N\times \N}$ and a constant $C_M>0$ such that
\begin{align*}
 \norm{M^\eps-M}{2}\leq \eps.
\end{align*}
as well as
\begin{align}\label{eq:Mbanded}
 \Big(|L(v_i)-L(w_j)|> C_M \text{ or } d_1(v_i,w_j)> 1\Big)\quad\implies\quad M^\eps_{ij}=0.
\end{align}
\end{lemma}
\begin{proof}
Define $I:=\set{(i,j)\in\N^2}{{\rm supp}(v_i)\cap {\rm supp}(v_i)\neq \emptyset}$.
Let $i,j\in I$ and set $\ell:=L(v_i)$, $k:=L(w_j)$.
First, consider the case $k\leq \ell$. There holds
\begin{align*}
|\dual{v_i}{w_j}_\Gamma|\leq \norm{v_i}{\widetilde H^{-1}({\rm supp}(v_i))}\norm{w_j}{H^1({\rm supp}(v_i))}.
\end{align*}
There holds with an inverse estimate,
scaling estimates, and~\eqref{eq:scaling12}
\begin{align*}
\norm{w_j}{H^1({\rm supp}(v_i))}&\lesssim C_{\rm mesh}^{-|k-\ell|/2}\norm{w_j}{H^1({\rm supp}(w_j))}
\lesssim  C_{\rm mesh}^{k/2-|k-\ell|/2}.
\end{align*}
The estimate~\eqref{eq:scalingm12} shows
\begin{align*}
\norm{v_i}{\widetilde H^{-1}({\rm supp}(v))}\simeq C_{\rm mesh}^{-\ell/2}.
\end{align*}
The combination of the above estimates shows $|M_{ij}|\lesssim C_{\rm mesh}^{-|k-\ell|}$.


Second, for $k>\ell$, we have
\begin{align*}
|\dual{v_i}{w_j}_\Gamma|\leq \norm{v_i}{L^2({\rm supp}(w_j))}\norm{w_j}{L^2({\rm supp}(w_j))}.
\end{align*}
As above, a scaling estimate and~\eqref{eq:scalingm12} shows
\begin{align*}
\norm{v_i}{L^2({\rm supp}(w_j))}&\lesssim C_{\rm mesh}^{-|k-\ell|/2}\norm{v_i}{L^2({\rm supp}(v_i))}
\lesssim  C_{\rm mesh}^{k/2-|k-\ell|/2}.
\end{align*}
Together with~\eqref{eq:scaling12}, this proves $|M_{ij}|\lesssim C_{\rm mesh}^{-|k-\ell|}$. Moreover, for $(i,j)\notin I$, we have $M_{ij}=0$ by definition.

Define $\widetilde I:=\set{(i,j)\in \N^2}{|L(v_i)-L(w_j)|\leq r}$ for some $r\in\N$.
Note that for $v_i\in B_\ell^{-1/2}$, there holds $\#\set{w_j\in B_k^{1/2}}{M_{ij}\neq 0}\lesssim C_{\rm mesh}^{k-\min\{k,\ell\}}$ and 
$\#\set{v_i\in B_\ell^{-1/2}}{M_{ij}\neq 0}\lesssim C_{\rm mesh}^{\ell-\min\{k,\ell\}}$. Thus, 
Lemma~\ref{lem:geodecay} shows with $q=C_{\rm mesh}^{-1}$ that $\norm{M|_{\N^2\setminus \widetilde I}}{2}\lesssim C_{\rm mesh}^{-r/2}$.

We define  $M^\eps_{ij}=M_{ij}$ for all $(i,j)\in \widetilde I$ and $M^\eps_{ij}=0$ else. The above estimates 
show $\norm{M-M^\eps}{2}\leq \eps$ for sufficiently large $r$. The bandedness~\eqref{eq:Mbanded} follows from the definition of the $\widetilde I$.
 This concludes the proof.
\end{proof}

\begin{lemma}\label{lem:Kscaling}
Let $M_{ij}:=\dual{K w_i}{v_j}_\Gamma$ for all $i,j\in\N$ with $w_i\in B^{1/2}$ and $v_j\in B^{-1/2}$.
Given $\eps>0$, there exists $M^\eps\in\R^{\N\times \N}$ and a constant $C_M>0$ such that
\begin{align*}
 \norm{M^\eps-M}{2}\leq \eps.
\end{align*}
as well as
\begin{align}\label{eq:Kbanded}
 \Big(|L(v_i)-L(w_j)|> C_M \text{ or } d_1(v_i,w_j)> C_M\Big)\quad\implies\quad M^\eps_{ij}=0.
\end{align}
\end{lemma}
\begin{proof}
Since $\Gamma$ is polygonal,~\cite[Theorem~6.34 and subsequent remark]{steinbach} ensures that there exists $\delta>0$ (depending only on the smallest interior angle of $\Gamma$) 
such that $K\colon H^{1+\delta}(\Gamma)\to H^{1+\delta}(\Gamma)$ and $K^\prime\colon H^{\delta}(\Gamma)\to H^{\delta}(\Gamma)$ are continuous.
	This shows for $k:=L(v_j)\leq \ell:=L(w_i)$
	\begin{align*}
		|\dual{Kw_i}{v_j}_\Gamma|&\leq \norm{w_i}{\widetilde H^{-\delta}({\rm supp}(w_i))}\norm{K^\prime v_j}{H^{\delta}({\rm supp}(w_i))}\lesssim 
		\norm{w_i}{\widetilde H^{-\delta}({\rm supp}(w_i))}\norm{K^\prime v_j}{H^{\delta}(\Gamma)}\\
		&\lesssim 
		\norm{w_i}{\widetilde H^{-\delta}({\rm supp}(w_i))}\norm{v_j}{H^{\delta}(\Gamma)}.	
	\end{align*}
	An inverse estimate for $v_j$ shows $\norm{v_j}{H^{\delta}(\Gamma)}\lesssim C_{\rm mesh}^{\delta k}\norm{v_j}{L^2(\Gamma)}=
	C_{\rm mesh}^{\delta k}\norm{v_j}{L^2({\rm supp}(v_j))}$ and hence~\eqref{eq:scaling12} together with~\eqref{eq:scalingm12} conclude
	\begin{align}\label{eq:Kscaling}
	 |\dual{Kw_i}{v_j}_\Gamma|\lesssim C_{\rm mesh}^{(1/2+\delta)|L(v_j)-L(w_i)|}.
	\end{align}
	For $k\geq \ell$, we have
	\begin{align*}
		|\dual{Kw_i}{v_j}_\Gamma|&\leq \norm{K w_i}{\widetilde H^{1+\delta}({\rm supp}(v_j))}\norm{v_j}{H^{-1-\delta}({\rm supp}(v_j))}\lesssim 
		\norm{K w_i}{\widetilde H^{1+\delta}(\Gamma)}\norm{v_j}{H^{-1-\delta}({\rm supp}(v_j))}\\
		&\lesssim 
		\norm{ w_i}{\widetilde H^{1+\delta}(\Gamma)}\norm{v_j}{H^{-1-\delta}({\rm supp}(v_j))}.	
	\end{align*}
	An inverse estimate for $w_i$ shows $\norm{w_i}{H^{1+\delta}(\Gamma)}\lesssim C_{\rm mesh}^{(1+\delta) \ell}\norm{w_i}{L^2(\Gamma)}=
	C_{\rm mesh}^{(1+\delta) \ell}\norm{w_i}{L^2({\rm supp}(w_i))}$ and hence~\eqref{eq:scaling12} together with~\eqref{eq:scalingm12} conclude~\eqref{eq:Kscaling} for all $i,j\in\N$.
	
	We first restrict the index set $\N\times\N$ by
 \begin{align*}
  \widetilde I:=\set{(i,j)\in\N^2}{d_1(w_i,v_j)\leq c },
 \end{align*}
 for some constant $c\geq 1$.
 Furthermore, define for some $r\in\N$
 \begin{align*}
  I:=\set{(i,j)\in \widetilde I}{|L(w_i)-L(v_j)|\leq r}.
 \end{align*}
	Define $\widetilde{M}^\eps_{ij}:=M_{ij}$ for all $(i,j)\in \widetilde I$ and zero elsewhere, and define $M^\eps_{ij}:=M_{ij}$ for all $(i,j)\in I$ and zero elsewhere.
	Then, Lemma~\ref{lem:Kbase} together with Lemma~\ref{lem:l2approx} show $\norm{M-\widetilde{M}^\eps} {2}\leq \eps/2$ for sufficiently large $c$.
	Moreover, Lemma~\ref{lem:geodecay} shows $\norm{\widetilde{M}^\eps-M^\eps}{2}\lesssim C_{\rm mesh}^{-\delta r}$ and hence $\norm{M-M^\eps}{2}\leq \eps$ for sufficiently large $r$.
The definition of $\widetilde I$ and $I$ implies that $M^\eps$ satisfies~\eqref{eq:Kbanded}.
\end{proof}

\begin{lemma}\label{lem:Vscaling}
 Let $M_{ij}:=\dual{V v_i}{w_j}_\Gamma$ for all $i,j\in\N$ with $v_i,w_j\in B^{-1/2}$.
Given $\eps>0$, there exists $M^\eps\in\R^{\N\times \N}$ and a constant $C_M>0$ such that
\begin{align*}
 \norm{M^\eps-M}{2}\leq \eps.
\end{align*}
as well as
\begin{align}\label{eq:Vbanded}
 \Big(|L(v_i)-L(w_j)|> C_M \text{ or } d_1(v_i,w_j)> C_M\Big)\quad\implies\quad M^\eps_{ij}=0.
\end{align}
\end{lemma}
\begin{proof}
Assume $k:=L(v_i)\leq \ell:=L(w_j)$. 
Since $\Gamma$ is polygonal, there exists $\delta>0$ such that $V\colon H^{\delta}(\Gamma)\to H^{1+\delta}(\Gamma)$ is continuous.
This shows
\begin{align*}
|\dual{Vv_i}{w_j}_\Gamma|&\leq \norm{v_i}{\widetilde H^{-1-\delta}({\rm supp}(v_i))}\norm{Vw_j}{H^{1+\delta}({\rm supp}(v_i))}\lesssim 
 \norm{v_i}{\widetilde H^{-1-\delta}({\rm supp}(v_i))}\norm{Vw_j}{H^{1+\delta}(\Gamma)}\\
 &\lesssim 
 \norm{v_i}{\widetilde H^{-1-\delta}({\rm supp}(v_i))}\norm{w_j}{H^{\delta}(\Gamma)}.	
\end{align*}
An inverse estimate for $w_j$ shows $\norm{w_j}{H^{\delta}(\Gamma)}\lesssim C_{\rm mesh}^{\delta \ell}\norm{w_j}{L^2(\Gamma)}=
C_{\rm mesh}^{\delta\ell}\norm{w_j}{L^2({\rm supp}(w_j))}$ and hence~\eqref{eq:scalingm12} concludes
\begin{align*}
 |\dual{Vv_i}{w_j}_\Gamma|\lesssim C_{\rm mesh}^{(1/2+\delta)|L(v_j)-L(w_i)|}.
\end{align*}
Symmetry of the problem shows the above result also for $k\geq \ell$ and hence for all $i,j\in\N$. With the index sets $I$ and $\widetilde I$ from the proof of Lemma~\ref{lem:Kscaling}, we conclude the proof
by use of Lemma~\ref{lem:Vbase}, Lemma~\ref{lem:l2approx}, and Lemma~\ref{lem:geodecay} as in the proof of Lemma~\ref{lem:Kscaling}.
\end{proof}

\begin{lemma}\label{lem:Ascaling}
Let $M_{ij}:=\dual{\nabla v_i}{\nabla v_j}_\Omega$ for all $i,j\in\N$ with $v_i,v_j\in B^{1}$.
Given $\eps>0$, there exists $M^\eps\in\R^{\N\times \N}$ and a constant $C_M>0$ such that
\begin{align*}
 \norm{M^\eps-M}{2}\leq \eps.
\end{align*}
as well as
\begin{align}\label{eq:Abanded}
 \Big(|L(v_i)-L(v_j)|> C_M \text{ or } d_1(v_i,v_j)> C_M\Big)\quad\implies\quad M^\eps_{ij}=0.
\end{align}
\end{lemma}
\begin{proof}
Assume $k:=L(v_i)\leq \ell:=L(v_j)$. Define $\omega_T:={\rm supp}(v_j)\cap T$ and compute for some $\delta>0$
\begin{align*}
|\dual{\nabla v_i}{\nabla v_j}_\Omega| &\leq\sum_{T\in\widehat\TT_k} |\dual{\nabla v_i}{\nabla v_j}_{\omega_T}|\\
&\leq \sum_{T\in\widehat\TT_k} |\dual{\Delta v_i}{v_j}_{\omega_T}|+ |\dual{\partial_n v_i}{v_j}_{\partial\omega_T}|\\
&\lesssim \sum_{T\in\widehat\TT_k}\norm{v_i}{H^2(\omega_T)}\norm{v_j}{L^2(\omega_T)}+\norm{v_i}{H^{3/2}(\omega_T)}\norm{v_j}{H^{1/2+\delta}(\omega_T)}
\end{align*}
From this, a scaling argument together with an inverse estimate concludes for $s\in\{1/2,1\}$
\begin{align*}
  \norm{v_i}{H^{1+s}(\omega_T)}
 &\lesssim C_{\rm mesh}^{-|k-\ell|/2}\norm{v_i}{H^{1+s}(T)}\leq C_{\rm mesh}^{-|k-\ell|/2+s k}\norm{v_i}{H^1({\rm supp}(v_i))}.
\end{align*}
Altogether,~\eqref{eq:scaling1} shows
\begin{align}\label{eq:Ascaling}
 |\dual{\nabla v_i}{\nabla v_j}_\Omega| \lesssim C_{\rm mesh}^{-(3/2-\delta)|L(v_i)-L(v_j)|}.
\end{align}
Symmetry of the problem shows the above also for $\ell\leq k$ and hence for all $i,j\in\N$.
We restrict the index set by
\begin{align*}
  I:=\set{(i,j)\in \N^2}{|L(v_i)-L(v_j)|\leq r}
 \end{align*}
 and define $M^\eps_{ij}:=M_{ij}$ for all $(i,j)\in I$ and zero elsewhere. 
 Note that $\#\set{v_j\in B_k^1}{M_{ij}\neq 0}\lesssim C_{\rm mesh}^{2(k-\min\{\ell,k\})}$ and $\#\set{v_i\in B^1_\ell}{M_{ij}\neq 0}\lesssim C_{\rm mesh}^{2(\ell-\min\{k,\ell\})}$.
 Estimate~\eqref{eq:Ascaling} and Lemma~\ref{lem:geodecay} with $q=C_{\rm mesh}^{-2}$ show $\norm{M-M^\eps}{2}\lesssim C_{\rm mesh}^{-\eps r}$.
 The implication~\eqref{eq:Abanded} follows as in the proof of Lemma~\ref{lem:Mscaling}.
 Thus, we conclude the proof by choosing $r\in\N$ sufficiently large.
\end{proof}

\section{Auxiliary Results}
\begin{lemma}\label{lem:log}
 For $a,b\geq 0$, there holds
 \begin{align}\label{eq:log}
  \log(1+a+b)\leq \log(1+a)+\log(1+b),
 \end{align}
 whereas for $b\geq 1$, there holds
 \begin{align}\label{eq:log2}
  \log(1+ab)\leq \log(1+a) +\log(b).
 \end{align}
 \end{lemma}
\begin{proof}
Exponentiation trivializes the inequalities.
\end{proof}
\begin{lemma}\label{lem:derivative}
Given a connected $\omega\subseteq \Gamma$, the arc-length derivative satisfies
\begin{align*}
\norm{\partial_\Gamma v}{\widetilde H^{-1}(\omega)}\simeq\norm{ v}{L^2(\omega)}\quad\text{for all }v\in H^1_0(\omega),
\end{align*}
where the hidden constant depends only on the arc-length of $\Gamma$.
Moreover, the arc-length derivative is an isomorphism $\partial_\Gamma\colon  H^s_\star(\Gamma)\to \widetilde H^{s-1}_\star(\Gamma)$ for all $0\leq s\leq 1$, where
$\YY_\star:= \set{v\in \YY}{\dual{v}{1}_\Gamma =0}$ for $\YY\in\{H^s(\Gamma),\widetilde H^{s-1}(\Gamma)\}$.
\end{lemma}
\begin{proof}
Let $V\in H^1(\omega)$ be the anti-derivative such that $\partial_\Gamma V=v$ in $\omega$ and $\int_\omega V\,dx =0$. Then, there holds
\begin{align*}
\norm{\partial_\Gamma v}{\widetilde H^{-1}(\omega)}= \sup_{w\in H^1(\omega)}\frac{-\dual{v}{\partial_\Gamma w}_\omega}{\norm{w}{H^1(\omega)}}
\geq \frac{-\dual{v}{\partial_\Gamma V}_\omega}{\norm{V}{H^1(\omega)}}
\gtrsim \frac{-\dual{v}{\partial_\Gamma V}_\omega}{\norm{\partial_\Gamma V}{L^2(\omega)}}
=\norm{v}{L^2(\omega)},
\end{align*}
where we used a Poincar\'e inequality in the penultimate step.
On the other hand, we have
\begin{align*}
\norm{\partial_\Gamma v}{\widetilde H^{-1}(\omega)} = \sup_{w\in H^1(\omega)}\frac{-\dual{v}{\partial_\Gamma w}_\omega}{\norm{w}{H^1(\omega)}}\leq \norm{v}{L^2(\omega)}.
\end{align*}
This concludes the proof of the first statement. The second statement follows analogously for $s=0$. The case $s=1$ is obvious. Interpolation between those cases concludes the proof.
\end{proof}

\begin{lemma}\label{lem:geodecay}
	Let $M\in\R^{\N\times\N}$ and let $\N=\bigcup_{\ell\in\N}B_\ell$ such that $|M_{ij}|\leq Cq^{(1/2+\eps)|\ell-k|}$ for all $i\in B_\ell$, $j\in B_k$,
	where for all $i\in B_\ell$ there holds $\#\set{j\in B_k}{M_{ij}\neq 0}\leq C q^{\min\{k,\ell\}-k}$ as well as  for all $j\in B_k$ there holds $\#\set{i\in B_\ell}{M_{ij}\neq 0}\leq C q^{\min\{k,\ell\}-\ell}$
	for some $0<q<1$ and some $C>0$. 
	Let additionally $M_{ij}=0$ if $i\in B_\ell$, $j\in B_k$ and $|\ell-k|> r$ for some $r\in\N$. Then, there holds
	\begin{align*}
	 \norm{M}{2}\leq C_{\rm geo} q^{\eps r},
	\end{align*}
	where $C_{\rm geo}>0$ depends only on $C$ and $q$.
\end{lemma}
\begin{proof}
	Let $k,\ell\in\N$. There holds
	\begin{align*}
	 \norm{M|_{B_\ell\times B_k}}{\infty}=\sup_{i\in B_\ell} \sum_{j\in B_k}|M_{ij}|\lesssim q^{(1/2+\eps)|\ell-k|+\min\{k,\ell\}-k}
	\end{align*}
	as well as
	\begin{align*}
	 \norm{M|_{B_\ell\times B_k}}{1}=\sup_{j\in B_k} \sum_{i\in B_\ell}|M_{ij}|\lesssim q^{(1/2+\eps)|\ell-k|+\min\{k,\ell\}-\ell}.
	\end{align*}
	The standard interpolation estimate shows
	 \begin{align*}
	 \norm{M|_{B_\ell\times B_k}}{2}^2\leq \norm{M|_{B_\ell\times B_k}}{1}\norm{M|_{B_\ell\times B_k}}{\infty}\lesssim q^{(1+2\eps)|\ell-k|+2\min\{k,\ell\}-\ell-k}=q^{2\eps|\ell-k|}.
	 \end{align*}
	 With this, we prove
	 \begin{align*}
	  \norm{Mx}{\ell_2}^2&=\sum_{\ell\in\N}\norm{(Mx)|_{B_\ell}}{\ell_2}^2 = \sum_{\ell\in\N}\norm{\sum_{k\in\N\atop |\ell-k|>r}M|_{B_\ell\times B_k}x|_{B_k}}{\ell_2}^2\\
	  &=\sum_{\ell\in\N}\sum_{k\in\N\atop |\ell-k|>r}\sum_{n\in\N\atop |\ell-n|>r}(M|_{B_\ell\times B_k}x|_{B_k})\cdot (M|_{B_\ell\times B_n}x|_{B_n})\\
	  &\lesssim \sum_{\ell\in\N}\sum_{k\in\N\atop |\ell-k|>r}\sum_{n\in\N\atop |\ell-n|>r}q^{\eps(|\ell-k|+|\ell-n|)} \norm{x|_{B_k}}{\ell_2}\norm{x|_{B_n}}{\ell_2}.
	 \end{align*}
	 With $2\norm{x|_{B_k}}{\ell_2}\norm{x|_{B_n}}{\ell_2}\leq \norm{x|_{B_k}}{\ell_2}^2 +\norm{x|_{B_n}}{\ell_2}^2$, we obtain
	 \begin{align*}
	  \norm{Mx}{\ell_2}^2&\lesssim \sum_{\ell\in\N}\sum_{k\in\N\atop |\ell-k|>r}\sum_{n\in\N\atop |\ell-n|>r}q^{\eps(|\ell-k|+|\ell-n|)}\big(\norm{x|_{B_k}}{\ell_2}^2 +\norm{x|_{B_n}}{\ell_2}^2\big)\\
	  &\leq 2\sum_{\ell\in\N}\sum_{k\in\N\atop |\ell-k|>r}\sum_{n\in\N\atop |\ell-n|>r}q^{\eps(|\ell-k|+|\ell-n|)}\norm{x|_{B_k}}{\ell_2}^2\\
	  &=2\sum_{k\in\N}\norm{x|_{B_k}}{\ell_2}^2\sum_{\ell\in\N\atop |\ell-k|>r}q^{\eps|\ell-k|}\sum_{n\in\N\atop |\ell-n|>r}q^{\eps|\ell-n|}\lesssim q^{\eps r}\norm{x}{\ell_2}^2.
	 \end{align*}
This concludes the proof.
\end{proof}
\begin{lemma}\label{lem:blockstab}
 Let $M\in\R^{\N\times\N}$ and let there exist a block-structure $n_1,n_2,\ldots\in\N$ such that $n_1=1$ and $n_i<n_j$ for all $i<j$.
 Let $M$ be block-banded in the sense that there exists $b\in\N$ such that $M(i,j)=0$ for all $|i-j|>b$. 
 Then, there holds
 \begin{align*}
  \norm{M}{2}\leq (2b+1)^2\sup_{i,j\in\N}\norm{M(i,j)}{2}.
 \end{align*}
\end{lemma}
\begin{proof}
 We obtain for $x\in\ell_2$ that
 \begin{align*}
  \norm{Mx}{\ell_2}^2 &= \sum_{i=1}^\infty \norm{\sum_{j=i-b}^{i+b} M(i,j) x|_{\{n_j,\ldots,n_{j+1}-1\}}}{\ell_2}^2\\
  &\leq (2b+1)
  \sum_{i=1}^\infty \sum_{j=i-b}^{i+b}\norm{ M(i,j) x|_{\{n_j,\ldots,n_{j+1}-1\}}}{\ell_2}^2\\
  &\leq  (2b+1)\Big(\sup_{i,j\in\N}\norm{M(i,j)}{2}\Big)\sum_{i=1}^\infty \sum_{j=i-b}^{i+b}\norm{ x|_{\{n_j,\ldots,n_{j+1}-1\}}}{\ell_2}^2\\
  &= (2b+1)\Big(\sup_{i,j\in\N}\norm{M(i,j)}{2}\Big)\sum_{j=1}^\infty \norm{ x|_{\{n_j,\ldots,n_{j+1}-1\}}}{\ell_2}^2\sum_{i=j-b}^{j+b}\\
  &\leq (2b+1)^2 \Big(\sup_{i,j\in\N}\norm{M(i,j)}{2}\Big)\norm{x}{\ell_2}.
 \end{align*}
This concludes the proof.
\end{proof}

\bibliographystyle{plain}
\bibliography{literature}
\end{document}